\documentclass[a4paper]{amsart}

\usepackage{amsmath,amssymb}
\usepackage{amsthm}
\usepackage{amsfonts}
\usepackage{hyperref}
\usepackage{enumerate}
\usepackage[all]{hypcap}[2006/02/20]
\usepackage{color}

\newcommand{\bb}[1]{\mathbb{#1}}

\newcommand{\pard}[2]{\frac{\partial #1}{\partial #2}}

\newcommand{\ho}{\left(\frac{d}{dt} -\Delta \right)}
\newcommand{\ddt}[1]{\frac{ d #1}{dt}}
\newcommand{\hos}{\left(\frac{d}{ds} -\ti\Delta \right)}

\newcommand{\ip}[2]{\left \langle #1 , #2 \right\rangle}

\newcommand{\n}{\nabla}
\newcommand{\e}{\epsilon}

\renewcommand{\l}{\lambda}
\renewcommand{\t}{\theta}

\renewcommand{\d}{\delta}

\renewcommand{\a}{\alpha}
\newcommand{\Si}{\Sigma}
\newcommand{\ra}{\rightarrow}
\newcommand{\ov}{\overline}

\renewcommand{\ddt}[1]{\frac{d #1}{dt}}
\newcommand{\II}{{I\!I}}

\newcommand{\nt}{\ov\n_\ddt{}}
\newcommand{\np}{\n^\perp}
\newcommand{\ti}[1]{\tilde{#1}}

\newcommand{\sensible}{tame}

\newcommand{\AS}{{I\!I}^\Sigma}
\newcommand{\U}{\hat{u}}
\DeclareMathOperator\GG{G}
\DeclareMathOperator\GL{GL}

\numberwithin{equation}{section}
\begin{document}
\theoremstyle{plain}
\newtheorem{theorem}{Theorem}
\newtheorem{lemma}[theorem]{Lemma}
\newtheorem{claim}[theorem]{Claim}
\newtheorem{proposition}[theorem]{Proposition}
\newtheorem{cor}[theorem]{Corollary}
\theoremstyle{definition}
\newtheorem{Examples}[theorem]{Examples}
\newtheorem{Example}[theorem]{Example}
\newtheorem{defses}[theorem]{Definition}
\newtheorem{assumption}[theorem]{Assumption}
\newtheorem*{ack}{Acknowledgement}
\theoremstyle{remark}
\newtheorem{remark}[theorem]{Remark}
\numberwithin{theorem}{section}
\setcounter{tocdepth}{1}

\title[Spacelike MCF]{Spacelike mean curvature flow
}
\author{Ben Lambert}
\address{Department of Mathematics, University College London, Gower Street, London, WC1E 6BT, United Kingdom}

\email{b.lambert@ucl.ac.uk}
\author{Jason D. Lotay}
\address{Mathematical Institute, University of Oxford, Woodstock Road,	Oxford,	OX2 6GG, United Kingdom}
\email{jason.lotay@maths.ox.ac.uk}

\date{\today}


\begin{abstract}
We prove long-time existence and convergence results for spacelike solutions to mean curvature flow in the pseudo-Euclidean space $\bb{R}^{n,m}$, which are entire or defined on bounded domains and satisfying Neumann or Dirichlet boundary conditions.  As an application, we prove long-time existence and convergence of the $\GG_2$-Laplacian flow in cases related to coassociative fibrations.
\end{abstract}
\maketitle
\tableofcontents
\section{Introduction}\label{s:intro}

Whilst mean curvature flow (MCF) in Euclidean space, particularly in the case of hypersurfaces, has been much studied with many celebrated results, and continues to be a very active area of research, 
the corresponding MCF in pseudo-Euclidean space $\bb{R}^{n,m}$ has received relatively little attention.  A simple but important observation is that the condition for a $n$-dimensional submanifold $M$ of $\bb{R}^{n,m}$ to be \emph{spacelike}, in the sense that the ambient quadratic form of signature $(n,m)$ restricts to be a Riemannian metric on $M$, is preserved by MCF, naturally leading to the notion of spacelike mean curvature flow, whose critical points are called \emph{maximal} submanifolds.  Surprisingly, as we shall demonstrate in this article, spacelike MCF is very well-behaved in $\bb{R}^{n,m}$ for any $m\geq 1$ (i.e.~regardless of the codimension of the flowing spacelike submanifold). This is in marked contrast to the usual mean curvature flow of $n$-dimensional submanifolds in $\bb{R}^{n+m}$, where the difference between the setting of hypersurfaces (i.e.~$m=1$) and higher codimension submanifolds is significant.  We show that spacelike MCF for entire graphs in any codimension always has smooth long-time existence under weak initial assumptions. We also show the same is true for spacelike MCF on bounded domains satisfying the natural Neumann and Dirichlet boundary conditions, where we also get convergence to a maximal submanifold.   
These results for the boundary value problems are particularly striking in the context of the Dirichlet problem, since it is known that for higher codimension MCF with Dirichlet boundary conditions in Euclidean space one cannot always 
have convergence by results in \cite{LawsonOsserman}.  Moreover, our result in the Dirichlet case can be seen as an extension of the very recent work in \cite{YangLi} on the Dirichlet problem for maximal submanifolds in $\bb{R}^{n,m}$.

There is a direct, yet surprising, link between spacelike mean curvature flow  and Bryant's \cite{BryantRemarks} $\GG_2$-Laplacian flow in 7 dimensions, whose critical points define metrics with exceptional holonomy $\GG_2$ (and are thus Ricci-flat).  Finding holonomy $\GG_2$ metrics is a challenging problem, and the $\GG_2$-Laplacian flow is a potentially powerful and attractive means for tackling it.  For a simply connected domain $B$ in $\bb{R}^3$, spacelike MCF of $B$ in $\bb{R}^{3,3}$ is equivalent to the $\GG_2$-Laplacian flow on $Z^7=B\times T^4$, where the evolving closed $\GG_2$-structure $\varphi$ is $T^4$-invariant and $Z$ is a (trivial) coassociative $T^4$-fibration over $B$.  Here, coassociative means   the submanifold is calibrated by $*\varphi$, and the aforementioned correspondence is an extension of a result in \cite{Baraglia}.   Moreover, it follows from work in \cite{DonaldsonAdiabatic} 
that spacelike MCF in $\bb{R}^{3,19}$ is the \emph{adiabatic limit} of the $\GG_2$-Laplacian flow on $Z^7$ which is a coassociative K3 fibration: i.e., spacelike MCF appears in the limit as the $\GG_2$-Laplacian flow in this setting as one sends the volume of the coassociative K3 fibres to zero.  Coassociative fibrations are expected to play a key role in $\GG_2$ geometry, motivated by ideas both from mathematics (e.g.~\cite{Baraglia,DonaldsonAdiabatic}) and M-Theory in theoretical physics (e.g.~\cite{AcharyaWitten,GukovYauZaslow}).

Despite recent  progress in the study of the $\GG_2$-Laplacian flow, it seems difficult in general to obtain long-time existence.  By utilizing the link to spacelike MCF, we obtain long-time existence and convergence results for the $\GG_2$-Laplacian flow in settings pertinent to the study of the important topic of coassociative fibrations.  These are   the first such general results for the $\GG_2$-Laplacian flow without assumptions about closeness to a critical point or curvature bounds along the flow.

\subsection{Main results}  
 Let $\Omega$ be a domain in $\bb{R}^n$ (which we will often identify with the standard spacelike $\bb{R}^n$ in $\bb{R}^{n,m}$) and let $\hat{X}_0:\Omega\ra\bb{R}^{n,m}$ be an initial smooth spacelike immersion.  We   consider   unparameterised mean curvature flow starting at $\hat{X}_0$: a one-parameter family of immersions,  given by $\hat{X}:\Omega\times [0,T) \ra \bb{R}^{n,m}$ with
\begin{equation}\
\begin{cases}\left(\displaystyle\ddt{\hat{X}}\right)^\perp  = H & \text{on } \Omega
\times [0,T) ,\\[12pt]
 \hat X(\cdot,0)=\hat{X}_0(\cdot) & \text{on } \Omega ,
\end{cases}
 \label{UnparametrisedMCF}
\end{equation}
where $H$ is the mean curvature of $M_t$, the image of $\hat{X}$ at time $t$, in $\bb{R}^{n,m}$.
Locally (in space and time)  there exists a parametrisation $X$ of $M_t$ which satisfies the standard mean curvature flow equation:
\begin{equation}\label{MCF}
\ddt{X} = H .
\end{equation}

Since any spacelike submanifold in $\bb{R}^{n,m}$ is a graph over a domain in the standard spacelike $\bb{R}^n$,  we may consider \eqref{UnparametrisedMCF} as (locally) equivalent to a parabolic system  for graph function $\U=(\U^1,\ldots,\U^m):\Omega\times [0,T)\to \bb{R}^m$ with initial condition $\U_0$ (see Appendix \ref{MCFasGraphs} for details):   
\begin{equation}\label{EntireMCF}
 \begin{cases} 
  \displaystyle\ddt\U 
   -g^{ij}(D\U)D^2_{ij} \U
   =0 & \text{on } \Omega \times[0,T) ,\\[8pt]
  \U
  (\cdot,0) =\U
  _0(\cdot) & \text{on } \Omega ,
 \end{cases}
\end{equation}
for $i,j\in\{1,\ldots,n\}$, where $g^{ij}$ is the inverse of the induced metric.

Spacelike mean curvature flow has been studied in codimension 1 by Ecker and Huisken \cite{EckerHuiskenSpacelike}, Ecker \cite{EckerEntire}\cite{EckerNull} and also Gerhardt \cite{Gerhardt}. The first author has also worked on boundary conditions for this flow \cite{Blatentselfreference}\cite{LambertMinkowski}. The elliptic counterpart was studied by Bartnik \cite{Bartnik} and Bartnik and Simon \cite{BartnikSimon}. For higher codimensions, less is known. The flow of compact manifolds was investigated by G.~Li and Salavessa \cite{LiSalavessa}. The higher codimensional maximal surface equation was recently studied by Y.~Li \cite{YangLi}. 

\subsubsection*{Entire graphs} There are several well-known explicit long-time solutions to spacelike MCF.   Throughout the article we let $\langle\cdot,\cdot\rangle$ denote the standard quadratic form with signature $(n,m)$ on $\bb{R}^{n,m}$ and let $|x|^2=\langle x,x\rangle$ for $x\in\bb{R}^{n,m}$.   Recall that $x\neq 0$ is \emph{spacelike} if $|x|^2>0$, \emph{lightlike} or \emph{null} if $|x|^2=0$ and \emph{timelike} if $|x|^2<0$.   The light cone is the set of lightlike vectors.

\begin{Example}[Grim Reaper]\label{ex:GrimReaper}
The \emph{Grim Reaper} is the unique translating solution to  \eqref{EntireMCF} in $\bb{R}^{1,1}$ (up to translations, dilations and rotations), given by
  \[\hat{u}(x,t) = \log \cosh x +t\ .\]  
\end{Example}

\begin{Example}[Hyperbolic space]\label{ex:hyperbolic}
In $\bb{R}^{n,1}$,  
  \[M_t:=\{x\in \bb{R}^{n,1}| |x|^2=-2nt\}\]
  is a self-expander for \eqref{MCF} (i.e.~$X^\perp=tH$) coming out of the light cone. For each $t$, $M_t$ is an embedded hyperbolic space in $\bb{R}^{n,1}$.
\end{Example}

\noindent Explicit solutions may be constructed from Examples \ref{ex:GrimReaper} and \ref{ex:hyperbolic} in higher codimension, simply by evolving in $\bb{R}^{n,1}\subset\bb{R}^{n,m}$.

All the examples described thus far are entire graphs, and so it is natural to study this setting, where we have the following existence theorem.
\begin{theorem}\label{thm:entire} Let $\Omega=\bb{R}^n$, so the initial spacelike submanifold $M_0$ is an entire graph.  There exists a spacelike solution $M_t$ of mean curvature flow starting at $M_0$ which is smooth and exists for all $t>0$.  
Furthermore, if the mean curvature of $M_0$ is bounded, then $M_t$ attains the initial data $M_0$ smoothly as $t\to 0$.
\end{theorem}
\noindent See Theorem \ref{EntireExistence} for further details.  Notice that we make no assumption on the spacelike condition at infinity, so we can start with initial data that asymptotically develops lightlike directions (like the Grim Reaper), and that we obtain long-time existence even without an initial bound on the mean curvature. This theorem is an extension and improvement of the codimension 1 result proven in \cite[Theorem 4.2]{EckerEntire} (see also Remark \ref{rem:AlternativeLocalEst}).
\begin{remark}
As is to be expected for entire flows we make no statement about uniqueness in Theorem \ref{thm:entire}, and solutions to (\ref{EntireMCF}) are not unique in general.  For example, if we take $M_0$ to be the Grim Reaper in Example \ref{ex:GrimReaper} at $t=0$, which is a translating solution, then any solution constructed by our proof of Theorem \ref{thm:entire} cannot remain a translator (since it would satisfy $|\U(x,t)-\U_0(x)|\leq \sqrt{2nt}$).
\end{remark}

We are also prove results on the qualitative behaviour of entire flows. In Section \ref{s:sensible} we develop further estimates for entire spacelike MCF in which, in particular, demonstrate the following result.

\begin{proposition}
 There are no shrinking or translating solutions to spacelike MCF with bounded gradient and mean curvature.
\end{proposition}

Finally, in Section \ref{s:sensibleconvergence} we show that if $M_0$ is asymptotic to a strictly spacelike cone, then the entire renormalised flow converges subsequentially to a self-expanding solution to MCF, see Theorem \ref{ConvergenceTheorem} for full details.

\subsubsection*{Boundary conditions} To prove Theorem \ref{thm:entire}, we solve auxiliary problems on compact domains with boundary conditions.  In this article we solve for both the Neumann and Dirichlet cases.
A key step in the proof of Theorem \ref{thm:entire} is that quasi-sphere expanders acts as barriers to the flow on compact domains, a notion we now define.  
\begin{defses}[Quasi-sphere expander]\label{quasisphere}  
 A \emph{quasi-sphere expander} with centre $p$ and starting square radius $-R^2$ is given by 
 \[S_t:=\{x\in\bb{R}^{n,m}|\, |p-x|^2 = -R^2-2nt\}\ .\]
We define the \emph{inside} of $S_t$ to be \[I_t:=\{x\in\bb{R}^{n,m}| \,|p-x|^2 \geq -R^2-2nt\}.\] An expanding quasi-sphere $S_t$ is said to be an \emph{outer barrier} if the property $M_t\subset I_t$ is preserved by the mean curvature  flow.
\end{defses}

We have the following existence and convergence theorems for mean curvature flow with Neumann and Dirichlet boundary conditions.  Here we need the initial submanifold to be uniformly spacelike, i.e.~the submanifold does not asymptotically develop lightlike directions at the boundary.  We first state the Neumann case.
 \begin{theorem}\label{thm:Neumann}
 Let $\Omega$ be a bounded convex domain with smooth boundary 
 and let $M_0$ be uniformly spacelike satisfying the Neumann boundary condition.  There exists a unique spacelike solution $M_t$ of mean curvature flow starting at $M_0$ satisfying the Neumann boundary condition, which is smooth, exists for all $t>0$, and converges smoothly to a translate of $\Omega$ as $t\ra\infty$.
  Furthermore, expanding quasi-spheres with centre in $\Omega \times \bb{R}^m$ act as outer barriers to the flow.
\end{theorem}
\noindent See Theorem \ref{Neumanntheorem} for a more precise statement, including conditions for regularity up to $t=0$.

For the Dirichlet condition, we require a constraint on the boundary data, which is called \emph{acausal} (see \eqref{acausal} for a definition): this condition is necessary on a convex domain to have a spacelike graph with the given boundary data.
 \begin{theorem}\label{thm:Dirichlet}
 Let $\Omega$ be a bounded domain with smooth boundary and let $M_0$ be uniformly spacelike with acausal boundary.
  There exists a unique spacelike solution $M_t$ to mean curvature flow starting at $M_0$ satisfying $\partial M_t=\partial M_0$ which is smooth, exists for all $t>0$ and converges smoothly to the unique maximal submanifold with boundary $\partial M_0$ as $t\ra\infty$.
 Furthermore, expanding quasi-spheres with centre in $\Omega \times \bb{R}^m$ act as outer barriers to the flow.
\end{theorem}
\noindent See Theorem \ref{Dirichlettheorem} for further details and a more precise statement, including conditions for improved regularity up to the initial time.  We emphasise that existence and uniqueness of a maximal submanifold with given acausal boundary data, given as a graph on a bounded convex domain, is shown in \cite[Theorem 2.1]{YangLi}.  Theorem \ref{thm:Dirichlet} extends this result to the MCF setting.

\subsection{Applications to $\GG_2$-Laplacian flow} If we view $\bb{R}^7=\bb{R}^3\times\bb{R}^4$ and let $(x_1,x_2,x_3)$ be coordinates on $\bb{R}^3$ and $(y_0,y_1,y_2,y_3)$ be coordinates on $\bb{R}^4$, we can define a 3-form $\varphi_0$ on $\bb{R}^7$ by
\begin{equation}\label{eq:phi0}
\varphi_0=-dx_1\wedge dx_2\wedge dx_3+dx_1\wedge \omega_1+dx_2\wedge\omega_2+dx_3\wedge\omega_3
\end{equation}
where 
\begin{equation}\label{eq:sd}
\omega_1=dy_0\wedge dy_1+dy_2\wedge dy_3,\, \omega_2=dy_0\wedge dy_2+dy_3\wedge dy_1,\,\omega_3=dy_0\wedge dy_3+dy_1\wedge dy_2.
\end{equation}
The stabilizer of $\varphi_0$, under the action of $\GL(7,\bb{R})$, is the exceptional Lie group $\GG_2$.  Given an oriented 7-manifold $Z^7$, we can define a 3-form $\varphi$ to be \emph{positive} if at every point $p\in Z$ there exists an orientation preserving isomorphism between $T_pZ$ and $\bb{R}^7$ identifying $\varphi|_p$ with $\varphi_0$.  A positive 3-form (which will exist if and only if $Z$ is also spin) naturally defines a principal $\GG_2$-subbundle of the oriented frame bundle of $Z$; in other words, a $\GG_2$-structure.  We therefore often call a choice of positive 3-form (or simply the 3-form itself) a $\GG_2$-structure.

The interest in $\GG_2$-structures $\varphi$ comes from the fact that they always define a metric $g_{\varphi}$ and an orientation (since $\GG_2\subset\text{SO}(7)$) and one sees that the holonomy group of  $g_{\varphi}$ is contained in $\GG_2$ when $\varphi$ is \emph{torsion-free}, which is equivalent to
\begin{equation}\label{eq:torsionfree}
d\varphi=0\quad\text{and}\quad d^*_{\varphi}\varphi=0.
\end{equation}
It should be noted here that the first equation is linear, whilst the second is nonlinear, since the adjoint $d^*_{\varphi}$ of the exterior derivative depends on $g_{\varphi}$ and the orientation $\varphi$ defines.  A metric with holonomy   contained in $\GG_2$ is Ricci-flat, and this is the only known means to obtain non-trivial examples of Ricci-flat metrics in odd dimensions.

Solving the torsion-free conditions \eqref{eq:torsionfree} is very challenging in general, with the only compact examples arising from sophisticated gluing techniques, going back to work of Joyce (see \cite{Joyce}).  The key to these methods is the fundamental work of Joyce, which allows one to perturb a closed $\GG_2$-structure (i.e.~one with $d\varphi=0$) which is ``close'' to torsion-free in a suitable sense, to become torsion-free.  As an alternative approach to the problem of solving \eqref{eq:torsionfree}, Bryant \cite{BryantRemarks} proposed the following $\GG_2$-Laplacian flow for closed $\GG_2$-structures:
\begin{equation}\label{eq:Lflow}
\begin{cases}
\displaystyle\ddt\varphi=\Delta_\varphi\varphi=(dd^*_{\varphi}+d^*_{\varphi}d)\varphi,\\[8pt]
d\varphi=0.
\end{cases}
\end{equation}
Important foundational results for this flow have been developed \cite{GaoChen,LotayWei3,LotayWei2,LotayWei1} and recent impressive results have been obtained in the special case when $Z^7=T^3\times N^4$, where $N^4$ is compact and the flow is $T^3$-invariant \cite{FineYao}.  In general, however, there are many unresolved questions concerning the $\GG_2$-Laplacian flow, in particular regarding long-time existence, convergence and the formation of singularities.

\subsubsection*{Semi-flat coassociative  $T^4$-fibrations} For our applications, we let $B$ be a domain in $\bb{R}^3$ and consider $Z^7=B\times T^4$, where $T^4=\bb{R}^4/\bb{Z}^4$ is the standard flat 4-torus, which we can view as a trivial $T^4$-fibration over $B$.  Everything we now describe can be found in \cite{Baraglia, DonaldsonAdiabatic, DonaldsonRemarks}.  

Recall the model $\GG_2$-structure $\varphi_0$ in \eqref{eq:phi0}.  This can equivalently be written as
\begin{equation*}
\varphi_0=-\text{vol}_{\bb{R}^3}+d(x_1\omega_1+x_2\omega_2+x_3\omega_3)
\end{equation*}
   Therefore, to define a $\GG_2$ structure on $Z$ we need to find a 2-form on $Z$ to play the role of $x_1\omega_1+x_2\omega_2+x_3\omega_3$. 
   Notice that constant 2-forms on $T^4$ are in one-to-one correspondence with cohomology classes in $H^2(T^4)$.
   We now observe that the cup product on $H^2(T^4)$
naturally identifies $H^2(T^4)$ with $\bb{R}^{3,3}$ via
 \[\langle [\alpha],[\beta]\rangle=\int_{T^4}\alpha\wedge\beta.\]  Thus, given an immersion $X:B\ra\bb{R}^{3,3}\cong H^2(T^4)$, we have that 
\begin{equation*}
\frac{\partial X}{\partial x_i}=[\omega_i]
\end{equation*}
for some unique constant 2-forms $\omega_i$.  We therefore see that we can write
\begin{equation*}
dx_1\wedge \omega_1+dx_2\wedge\omega_2+dx_3\wedge\omega_3=dX.
\end{equation*}
 We may then define
\begin{equation}\label{eq:phi}
\varphi=-X^*\text{vol}_{X(B)}+dX.
\end{equation}
It is observed in \cite{DonaldsonAdiabatic} that the condition for $\varphi$ in \eqref{eq:phi} to be positive, and thus to define a $\GG_2$-structure, is precisely that $X:B\ra\bb{R}^{3,3}$ is spacelike.  Notice that at each point $p$ of $X(B)$, the mean curvature $H$ of $X(B)$ in $\bb{R}^{3,3}\cong H^2(T^4)$ lies in the orthogonal complement of the maximal positive definite subspace $T_pX(B)=\text{Span}\{[\omega_1],[\omega_2],[\omega_3]\}$, and so $H(p)$ can be identified with an anti-self-dual 2-form on $T^4$ using the metric determined by the $\omega_i$. Moreover, by construction, $d\varphi=0$ and one finds from \cite[Lemma 6]{DonaldsonAdiabatic} that
\begin{equation}\label{eq:d*phi}
d^*_{\varphi}dX=H,
\end{equation}
where $H$ is identified with a 2-form on $Z$ 
as described above.   
Using the formula \eqref{eq:d*phi}, we see that
 if $X$ satisfies spacelike mean curvature flow \eqref{MCF} then $\varphi$ in \eqref{eq:phi} satisfies the $\GG_2$-Laplacian flow \eqref{eq:Lflow}.  This correspondence can also be deduced from the relationship between the volume form on $Z$ and the induced volume form on $X(B)$, since the $\GG_2$-Laplacian flow \eqref{eq:Lflow} is the gradient flow for the volume on $Z$ determined by $\varphi$, and spacelike mean curvature flow \eqref{MCF} is the gradient flow for the volume on $X(B)$.  

\begin{remark}
Formula \eqref{eq:d*phi} shows that the \emph{torsion} of the closed $\GG_2$-structure $\varphi$ in \eqref{eq:phi} is essentially the mean curvature $H$ of $X(B)$ in $\bb{R}^{3,3}$.  The (necessarily non-positive) scalar curvature of $g_{\varphi}$ is thus proportional to $\|H\|^2$, and so a bound on the scalar curvature (or equivalently the torsion) along the $\GG_2$-Laplacian flow in this setting directly corresponds to a bound on the mean curvature in spacelike MCF. 
\end{remark}

Notice that for the closed $\GG_2$-structure $\varphi$ in \eqref{eq:phi} on $Z^7$ we have that $\varphi$ vanishes on the $T^4$ fibres.  It follows, by the choice of orientation, that the restriction of the 4-form $*_{\varphi}\varphi$ to a $T^4$ fibre is equal to the volume form of the induced metric.  This means the fibres are \emph{coassociative}, i.e.~they are calibrated by $*_{\varphi}\varphi$.  Moreover, the fibres are obviously flat orbits of an isometric $T^4$-action for the metric $g_{\varphi}$, so the fibration is called \emph{semi-flat}.  

Suppose we have a 7-manifold $Z^7$ with a closed $\GG_2$-structure $\varphi$ that is a semi-flat coassociative $T^4$-fibration over a simply connected domain $B$ in $\bb{R}^3$, i.e.~$\varphi$ is $T^4$-invariant and the fibres are flat orbits of the action. Then the discussion above shows the following.

\begin{proposition}\label{prop:correspondence}  Suppose we have a 7-manifold $Z^7$ with a closed $\GG_2$-structure $\varphi_0$ that is a semi-flat coassociative $T^4$-fibration over a simply connected domain $B$ in $\bb{R}^3$. The $\GG_2$-Laplacian flow on $Z^7$ starting at $\varphi_0$ is equivalent to spacelike MCF of $B$ in $\bb{R}^{3,3}$.
\end{proposition}

\noindent This extends the correspondence in \cite{Baraglia} between torsion-free $\GG_2$-structures on semi-flat coassociative $T^4$-fibrations and maximal submanifolds in $\bb{R}^{3,3}$.

Proposition \ref{prop:correspondence}, together with our Theorems \ref{thm:entire} and \ref{thm:Dirichlet}, provide immediate long-time existence and convergence results for the $\GG_2$-Laplacian flow. 

\begin{theorem}  Let $(Z^7,\varphi_0)$ be a semi-flat coassociative $T^4$-fibration over $B$ as in Proposition \ref{prop:correspondence}.  
\begin{enumerate}[(a)]
\item If $B=\bb{R}^3$, there is a solution $\varphi_t$ to $\GG_2$-Laplacian flow \eqref{eq:Lflow} starting at $\varphi_0$ which is smooth and exists for all time. 
\item If $B$ is a bounded domain in $\bb{R}^3$ with smooth boundary, suppose that $\varphi_0$ is such that the boundary data for the corresponding initial immersion of $B$ in $\bb{R}^{3,3}$ is acausal.  
There is a solution $\varphi_t$ to the $\GG_2$-Laplacian flow \eqref{eq:Lflow} starting at $\varphi_0$ satisfying $\varphi_t|_{\partial Z}=\varphi_0|_{\partial Z}$, which is smooth, exists for all time and converges to a torsion-free $\GG_2$-structure $\varphi_\infty$ on $Z$ with 
$\varphi_{\infty}|_{\partial Z}=\varphi_0|_{\partial Z}$.
\end{enumerate}
\end{theorem}

\noindent Further discussion of the boundary value problem for semi-flat coassociative $T^4$-fibrations, as well as coassociative K3 fibrations, can be found in \cite{DonaldsonRemarks}.  In particular, it is expected that notions of convexity for the boundary value of a closed $\GG_2$ structure suggested in \cite{DonaldsonRemarks} will imply the acausal boundary condition for the corresponding immersion of $B$ in $\bb{R}^{3,3}$. 

\subsubsection*{Adiabatic limits}  Another important area of study in $\GG_2$ geometry is that of coassociative K3 fibrations.  Here, the curvature of the K3 fibres makes the torsion-free condition more difficult to analyse.  However, one can still make a similar ansatz for a closed $\GG_2$-structure as in \eqref{eq:phi}, now using a spacelike immersion $X:B\ra \bb{R}^{3,19}\cong H^2(K3)$.   In \cite{DonaldsonAdiabatic}, Donaldson studies the torsion-free condition in the \emph{adiabatic limit} as the volume of the fibres tends to zero.  If the volume of the fibres is $\epsilon$, then the $\GG_2$-Laplacian flow corresponds to spacelike MCF in the limit as $\epsilon\to 0$.  Thus, our long-time existence and convergence results for spacelike MCF in $\bb{R}^{3,19}$ have significant implications for the study of the $\GG_2$-Laplacian flow for coassociative K3 fibrations.

\subsection{Summary}

We briefly summarise the contents of this article.

\smallskip

\noindent Sections \ref{s:prelim}--\ref{s:local} consist of background material and the derivation of the key evolution inequalities we shall require for our study. 
\begin{itemize}
\item  In Section \ref{s:prelim} we introduce the basic notation and quantities.  
\item In Section \ref{s:evol} we derive  essential evolution equations and inequality for key quantities.  
\item We then localise these evolution inequalities in Section \ref{s:local}.
\end{itemize}
\noindent Sections \ref{s:entire}--\ref{s:Dirichlet} contain the proofs of our main results. 
\begin{itemize}
\item In Section \ref{s:entire} we establish  long-time existence of entire graphical  solutions to spacelike MCF, 
assuming the long-time existence of solutions on bounded convex domains satisfying Neumann boundary conditions, which we prove in Section  \ref{s:Neumann}.
\item   We also prove the convergence of spacelike MCF
with Neumann conditions in Section \ref{s:Neumann}.
\item  In Section \ref{s:Dirichlet}, we  prove long-time existence and convergence of spacelike MCF on bounded  domains satisfy Dirichlet boundary conditions; namely, that the fixed boundary data is acausal (see Definition \ref{acausal}). Here, we make use of work in \cite{YangLi}.
\end{itemize}

\noindent Sections \ref{s:sensible}--\ref{s:sensibleconvergence} concern the properties of entire solutions.
\begin{itemize}
\item Since we have non-uniqueness for entire spacelike MCF, in Section \ref{s:sensible} we introduce the notion of {\sensible} solutions which satisfy mild natural assumptions.  We show that {\sensible} solutions satisfy estimates similar to expanders, and deduce some immediate consequences.
\item In Section \ref{s:sensibleconvergence} we show that under the assumption that $M_0$ converges to a spacelike cone at infinity (in $C^0$), then the renormalised flow converges subsequentially in $C^\infty_\text{loc}$ to a self-expanding solution to the flow.
\end{itemize}
\noindent Appendices \ref{MCFasGraphs}--\ref{app:evol} consist of further technical results.
\begin{itemize}
\item In Appendix \ref{MCFasGraphs} we derive the expression of spacelike MCF and describe the Neumann boundary condition in terms of graphs.  
\item In Appendix \ref{app:unique} we demonstrate uniqueness for the flow over compact domains.
\item In Appendix \ref{app:evol} we derive the evolution equation for an ambient symmetric 2-tensor along the flow.
\item In Appendix \ref{app:extension} we demonstrate that boundary quantities may be suitably extended.
\end{itemize}

\begin{ack}
This research was supported by Leverhulme Trust Research Project Grant RPG-2016-174.
\end{ack}

\section{Preliminaries}\label{s:prelim}

In this section we introduce some notation we shall employ throughout the article, and we describe the basic quantities that are needed for our study.

\subsection{Basic notation}
Throughout this paper we employ summation convention between raised and lowered indices, where lower-case Latin indices range over $1\leq i,j,k, \ldots \leq n$, upper-case Latin indices range over $1\leq A,B,C, \ldots \leq m$, and  Greek indices range over $1\leq \a,\beta ,\gamma, \ldots \leq n+m$. 

We take the standard orthonormal basis $\{f_1, \ldots, f_n, e_1, \ldots, e_m\}$ of $\bb{R}^{n,m}$ such that each $f_i$ is spacelike and each $e_A$ is  timelike.  Therefore, if $x=x^if_i+x^{n+A}e_A$ and $y=y^if_i+y^{n+A}e_A$ then  the $\bb{R}^{n,m}$ scalar product is given by
\[\ip{x}{y}=x^i\delta_{ij}y^j-x^{n+A}\delta_{AB}y^{n+B} .\]
We recall that for any vector $x\in \bb{R}^{n,m}$, we let
\[|x|^2 = \ip{x}{x},\]
 and reiterate that (despite appearances) this quantity can be negative.

We let $M$ be an $n$-dimensional spacelike submanifold of $\bb{R}^{n,m}$.
We let $X:\Omega\ra \bb{R}^{n,m}$, for some $\Omega\subset\bb{R}^n$, denote the position vector of $M$.  We also let $x^i$ be coordinates on $\Omega$ and let
\[X_i:=\frac{\partial X}{\partial x^i}.\]

We write the usual Levi-Civita connection on $\bb{R}^{n,m}$ by $\ov \n$.  For any tangent vector fields $U, V$ on $M$ and normal vector field $\nu$ we write the induced and normal connection  as
\[ \n_U V = \left(\ov\n_U V\right)^\top,\qquad \np_U \nu =\left(\ov\n_U \nu\right)^\perp\] 
respectively. We will often use the abbreviated notation $\ov\n_i=\ov\n_{X_i}$, $\n_i=\n_{X_i}$ and $\n_i^{\perp}=\n_{X_i}^{\perp}$. We may now use the usual definition to extend tensor derivatives to tensors taking values in the normal bundle, for example for a tensor (1+1)-covariant tensor $T = T(V, \nu)$ (for $U$, $V$ and $\nu$ as above) then
\begin{equation}
 \n_U T(V, \nu) = U(T(V, \nu)) - T(\n_U V, \nu) - T(U, \n^\perp_U \nu). \label{tensorderiv}
\end{equation}

Throughout we will assume that at any point $p\in M$, $\nu_1, \ldots, \nu_m$ are an orthonormal frame of $N_pM$.
To avoid sign confusion for timelike quantities, for any timelike $z\in\bb{R}^{n,m}$ we write
\[0\leq \|z\|^2 = -|z|^2\ .\]
This defines a norm on $NM$.

When dealing with a graph defined by $\U^A:\Omega \ra\bb{R}$ for $1\leq A\leq m$, we will write
\[\U = \sum_{A=1}^m \U^Ae_A\ .\]

\subsection{Gradients} As in \cite{Bartnik} we will require a quantity on a spacelike manifold $M$ that measures how close to lightlike the manifold is at a point. To this end we define the projection matrix
\[W_{AB}:= \ip{e_A}{\nu_B}\]
and the partial gradients
\[w^2_A := \|e_A^\perp\|^2=\sum_{B=1}^m W_{AB}W_{AB}\ .\]
We define the full gradient to be
\begin{equation}\label{eq:v}
v^2 = \sum_A w_A^2=\sum_A \|e_A^\perp\|^2  ,
\end{equation}
which is essentially a matrix norm of $W_{AB}$. We note that this is different to the ``determinant-type'' gradient in \cite{LiSalavessa} (although, due to the arithmetic-geometric inequality, they are equivalent).  We see that a bound on $v(p)$ gives a measure of how close $T_pM$ is to the lightcone and so is a key quantity.  In particular, it leads to the following definition.

\begin{defses}\label{def:uspacelike} An $n$-dimensional submanifold $M$ of $\bb{R}^{n,m}$ is 
 \emph{uniformly spacelike} if there is $C>0$ such that for all $p\in M$, $v(p)<C$, where $v$ is defined in \eqref{eq:v}.
\end{defses}

The gradient also plays a vital role in the estimate of ambient tensors on the flowing submanifold. Suppose $T$ is an $(a+b)$-covariant tensor on $\bb{R}^{n,m}$. Defining
\[T_{i_1, \ldots, i_a, A_1, \ldots, A_b} := T(X_{i_1}, \ldots, X_{i_a}, \nu_{A_1}, \ldots, \nu_{A_b}),\]
 we have that
\begin{equation}|T_{i_1, \ldots, i_a, A_1, \ldots, A_b}|\leq v^{a+b}|T|_{\bb{R}^{n+m}}\label{Generaltensorest}
\end{equation}

It will also be useful to consider the evolution of $\|X^\perp\|^2$.
This will be used to calculate the gradient of cutoff functions, ultimately allowing us to obtain local gradient estimates that only depend on $C^0$ bounds. We will repeatedly use that, since $\ov \n |X|^2 = 2X$, we have
\begin{equation}|\n |X|^2|^2 =4|X - X^\perp |^2 =4(|X|^2 - |X^\perp |^2) =  4(|X|^2 + \|X^\perp \|^2).
 \label{JgradX}
\end{equation}

\subsection{Flow quantities}
Suppose $V$ is any time-dependent vector field on $M$. As in \cite{Smoczyk}, for any given local coordinates $y^{\alpha}$ on $\bb{R}^{n,m}$ in a neighbourhood of a point in $M$ we define
\[\ov\n_\ddt{} V=\ov\n_\ddt{} \left[V^\alpha \pard{}{y^\alpha}\right] = \left[\pard{V^\alpha}{t} + V^\beta H^\gamma\ov\Gamma_{\beta\gamma}^\alpha\right]\pard{}{y^\alpha},\]
where $\ov\Gamma_{\beta\gamma}^\alpha$ are the Christoffel symbols of $\ov\n$ with respect to the coordinates $y^{\alpha}$.   
 This is compatible with the metric:
\[\ddt{}\ip{V}{W} = \ip{\ov\n_\ddt{} V}{W}+ \ip{V}{\ov\n_\ddt{}W}\ .\]
Furthermore, it is easy to see that
\[\nt X_i = \ov \n_{X_i} \ddt{X}=\ov\n_iH,\]
and if $V$ is a vector field on $\bb{R}^{n,m}$ then along MCF we have
\[\nt V = \ov \n_H V  .\]

For example,
since $\ip{\nu_A}{X_i}=0$, we have
\begin{flalign*}
\ip{\nt \nu_A}{X_i}& 
=-\ip{\nu_A}{\ov\n_{X_i} H} .
\end{flalign*}
On the other hand $\ip{\nt \nu_A}{\nu_B}$ is not determined by the flow.  However, we may make the following choice.
\begin{lemma} \label{ddtnu}
 Suppose that $V$ is a set with compact closure in the preimage of $X$ and at some time $t_0>0$, for all $p\in V$ there exists a smoothly varying orthonormal basis $\nu_1(p,t_0), \ldots, \nu_m(p,t_0)$ of $N_{X(p)}M_{t_0}$. Then there exists an $\e>0$ such that for all $t\in (t_0-\e, t_0+\e)$ and $p\in V$ there exists an orthonormal basis $\nu_1(p,t), \ldots, \nu_m(p,t)$ of $N_pM_t$ such that
 \[\nt \nu_A = -\ip{\nu_A}{\n^\perp_i H}g^{ij}X_j  .\]
\end{lemma}
\begin{proof}
 Due to the above calculations there exists an orthonormal frame $\tilde \nu_1, \ldots, \tilde\nu_m$ such that
 \[\nt \tilde{\nu}_A = -\ip{\nu_A}{\n^\perp_i H}g^{ij}X_j +C_A^B\tilde{\nu}_B ,\]
 where $C_A^B$ is smooth in both time and space and satisfies $C_A^B = -C_B^A$ (as $\ip{\tilde{\nu}_A}{\tilde{\nu}_B}$ is constant). We suppose that $Z(p,t)$ is a time-dependent $m\times m$ matrix and define
 \[\nu_A = Z_A^B\tilde{\nu}_B .\]
 We see that 
 \[\nt \nu_A = -\ip{\nu_A}{\n^\perp_i H}g^{ij}X_j +\left[\ddt{Z_A^D}+Z_A^BC_B^D\right]\tilde{\nu}_D .\]
 By the Picard--Lindel\"of Theorem, there exists a unique solution on $V$ to $\ddt{Z_A^D}=-Z_A^BC_B^D$ for $t\in(t_0-\e,t_0+\e)$ starting from $Z_A^B(p,t_0)=\delta_A^B$ for all $p\in V$. We see that for such a solution, writing $Z^t$ for the transpose of $Z$,
 \[\ddt{}(Z_A^B Z^t{}_B^C) = -Z_A^DC_D^BZ^t{}_B^C -Z{}_A^BC^t{}_B^DZ^t{}_D^C=0 ,\]
due to the skew-symmetry of $C$. Since  $Z_A^B$ is orthogonal at time $t_0$,  $Z_A^B$ is orthogonal for all $t\in(t_0-\e,t_0+\e)$. The $\nu_A$ therefore form an orthonormal frame with the claimed property.
\end{proof}

All normal quantities in evolution equations below will be calculated locally in such a basis. Furthermore, to avoid sign changes we raise and lower normal indices by minus the normal metric:
\[T^A=T_B\delta^{BA}  .\]

\subsection{Curvature}
For $U,V\in T_pM$, we define the second fundamental form by
\[\II(U,V)=\left(\ov\n_U V\right)^\perp\]
and we use the notation
\[\II_{ij} = \II(X_i,X_j) = (\ov \n_{X_i} X_j)^\perp\ .\]
Therefore,  for $\nu \in NM$, 
\[\left(\ov \n_i \nu\right)^\top = -\ip{\nu}{\II_{ij}}g^{jk} X_k\]

We write 
\[h_{ij}^A = -\ip{\II_{ij}}{\nu_A}\quad\text{
so that}\quad
\II_{ij} = h_{ij}^A\nu_A\ .\]
Similarly we define
\[H^A = -\ip{H}{\nu_A} = g^{ij}h_{ij}^A\ .\]
We observe that
\[\np_k \II_{ij} = \n_kh_{ij}^A\nu_A\ .\]
The Codazzi--Mainardi equations imply that 
 \[\n^\perp_U\II(V,W) = \n_U^\perp\II(W,V)=\n^\perp_V\II(U,W).\]

\section{Evolution equations}\label{s:evol}

In this section we derive equations and inequalities that are satisfied by key quantities along the spacelike MCF.  We will use the notation $\Delta=g^{ij}\nabla_{ij}^2$ for the Laplacian on $M$ and recall that $M_t$ denotes the spacelike submanifold at time $t$ along the mean curvature flow starting at $M$.  

\subsection{$C^0$ quantities}
We begin with the following standard observation.
\begin{lemma}\label{GeneralC0evol}
 Let $f: \bb{R}^{n,m}\ra \bb{R}$ be a $C^2$ function. Then under MCF  we have
  \[\ho f = -g^{ij} \ov \n^2_{ij} f.\]
\end{lemma}
\begin{proof}
 We have:
\begin{align*}
\n _{ i}f& = \ip{\ov \n f}{X_i},\\
  \n^2_{ij} f &= \ip{\ov \n _{X_j}(\ov \n f)}{X_i} + \ip{\ov \n f}{\ov \n_{X_j} X_i - \n_{X_j} X_i} = \ov \n^2_{X_i X_j} f +\ip{\ov \n f}{\II_{ij}},\\
\Delta f &=g^{ij} \ov \n^2_{X_i X_j} f +\ip{\ov \n f}{H},\\
 \ddt{} f
  & = \ip{\ov \n f}{H}.
\end{align*} 
 The result follows immediately from these formulae.
\end{proof}

We define the following quantities at $x\in\bb{R}^{n,m}$:
\begin{align*}
 u_A&:=-\ip{e_A}{x},\\
 r^2 &:= |x|^2+\sum_A u_A^2.
\end{align*}
 Since $\ov \n^2_{UV} |x|^2 = 2\ip{U}{V}$ we have the following corollary to Lemma \ref{GeneralC0evol}.
\begin{cor}\label{C0evol} Under mean curvature flow,
\begin{align*}
&\ho |X|^2 = -2n, & \ho (R^2 - 2nt -|X|^2) = 0,&\\
&\ho u_A = 0, & \ho r^2 = -2n -2\sum_A (w_A^2-1) .&
\end{align*}
\end{cor}
\noindent The second equation in Corollary \ref{C0evol} shows that we have a good cutoff function with support on shrinking quasi-spheres.

\subsection{Gradient quantities}
We first write down a  general observation.
\begin{lemma}\label{Generalnormaloneform}
 Suppose that $V$ is a smooth $1$-form on $\bb{R}^{n,m}$. We write the components of the restriction of this tensor to 
 $NM_t$ as
 \[
 V_A = V(\nu_A) .\]
 Then $V_A$ satisfies
 \[\n_kV_A = \ov \n_k V_A + V(X_l)h^{lA}_k\ ,\]
 and, choosing $\nu_1, \ldots, \nu_m$ locally as in Lemma \ref{ddtnu}, we have that
 \[\ho V_A = -g^{ij}\ov\n_i\ov\n_j V_A -2\ov\n_i V(X_l) h^{il}_A-V_Bh^{B}_{ij}h^{ij}_A\ .\]

\end{lemma}
\begin{proof}
 We see that
 \begin{flalign*}
 \ddt{} V_A &= \ov \n_H V(\nu_A) +V(\nt \nu_A)
 =\ov \n_H V(\nu_A) -V(X_j)\ip{\nabla_j^{\perp} H}{\nu_A}.
 \end{flalign*}
 We calculate
\[\n_kV_A = \ov \n_k V_A +V((\ov \n_k \nu_A)^\top) = \ov \n_k V_A -V(X_l)\ip{\II^l_k}{\nu_A}.\]
Furthermore, recalling (\ref{tensorderiv}) we have that,
\begin{flalign*}
 \n_j\n_i V_A &=\pard{}{x^j}(\n_i V_A) - \n_{\n_{X_j} X_i} V(\nu_A) - \n_{X_i} V(\n^\perp_{X_j}\nu_A)\\
 &=\ov\n_{X_j} \ov \n_{X_i} V (\nu_A)+\ov\n_{\ov \n_{X_j} X_i} V(\nu_A)+\ov\n_{X_i} V(\ov \n_{X_j}\nu_A)\\
 &\qquad - \ov \n_{X_j} V(X_l)\ip{\II^l_i}{\nu_A}- V(\II_{jl})\ip{\II^l_i}{\nu_A}\\
 &\qquad-V(X_l)\left[\ip{\n^\perp_j \II^l_i+\II(X^l, \n_{X_j}X_i)}{\nu_A}+\ip{\II^l_i}{\n^\perp_{X_j}\nu_A}\right]\\
 &\qquad - \n_{\n_{X_j} X_i} V(\nu_A) - \n_{X_i} V(\n^\perp_{X_j}\nu_A)\\
 &= \ov \n_{X_j}\ov\n_{X_i} V(\nu_A) + \ov \n_{\II_{ij}} V(\nu_A) +\ov \n_{X_i} V (\left(\ov \n_{X_j} \nu_A\right)^\top)+\ov \n_{X_j}V(X_l)h^l_{iA}\\
 &\qquad-V(\II_{jl})\ip{\II^l_i}{\nu_A}-V(X_l)\ip{\np_j\II^l_i}{\nu_I}\\
 &= \ov \n_{j}\ov\n_{i} V_A + \ov \n_{\II_{ij}} V_A +\ov \n_{i} V_kh_{jA}^k+\ov \n_jV(X_l)h_{iA}^l+V_Bh^B_{jl}h^l_{iA}\\
 &\qquad-V(X_l)\ip{\np_j\II^l_i}{\nu_A}.
\end{flalign*}
Finally, using Codazzi--Mainardi gives the claimed result.
\end{proof}

\begin{lemma}\label{wkJevol}  Along MCF we have that
 \begin{align} 
 \ho\! w^2_A &=-2\ip{e_A}{\II_{ik}}g^{ij}g^{kl}\ip{\II_{jl}}{e_A}\! +\!2g^{ij}\ip{\II(X_i, e_A^\top)}{\II(X_j, e_A^\top)}, \nonumber\\
 \ho\! \|X^\perp\|^2 &=-2\ip{X}{\II_{ij}}\ip{\II^{ij}}{X}\!-\!4\ip{H}{X} \!+ 2\ip{\II(X_i, X^\top)}{\II(X^i, X^\top)},\nonumber
 \end{align}

\vspace{-16pt} 
 
 \begin{align*}
 \n_i w^2_A =2\ip{e_A}{\II(X_i, e_A^\top)}\quad\text{and}\quad
  \n_i \|X^\perp\|^2 = 2\ip{X}{\II(X_i,X^\top)} .
 \end{align*}
\end{lemma}
\begin{proof}
We define the 1-form $V(Z)=\ip{e_A}{Z}$ on $\bb{R}^{n,m}$, and note that $\ov \n V =0$, and $\ov\n^2 V=0$. We deduce from Lemma \ref{Generalnormaloneform} that:
\begin{align*}
\ho V_B &= -V_Ch^C_{ij}h^{ij}_B=\ip{e_A}{\II_{ij}}\ip{\II_{ij}}{\nu_B},\\
\n_k V_B &= V(X_l)h^{l}_{kB}=-\ip{\II(e_A^\top,X_k)}{\nu_B}.
\end{align*}
As $w_A^2=\|e_A^\perp\|^2 = \sum_{B=1}^m (\ip{e_A}{\nu_B})^2=\sum_{B=1}^m V_B^2$, we see that
 \begin{flalign*}
 &\ho w^2_A = \sum_B \left[2V_B \ho V_B - 2 \n_i V_B\n^i V_B\right]\\
 &\qquad=\sum_B [2\ip{\nu_B}{e_A} \ip{\nu_B}{\II_{ik}}g^{ij}g^{kl}\ip{\II_{jl}}{e_A}\\
 &\qquad\qquad\quad 
 - 2 \ip{\nu_B}{\II(X_i, e_A^\top)}g^{ij}\ip{\nu_B}{\II(X_j, e_A^\top)}]\\
 &\qquad=-2\ip{e_A}{\II_{ik}}g^{ij}g^{kl}\ip{\II_{jl}}{e_A}-2\sum_B\ip{\nu_B}{\II(X_i, e_A^\top)}g^{ij}\ip{\nu_B}{\II(X_j, e_A^\top)}.
 \end{flalign*}
We also have
\[\n_i w^2_A =2 \sum_B V_{B}\n_i V_{B} = -2\sum_B\ip{\nu_B}{e_A}\ip{\nu_B}{\II(X_i, e_A^\top)}=2\ip{e_A}{\II(X_i, e_A^\top)}.\]

Similarly, we define the 1-form $U=\ip{X}{Z}$, where $X$ is the position vector, and we see that $\ov \n_Y U(X) = \ip{Y}{X}$ and $\ov \n^2 U=0$. An identical argument to the above yields the claimed equations for $\|X^{\perp}\|^2$. 
\end{proof}

As is often the case with MCF in indefinite spaces, the key to a local gradient estimate is to estimate the first term in the evolution of $w_A^2$ in Lemma \ref{wkJevol}  by slightly more than twice the gradient of $w_A$ using an eigenvalue estimate, originally employed by Bartnik \cite[Theorem 3.1]{Bartnik} (see also \cite{EckerEntire,EckerNull,EckerHuiskenSpacelike} for similar arguments).

\begin{cor}\label{wkevolest}
We may estimate
\[\ho w_A^2 \leq -\frac{|\n w^2_A|^2}{w_A^2}\]
or
 \[\ho w_A^2 \leq -\left(1+\frac{1}{2n}\right)\frac{|\n w^2_A|^2}{w_A^2}+2w_A^2\|H\|^2.\]
\end{cor}
\begin{proof}
 For the second term in the evolution of $w_A^2$ in Lemma \ref{wkJevol}, we have that  
 \begin{flalign*}
 |\n w^2_A|^2
 &= 4\|e_A^\perp\|^2 \ip{\frac{e_A^\perp}{\|e_A^\perp\|}}{\II(X_i, e_A^\top)}g^{ij}\ip{\frac{e_A^\perp}{\|e_A^\perp\|}}{\II(X_j, e_A^\top)} \\
 &\leq -4w_A^2 \ip{\II(X_i, e_A^\top)}{\II(X_j, e_A^\top)}g^{ij}.
 \end{flalign*}

We now estimate the first term in the evolution of $w_A^2$ in Lemma \ref{wkJevol}. Write $t_{ij} = \ip{e_A}{\II_{ij}}$ and the eigenvalues of $t_{ij}$ as $\l_1, \ldots, \l_n$ where $\l_1$ is the largest in absolute value.
Using $-1=|e_A^\top|^2-w_A^2$, we have that 
\[|\n w_A^2|^2 \leq 4\l_1^2|e_A^\top|^2=4 \l_1^2(w_A^2-1)\leq 4\l_1^2w_A^2.\]
The first estimate in the statement follows from the fact that $|t|^2\geq\l_1^2$. 

Since for any symmetric tensor $b_{ij}$, $n|b|^2\geq(\text{tr}b)^2$, we have that
\[|t|^2 = \l_1^2+\ldots+\l_n^2\geq \l_1^2 +\frac{1}{n-1}(\sum_{i=2}^n \l_i)^2\geq (1+\frac{1}{n})\l_1^2 - g^{ij}t_{ij},\]
where we used Young's inequality for the last estimate. We now see that
\begin{flalign*}
-2\ip{e_A}{\II_{ik}}g^{ij}g^{kl}\ip{\II_{jl}}{e_A}&\leq -2\left(1+\frac{1}{n}\right)\l_1^2+2(\ip{e_A}{H})^2\\
&\leq -\left(1+\frac{1}{n}\right)\frac{|\n w^2_A|^2}{2w_A^2}+2w_A^2\|H\|^2.
\end{flalign*}
The second estimate in the statement now follows.
\end{proof}
\begin{remark} The second estimate in Corollary \ref{wkevolest} is  the same as the evolution inequality satisfied by the codimension 1 gradient (see \cite[Corollary 2.5]{EckerEntire}).
\end{remark}

We now derive an evolution inequality for $\|X^\perp\|$ in a similar way.
\begin{cor}\label{Jevolest}
For $\e\in(0,1)$, at any point such that $|X|^2<\e\|X^\perp\|^2$ we have that
\[\ho \|X^\perp\|^2\leq  -\left(\frac{1}{2} + \frac{1+\frac{1}{n}}{2(1+\e)}\right)\frac{|\n\|X^\perp\|^2|^2}{\|X^\perp\|^2} -4 \ip{X}{H}+2\ip{X}{H}^2 .\]
\end{cor}
\begin{proof}
Since
\[|\n\|X^\perp\|^2|^2=4\ip{X}{\II(X^\top,X_i)}\ip{\II(X^i,X^\top)}{X} ,\]
we immediately see that 
\[- 2\left|\ip{\II(X_i, X^\top)}{\II(X^i, X^\top)}\right|\leq -\frac{|\n\|X^\perp\|^2|^2}{2\|X^\perp\|^2}.\]
Using equation (\ref{JgradX}), and estimating as in Corollary \ref{wkevolest}, we have
\[-2\ip{X}{\II_{ij}}\ip{\II^{ij}}{X}\leq-\left(1+\frac{1}{n}\right)\frac{|\n\|X^\perp\|^2|^2}{2(|X|^2+\|X^\perp\|^2)}+2\ip{H}{X}^2. \]
Altogether these estimates imply
\begin{align*}
\ho \|X^\perp\|^2 &\leq -\frac{1}{2}\left(1+\left(1+ \frac 1 n \right)\frac{\|X^\perp\|^2}{|X|^2+\|X^\perp\|^2}\right)\frac{|\n \|X^\perp\|^2|^2}{\|X^\perp\|^2}\\
&\qquad -4\ip{H}{X}+2\ip{H}{X}^2.
\end{align*}
The upper bound on $|X|^2$ now yields the claim.
\end{proof}

We also observe   we may use the height function to get a large negative evolution for the gradient (depending on local bounds on $u$). Similar ideas were used in 
\cite{Gerhardt}.
\begin{lemma}
 We have that \label{evolgoodwk}
\begin{flalign*}
 \ho w_A^2 e^{u_A^2}&\leq -2e^{u_A^2} w_A^2 (w_A^2-1) .
 \end{flalign*}
\end{lemma}
\begin{proof}
Recall Corollary \ref{C0evol}, Corollary \ref{wkevolest} and 
$|\n u_A|^2 = (w_A^2-1)$.  We may estimate using Young's inequality, for any smooth positive function $\phi:\bb{R}\ra\bb{R}$:
\begin{flalign*}
 \ho w_A^2 \phi(u_A) &\leq -\phi\frac{|\n w_A^2|^2}{w_A^2}-2 \phi'\ip{\n w_A^2}{\n u_A}-\phi''w_A^2|\n u_A|^2\\
 &\leq w_A^2\left(\frac{(\phi')^2}{\phi}|\n u_A |^2 -\phi'' |\n u_A|^2\right)\ \ .
\end{flalign*}
We choose to write $\phi = e^{\psi}$. Then
\[\phi' = \psi' \phi, \qquad \phi'' = \psi'' \phi + (\psi')^2 \phi\]
and
\begin{flalign*}
 \ho w_A^2 \phi(u_A) &\leq -\psi''\phi w_A^2 |\n u_A|^2= -\psi''\phi w_A^2 (w_A^2-1).
\end{flalign*}
Setting $\psi=u_A^2$ yields the claim.
\end{proof}

\subsection{Curvature quantities}

We recall the following evolution equations which may be found in \cite[Proposition 4.1, equation (5.7)]{LiSalavessa}.

\begin{lemma}\label{evolcurv}  The following evolution equations hold:
 \begin{flalign*}
\ho h_{ij}^A &=-h^A_{kr}h^{kr}_Bh^B_{ij}+2h_{kl}^Ah_i^{Bl}h_{Bj}^k - h_{il}^Ah^{Blk}h_{Bkj}- h_{jl}^Ah^{Blk}h_{Bki}\\
&\qquad +h_{ik}^Ah_j^{Bk}H_B+h_{jk}^Ah_i^{Bk}H_B;\\
\ho H^A &=-h^A_{kr}h^{kr}_BH^B;\\
\ho \|H\|^2 &= -2\ip{H}{\II_{ij}}\ip{H}{\II^{ij}} - 2\|\np H\|^2;\\
\ho \| \II\|^2 &=-2h^{ilA}h_{il}^Bh_A^{kr}h_{krB} -2|h_{iA}^{k}h_{kjB}-h_{iB}^lh_{jlA}|^2 - 2|\n_ph_{ik}^A|^2.
 \end{flalign*}
\end{lemma}
\begin{cor}\label{evolcurv2} The following evolution inequalities hold:
 \begin{align*}\ho \|H\|^2 &\leq -\frac{2}{n} \|H\|^4 - 2|\np H|^2;\\
  \ho \|\II\|^2 &\leq -\frac{2}{m}\|\II\|^4 - 2\|\np\II\|^2.
 \end{align*}
\end{cor}
\begin{proof}
Setting $r_{ij} = \ip{H}{\II_{ij}}$, we have $|r|^2\geq n^{-1} (\text{tr} \,r )^2 = n^{-1}\|H\|^2$, which gives the first inequality. The second follows by the same trick applied to $S^{AB} = h^{ilA}h_{il}^B$.
\end{proof}

\section{Local estimates}\label{s:local}

In this section we localise the estimates for our key quantities using an appropriate cut-off function.  To this end, we introduce the following notation.

\begin{defses}\label{def:cyl.sqsphere.cutoff}
We define the solid cylinder of radius $R$ centred at $p$ to be
\[\mathcal{C}_R(p):=\{x\in\bb{R}^{n,m}| r^2(x-p)< R^2\}\subset\bb{R}^{n,m},\]
and  the solid quasi-sphere centred at $p\in \bb{R}^{n,m}$ by 
\[\mathcal{Q}_R(p): = \{x\in\bb{R}^{n,m}| |x-p|^2<R^2\}\subset\bb{R}^{n,m},\]
where if the centre is omitted,   it is assumed to be $0\in\bb{R}^{n,m}$, that is
\begin{align*}
\mathcal{C}_R&:=\{x\in\bb{R}^{n,m}| r^2(x)< R^2\}\quad\text{and}\quad
\mathcal{Q}_R:=\{x\in\bb{R}^{n,m}| |x|^2< R^2\} . 
\end{align*}

We now define 
\[\eta_R := (R^2-|X|^2-2nt)_+\]
and note that $\eta_R(X)>0$ for $t<(2n)^{-1}R^2$ if and only if $X\in \mathcal{Q}_{\sqrt{R^2 -2nt}}$. 
\end{defses}

We remark that the above quasi-spheres have \emph{positive} square radius, and the support of the cutoff function collapses on to the interior of the light cone as $t$ goes to $\frac{R^2}{2n}$. These should not be confused with the (negative square radius) expanding quasi-spheres of Definition \ref{quasisphere}. 

We now include a lemma which is essentially \cite[Lemma 3.3]{EckerEntire}.
\begin{lemma} \label{evolutioncutoff}
Suppose that, along MCF, a $C^2$-function $f:\bb{R}^{n,m}\to\bb{R}^+$ satisfies
\[\ho f \leq g-(1+\d)\frac{|\n f|^2}{f}\]
for some function $g$ and $\delta>0$.  Then there exists  $p=p(\d)>0$ such that, writing $\eta_R = \eta_R(X)$, we have
 \begin{flalign*}
 \ho f\eta_R^p &\leq g\eta_R^{p} .
 \end{flalign*} 
 Furthermore for any $\Lambda>0$, there exists $q=q(\d,\Lambda)>0$ such that
 \begin{flalign*}
 \ho f\eta_R^q &\leq g\eta_R^{q} - \Lambda f\eta^{q-2}_R|\n|X|^2|^2\ .
 \end{flalign*} 
\end{lemma}
\begin{proof}
Let $p\geq 2$.  If $t<(2n)^{-1}R^2$ and $x\in \mathcal{Q}_{\sqrt{R^2 -2nt}}$ then at $(x,t)$:
 \begin{flalign*}
 \ho f\eta_R^p &\leq f\eta_R^{p-2}\bigg[ \eta_R^{2}f^{-1}g-(1+\d)\eta_R^2\frac{|\n f|^2}{f^2}-2p\eta_R\ip{\frac{\n f}{f}}{\n \eta_R}\\
 &\qquad\qquad -p(p-1)|\n \eta_R|^2\bigg]\\
 &\leq f\eta_R^{p-2}\left[ \eta_R^{2}f^{-1}g +p\left(1-\frac{\d}{1+\d}p\right)|\n |X|^2|^2\right].
 \end{flalign*}
Setting $p=\max\{\frac{1+\d}{\d},2\}$ 
gives the first equation. The second equation follows simply by making $q$ sufficiently larger than $p$.
\end{proof}

We first observe  we may easily get a local estimate for $w_A^2$ if $\|H\|^2$ is bounded.
\begin{lemma}\label{localw_KHbound}
 Suppose that for all $t\in[0,\frac{R^2}{2n})$, $\text{supp}\, \eta_R\cap M_t$ is compact and for all $y\in \text{supp} \,\eta_R\cap M_t$, \[\|H\|^2(y,t)<C_H.\] Then there exists $p>0$ such that for all $t\in[0,\frac{R^2}{2n})$,
 \[\sup_{M_t}w_A^2 \eta_R^p\leq e^{2C_H t}\sup_{M_0}w_A^2\eta_R^p.\]
\end{lemma}
\begin{proof}
Using Corollary \ref{wkevolest} and
 Lemma \ref{evolutioncutoff}, there exists  $p>0$ such that
\[\ho w_A^2\eta_R^pe^{-2C_H t} \leq 0.\]
The maximum principle now yields the result.
\end{proof}
\begin{remark}\label{rem:AlternativeLocalEst}
 As $w_A$ and $\|H\|^2$ have the same evolution as the gradient and square of the mean curvature in the codimension 1 case, we may follow an identical proof to \cite[Theorem 3.1]{EckerEntire}. However, note that to apply the maximum principle to \[g=\frac{w_A^2}{(\Lambda-\|H\|^2)^\frac{1}{q}}(R^2-|X|^2-2nt)^p\]
 we require that it is (at least) continuous; i.e.~we require a bound on $\|H\|^2$ so that the denominator is never zero (which is seemingly absent from the hypotheses of \cite[Theorem 3.1]{EckerEntire}). 
 Lemma \ref{localw_KHbound} therefore yields an equivalent statement.
\end{remark}

In general, we may not have a uniform bound on $\|H\|^2$ as assumed in Lemma \ref{localw_KHbound}. However, in the applications in Section \ref{s:entire}, we will have bounds of the form 
\[t\|H\|^2 \leq \frac{n}{2},\]
and we now prove local estimates under this assumption.

\begin{lemma}\label{Jtest}
 Suppose that for all $t\in[0,\frac{R^2}{2n})$, $\text{supp} \,\eta_R\cap M_t$ is compact and there is some $L>0$ such that, for all  $y\in \text{supp} \,\eta_R\cap M_t$, 
 \[t\|H\|^2(y,t) \leq \frac{n}{2} \qquad \text{ and }\qquad |X|^2(y,t)>-L.\]
 There exists $q=q(n)$ and $C=C(n, R, L)$ such that for all $t\in[0,\frac{R^2}{2n})$,
 \[\sup_{M_t}t\|X^\perp\|^2\eta_R^q(y,t) \leq C.\]
\end{lemma}
\begin{proof}
Due to Corollary \ref{Jevolest} (setting $\e = (1+2n)^{-1}$), we have that at any point where $|X|^2<\frac{\|X^\perp\|^2}{1+2n}$:
\[\ho \|X^\perp\|^2\leq  -\left(1+ \frac{1}{4n}\right)\frac{|\n \|X^\perp\|^2|^2}{\|X^\perp\|^2} +\frac{2\sqrt{2n} \|X^\perp\|}{\sqrt t}+n\frac{\|X^\perp\|^2} t .\]
We consider $f=t\|X^\perp\|^2\eta_R^q$ where $q$ is chosen as in Lemma \ref{evolutioncutoff} with $\Lambda=1$, which we want to show is uniformly bounded to prove the statement. 

Suppose that at time $t_0$, $y_0\in M_{t_0}$ is an increasing maximum of $f$ (that is, $f$ has a maximum in space at $y_0$ and $\ddt{f}(y_0,t_0)\geq 0$). Then we have that either 
$\|X^\perp\|^2\leq (1+2n)|X|^2$ (which implies that $f\leq C(n,R)$), or at $(y_0,t_0)$
\begin{flalign*}0&\leq\ho t\|X^\perp\|^2\eta_R^q \\&\leq  -t\|X^\perp\|^2 |\n|X|^2|^2\eta_R^{q-2}+2\sqrt{2nt} \|X^\perp\|\eta_R^q+n\|X^\perp\|^2\eta_R^q+\|X^\perp\|^2\eta_R^q\\
 &\leq -4t\|X^\perp\|^4\eta_R^{q-2} +4tL\|X^\perp\|^2\eta_R^{q-2}+2\sqrt{2nt} \|X^\perp\|\eta_R^q+(n+1)\|X^\perp\|^2\eta_R^q\ .
\end{flalign*}
Therefore at any increasing maximum of $f=t\|X^\perp\|^2\eta_R^q$ such that $R^2(1+2n)<\|X^\perp\|^2$, we have
\[f^2\leq 2n^{-1}LR^2f\eta_R^{q}+t\sqrt{8nf} \eta_R^{\frac{3}{2}q+2}+(n+1)f\eta_R^{q+2} ,\]
which implies
$f \leq C(R, n, L)$ due to our chosen range of $t$. 
\end{proof}

We now use Lemma \ref{Jtest} to show a full local gradient bound.
\begin{lemma}\label{localtgradest}
 Suppose that for all $t\in[0,\frac{R^2}{2n})$,  $\text{supp}\, \eta_R \cap M_t$ is compact and there is some $C_u>0$ such that, for all $y\in \text{supp} \,\eta_{R}\cap M_t$, 
 \[t\|H\|^2(y,t) \leq \frac{n}{2}\qquad \text{ and }\qquad \|u\|^2(y,t)<C_u.\] 
There exists  $p=p(n)$ and $C=C(n, R, C_u)$ such that for all $t\in[0,\frac{R^2}{2n})$,
 \[\sup_{M_t}tv^2\eta_R^p(y,t) \leq C .\]
\end{lemma}
\begin{proof}
The bound on $\|u\|^2$ implies that $|X|^2>-C_u$. As a result we may apply Lemma \ref{Jtest} to give that
 \[t\|X^\perp\|\eta_{R}^q<C(n,R,C_u) .\]

Setting $f = w_A^2e^{u_A^2}$, by Lemma \ref{evolgoodwk} we have   
\[\ho f \leq -cf^2+Cf\]
where $c$ and $C$ depend only on $C_u$. Therefore at an increasing maximum of $tf\eta_R^p$, for $p\geq q+2$, we may calculate:
 \begin{flalign*}
 0\leq\ho tf\eta_R^p 
 &=-c t\eta_R^p f^2+Cft\eta_R^p +f\eta_R^p -2\ip{\n f}{\n \eta_R^p} \\&\quad+tf \left( - p(p-1)\eta_R^{p-2} |\n |X|^2|^2\right)\\
 &\leq-c t\eta_R^p f^2+Cft\eta_R^p +f\eta_R^p +tfp(p+1)\eta_R^{p-2}|\n |X|^2|^2,
\end{flalign*}
where we used that $\n(tf\eta_R^p)=0$ on the final line. Using the estimate on $\|X^\perp\|^2$ and $|X|^2$, 
\[c t f^2\eta_R^p\leq Cft\eta_R^p +f\eta_R^p +tCf\eta_R^{p-2}+Cf,\]
We therefore obtain that $tw_A^2e^{u_A^2}\eta_R^p=tf\eta_R^p<C$, where $C$ depends only on $n, R, C_u$. Summing the estimates on the $w_A^2$ and using the bound on $\|u\|^2$ gives the lemma.
\end{proof}

We now prove local estimates on the second fundamental form.

\begin{lemma}$\ $\label{LocalCurvEst}\vspace{-10pt}
\begin{itemize}\item[]
\item[(a)] Suppose the hypotheses of Lemma \ref{localw_KHbound} hold, and additionally we have a uniform estimate 
\[v^2(y,t)\leq C_v\]
for all $y\in M_t$. There exists  $C_1 = C_1(n,R,C_v)$ such that
 \[\|\II\|^2\eta_R^p\leq C_1\underset{M_0}\sup \|\II\|^2\eta_R^p  .\]
\item[(b)] Suppose the hypotheses of Lemma \ref{localtgradest} hold. There exists $C_2=C_2(c,R,C_u)$ such that
 \[t\|\II\|^2\eta_R^p\leq C_2  .\]
\end{itemize}
\end{lemma}
\begin{proof}
Part (a) follows by a calculation similar to the proof of Lemma \ref{localtgradest}, but estimating $|\n |X|^2|^2$ and using the estimates of Lemma \ref{localw_KHbound} instead of  Lemma \ref{Jtest}.

Part (b) is identical to the proof of Lemma \ref{localtgradest}, replacing $f$ with $f=\|\II\|^2$, and using that $\ho \|\II\|^2 \leq -\frac{2}{m}\|\II\|^4$. 
\end{proof}
\begin{remark}\label{localcurvestremark}
Once we have a uniform bound $v^2<C_v$, we may use an identical proof to Lemma \ref{LocalCurvEst} but replacing $\eta_R^p$ with $\tilde{\eta}^2_R(r)=(R^2-r^2)^2_+$. This is exactly as in \cite[Proposition 3.6]{EckerEntire} and  yields estimates in cylinders of the form
 \[\sup_{M_t}\|\II\|^2\tilde{\eta}^2_R(r) \leq \sup_{M_0}\|\II\|^2\tilde{\eta}^2_R(r).\]
\end{remark}

We conclude this section with local higher order estimates on the second fundamental form.

\begin{lemma}
 \label{higherorderlocalest}
 Suppose that for all $t\in[0,T)$, $y\in\mathcal{C}_R\cap M_t$,
 \[v^2(y,t)<C_v, \qquad \|\II\|^2(y,t)<C_\II,\]
 for some constants $C_v, C_\II$. Then there exists a constant \[C_k := C_k(C_v, C_\II,n, m, k,\max_{1\leq l\leq k} \underset{M_0\cap\mathcal{C}_R}\sup \|\n^l\II\|^2)\] such that for all $t\in[0,T)$, 
 \[\sup_{M_t\cap\mathcal{C}_\frac{R}{2}}\|\n^k\II\|^2\leq C_k.\]
\end{lemma}
\begin{proof}
All required evolution equations may be estimated as in the codimension one case,  so the proof of \cite[Proposition 3.7]{EckerEntire} applies without alteration. 
\end{proof}

\section{Entire solutions}\label{s:entire}

We now demonstrate the long-time existence of entire graphical solutions to spacelike MCF, assuming the long-time existence of solutions on bounded convex domains satisfying Neumann boundary conditions, which we defer to Section \ref{s:Neumann}. 

The following lemma will be used to give the compactness hypothesis required in our local estimates, and is a higher codimension version of \cite[Proposition 1]{ChengYauMaximal}.
\begin{lemma}\label{goodgood}
 Suppose $0\in M_0$, $M_0$ is  spacelike and is given by the graph of $\U_0:\bb{R}^n\to\bb{R}^m$.
 Then
\begin{equation}\label{eq:goodgood1}\lim_{R\ra \infty}\inf_{x\in \bb{R}^n\setminus B_R}\left[ |x|^2 -\|\U_0\|^2\right]=\infty,
\end{equation}
or equivalently (recalling the notation of Definition \ref{def:cyl.sqsphere.cutoff}),
\begin{equation*}\lim_{R\ra \infty}\inf_{M_0\setminus \mathcal{C}_R}|X|^2=\infty.
\end{equation*}
Moreover, there exists  $\e>0$ such that
\begin{equation}\label{eq:goodgood2}\|\U_0\|\leq \chi_\e
\end{equation}
where
\[
\chi_\e(x):= 
\begin{cases}
 1-\e &\text{ for } x\in \ov{B_1(0)},\\
 |x|-\e&\text{ for } x\in \bb{R}^n\setminus B_1(0).
\end{cases}
\]
 \end{lemma}
\begin{proof}
 Suppose \eqref{eq:goodgood1} does not hold. Then there exists a constant $C>0$ and a sequence of points $x_i\in \bb{R}^n$ such that $1<|x_i|\ra \infty$ but $y_i:=x_i+\U_0(x_i)\in M_0$ has 
\begin{equation}
 |y_i|^2=|x_i|^2-\|\U_0(x_i)\|^2<C .\label{lies}
\end{equation}
Since $M_0$ is spacelike there exists $\e>0$ such that for all $x\in \partial B_1(0)$, 
\begin{equation}\e \leq |x|^2-\|\U_0(x)\|^2=1-\|\U_0(x)\|^2.
 \label{ineq1}
\end{equation}

We define $\tilde{x}_i:=\frac{x_i}{|x_i|}\in \partial B_1(0)$ and $\tilde{y}_i:=\tilde{x_i}+ \U_0(\tilde{x_i}) \in M_0$. Since $M_0$ is spacelike,
\begin{equation}0\leq|\tilde{y}_i-y_i|^2= (|x_i|-1)^2-\|\U_0(\tilde{x_i})-\U_0(x_i)\|^2.
 \label{ineq2}
\end{equation}
Equations (\ref{ineq1}) and (\ref{ineq2}) and the triangle inequality for the norm $\|\cdot\|$ now imply
\begin{flalign*}
|y_i|^2 &\geq |x_i|^2 -(\|\U_0(x_i) -\U_0(\tilde{x}_i)\|+\|\U_0(\tilde{x}_i)\|)^2\\
&\geq|x_i|^2 -(|x_i|-1 +1-\e)^2 =2\e|x_i|-\e^2.
\end{flalign*}
This contradicts (\ref{lies}) as $i\ra \infty$.

Equation \eqref{eq:goodgood2} follows from (\ref{ineq1}) and the fact that $M_0$ is spacelike.
\end{proof}

We  now prove our claimed long-time existence result in the entire setting.

\begin{theorem} \label{EntireExistence}Suppose that $M_0$ is smooth,  spacelike  and  given by the graph of $\U_0:\bb{R}^n\to\bb{R}^m$. 
There exists a solution $$\U \in C^\infty_\text{loc}(\bb{R}^n\times(0,\infty))\cap C^{0,\a}_\text{loc}(\bb{R}^n\times[0,\infty))$$ to graphical spacelike MCF \eqref{EntireMCF} satisfying  
 \[|\U(x,t)-\U_0(x)|\leq \sqrt{2nt}.\] 
 Furthermore, if there exists a constant $C_H>0$ such that 
\begin{equation}\sup_{M_0}\|H\|^2<C_H ,
 \label{initialHbound}
\end{equation}
then $\U \in C^\infty_\text{loc}(\bb{R}^n\times[0,\infty))$ 
and
\[\|H\|^2 \leq \frac{1}{(C_H+1)^{-1}+\frac{2} n t}\ .\]
\end{theorem}
\begin{proof}
\textbf{Case 1: $\|H\|^2$ bounded initially.} We first suppose  (\ref{initialHbound}) holds. Without loss of generality we assume that $\U_0(0)=0$.

By Lemma \ref{goodgood} there exist radii $R_i$ such that the graph of $\U_0$ over $\bb{R}^n\setminus B_{R_i}$ satisfies 
\begin{equation}
 \inf_{x\in \bb{R}^n\setminus B_{R_i}}\left[ |x|^2 - \|\U_0(x)\|^2 \right]\geq 2i+1 \label{wbound}
\end{equation}
and therefore  $M_0\cap\partial\mathcal{C}_{R_i}$ lies outside $\mathcal{Q}_{R_i}$, in the notation of Definition \ref{def:cyl.sqsphere.cutoff}.

We will now solve a sequence of auxiliary problems for spacelike MCF with Neumann boundary conditions and use our interior estimates to show that these converge to a solution of \eqref{EntireMCF}. Unfortunately, to get a solution to our auxiliary problem which is smooth to $t=0$, we need our initial data to satisfy compatibility conditions. To this end we now describe a way to modify $\U_0$ so  the initial data satisfies compatibility conditions of all orders. 

For some $\Lambda>0$, we define
\[\tilde{u}_{0,i}(x) = 
\begin{cases}
\U_0(x) &\text{for } x\in B_{R_i+\Lambda},\\
\U_0((R_i+\Lambda-|R_i+\Lambda - |x||)\frac{x}{|x|}) &\text{for } x\in  B_{2R_i+2\Lambda}\setminus B_{R_i+\Lambda},\\
0&\text{for } x\in  \bb{R}^n\setminus B_{2R_i+2\Lambda}.\\
\end{cases}
\]
Clearly $\tilde{u}_{0,i}$ is continuous (as $\U_0(0)=0$) and smooth away from the boundary   $(\partial B_{R_i+\Lambda})\cup(\partial B_{2R_i+2\Lambda})$ of the annulus. We now let $\U_{0,i}$ be a smoothing of $\tilde{u}_{0,i}$ such that $\U_{0,i} =\tilde{u}_{0,i}$ on $\ov{B_{R_i}}$ and $\U_{0,i} \equiv 0$ on $\bb{R}^n\setminus B_{2R_i+3\Lambda}$. By choosing $\Lambda$ large enough depending only on $C_H$, we may assume that the mean curvature of $\U_{0,i}$ satisfies $\|H\|^2<C_H +1$ and $|\U_{0,i}|^2<|\U_0|^2+1$. 

We now consider the spacelike MCF problems with Neumann boundary conditions given by
\begin{equation}
 \begin{cases}
  \displaystyle \ddt{\U_{i}} -g^{kl}(D\U)D^2_{kl} \U_i=0 & \text{ on } B_{2R_i+3\Lambda+1}\times[0,T),\\[4pt]
  D_\frac{x}{|x|} \U_i=0 & \text{ on } \partial B_{2R_i+3\Lambda+1} \times[0,T),\\
  \U_i(\cdot,0) =\U_{0,i}(\cdot) & \text{ on } B_{2R_i+3\Lambda+1}.
 \end{cases}
\end{equation}
Since compatibility conditions of all orders are satisfied, Theorem \ref{Neumanntheorem} below implies there exists a solution $\U_i\in C^\infty(B_{2R_i+2\Lambda}\times[0,\infty))$. 
Furthermore, as $\U_i$ is sufficiently regular, we have the bound  \[\|H_i\|^2\leq((1+C_H)^{-1} +2n^{-1}t)^{-1},\] which is uniform in $i$.   

We let $M_{t,i}:=\text{graph} \,\U_{i}$.
 By Lemma \ref{goodgood} and the preservation of height bounds for $\|u_i\|$, for all $t>0$ and $j>i$ we have 
$\partial M_{t,j}\cap\mathcal{Q}_{R_i}=\emptyset$. Therefore, each  $M_{t,j}\cap\mathcal{Q}_{R_i}$ has 
compact closure. 
 We may now apply the interior estimates of Lemma \ref{localw_KHbound} to obtain that for all $x\in M_{t,j}\cap \mathcal{Q}_{\sqrt{R_i^2-2nt}}$, we have a uniform bound on $v$. 
In particular, for all $t<\frac{R_i^2}{4n}$ and $j>i$,  we may apply Lemmas \ref{LocalCurvEst} and \ref{higherorderlocalest} on $M_{t,j}\cap \mathcal{C}_{\frac{R_i}{2}}$ 
to imply uniform $C^{k;\frac{k}{2}}$ bounds for all $k$.

We now use the Arzel\`a--Ascoli theorem to take a diagonal sequence which converges in $C^\infty_\text{loc}$ to the claimed solution $\U \in C^\infty_{\text{loc}}(\bb{R}^n\times[0,\infty))$ satisfying (\ref{EntireMCF}).

\medskip
\paragraph{\textbf{Case 2: No initial $\|H\|^2$ bound.}} If $\|H\|^2$ is not bounded initially, we proceed solving auxiliary problems as above, but this time on the solutions $\U_i$, we only have  
 \begin{equation} |\U_{0,i}(x) - \U_i(x,t)|\leq \sqrt{2nt}\qquad \text{ and }\qquad \|H\|^2\leq \frac{n}{2t} .
  \label{freeestimates}
 \end{equation}
 
 However since $\|\U_0\|<\chi_\e$ by Lemma \ref{goodgood}, 
we see that for any $x\in \partial B_1(0)$ the quasi-sphere centred at $\e x$ of radius $-(1-\e)^2$ contains  
 $\U_0$ and therefore $\U_{0,i}$ is contained within this quasi-sphere for any $i$. We therefore see, by evolving such solutions that
 \[\|\U_i(\cdot,t)\|<\chi_{\e,t},\]
 where
 \begin{equation*}
  \chi_{\e,t}(x)=\sqrt{\left|x-\e\frac{x}{|x|}\right|^2+1-\e^2+2nt}=\sqrt{|x|^2-2\e|x|+1+2nt}.
 \end{equation*}
We also observe that (writing points in $\bb{R}^{n,m}$ as pairs $(x,z)$ for $x\in\bb{R}^n$ and $z\in\bb{R}^m$) 
\[V=\{(x,z)\in\bb{R}^{n,m} | \|z\|<\chi_{\e,t}(x) \}\]
has compact intersection with $\mathcal{Q}_R$ 
 for any finite $R$.  Indeed, at any point $(x,z)\in V\cap \mathcal{Q}_R$,
\[|x|^2-R^2<\|z\|^2<|x|^2-2\e|x|+1+2nt,\]
which implies that $|x|\leq(2\e)^{-1}(1+2nt+R^2)$. As a result of the above, writing $M_t^i = \text{graph}\, \U_i(\cdot, t)$, for any time $t\in[0,T)$ and any point $(x, \U(x,t))\in M^i_t \cap\mathcal{Q}_R$, we have the estimate
\[\|\U_i(x,t)\|<C(\e,T,R),\]
which is uniform in $i$.

We may now apply Lemmas \ref{localtgradest} and \ref{LocalCurvEst} to obtain estimates on gradient and curvature on $M_t^i\cap \mathcal{Q}_\frac{R}{4}$ for all $t\in (0,\frac{R^2}{4n})$ which are uniform in $i$. Uniform higher order estimates also follow from Lemma \ref{higherorderlocalest}.
 
Taking a diagonal sequence as before now yields $C^\infty_\text{loc}$ convergence to a solution $\U$ which is smooth for $t>0$. Since for each $\U_i$ the estimates (\ref{freeestimates}) hold, 
these estimates pass to the limit, providing the claimed regularity of $\U$ to time $t=0$. 
\end{proof}

 \section{Neumann boundary conditions}\label{s:Neumann}
In this section we suppose that our spacelike MCF is over a compact domain $\Omega\subset\bb{R}^n \subset \bb{R}^{n,m}$ 
with smooth boundary $\partial \Omega$.  We shall prove long-time existence and convergence of spacelike MCF under the assumption of 
Neumann boundary conditions.  This, in particular, completes the proof of long-time existence in the entire setting of Section \ref{s:entire}.
 
\begin{defses}
 We define the \emph{boundary manifold} to be the hypersurface \[\Sigma:=\partial \Omega \times \bb{R}^m \subset \bb{R}^n\times \bb{R}^m = \bb{R}^{n,m}.\] 
 We denote the unit outwards (spacelike) normal to $\Sigma$ by $\mu$ and, by abuse of notation, we will also write the outward pointing unit normal to $\partial \Omega \subset \bb{R}^n$ as $\mu$. 
We denote the second fundamental form of $\Sigma$ by \[\AS(X,Y)=-\ip{\ov \n_X Y}{\mu},\] and observe that this tensor has $m$ zero eigenvectors in the directions $e_1, \ldots, e_m$. 
  The sign has been chosen so that the remaining eigenvalues are nonnegative if  $\Omega$ is convex. 
 \end{defses}
  
 We consider a Neumann boundary condition by requiring that at $\Sigma$, the normal space to 
$M_t$ must be contained in $T\Sigma$; that is, for any basis $\nu_1, \ldots, \nu_m$, 
 \begin{equation}\ip{\nu_A}{\mu}=0\label{boundarycondition}
 \end{equation}
 for $A=1, \ldots, m$. MCF with a Neumann boundary condition is therefore a one-parameter family of immersions of a disk, $X:D^n\times [0,T)\ra \bb{R}^{n,m}$, such that
 \begin{equation}
 \begin{cases}
 \left(\ddt{X}\right)^\perp  = H\ &\text{ on } D^n\times [0,T), \\
 X(\cdot,t) = X_0(\cdot) & \text{ on } D^n,\\
 X(\partial D^n, t) \subset \Sigma & \text{ for all } t\in[0,T),\\
 \ip{\nu_A}{\mu}=0 &\text{ on } \partial D^n\times [0,T)  .
 \end{cases}
 \label{MCFNeumann}
\end{equation}
Equivalently this may be rewritten in graphical coordinates. 
We say that $\U:\Omega\times[0,T)\ra \bb{R}$ satisfies MCF with a Neumann boundary condition if 
\begin{equation}\label{MCFgraphNeumann}
 \begin{cases}
  \displaystyle\ddt \U -g^{ij}(D\U)D^2_{ij} \U=0 & \text{ on } \Omega \times[0,T),\\[4pt]
  D_\mu \U=0 & \text{ on } \partial \Omega \times[0,T),\\
  \U(\cdot,0) =\U_0(\cdot) & \text{ on } \Omega.
 \end{cases}
\end{equation}
 See Appendix \ref{MCFasGraphs} for details.

Since we have boundary conditions, if we want a solution which does not ``jump'' at time $t=0$, we need some compatibility conditions (as mentioned in Section \ref{s:entire}). 
  Clearly we will require the zero order compatibility condition
\[D_\mu \U_0(x)=0\]
for all $x\in \partial \Omega$ and more generally the $l^\text{th}$ order compatibility condition 
\[D_\mu \frac{d^l}{dt^l}\U_0(x)=0\]
 for all $x\in \partial \Omega$, where $\frac{d}{dt} \U_0$ is defined recursively using the first line of (\ref{MCFgraphNeumann}). 
Higher order regularity is important as otherwise we cannot apply the maximum principle to quantities such as curvature to get estimates that depend on the initial data.
 
 We now state our long-time existence and convergence theorem in this Neumann setting.  The proof is somewhat lengthy and technical, and forms the remainder of this section.  Here we give an outline of the proof assuming 
the key technical results we shall prove below.  Recall the notion of uniformly spacelike from Definition \ref{def:uspacelike} and what it means for a quasi-sphere to be an outer barrier  in Definition \ref{quasisphere}.

 \begin{theorem}\label{Neumanntheorem}
 Suppose $\Omega$ is a bounded convex domain with smooth boundary $\partial \Omega$ and $\U_0$ is smooth, uniformly spacelike and satisfies compatibility conditions to 
 $l^\text{th}$ order for some $l\geq 0$. There exists 
 a solution $$\U \in C^{1+2l+\alpha;\frac{1+2l+\alpha}{2}}(\Omega \times [0,\infty))\cap C^\infty(\Omega \times (0,\infty))$$ of \eqref{MCFgraphNeumann} 
which is unique if $l\geq1$ and converges smoothly to a constant function as $t\ra \infty$. Furthermore, expanding quasi-spheres centred in $\Omega\times\bb{R}^m$ act as outer barriers to the flow, and we have the uniform bounds 
 \begin{flalign*}
  |\U_0(x)-\U(x,t)| &\leq \sqrt{2nt}\quad\text{and}\quad\|\U\|^2\leq \sup_{M_0} \|\U_0\|^2 ,
 \end{flalign*}
 and, if $l\geq 1$,
 \[ \|H\|^2 \leq \frac{1}{(\sup_{M_0}\|H\|^2)^{-1}+\frac{2}{n}t} .\]
\end{theorem}
\begin{proof}
 Although (\ref{MCFgraphNeumann}) is a system, it has linear boundary conditions and is in the form of $m$ parabolic PDEs.  Therefore, standard application of fixed point theory and Schauder estimates for parabolic PDEs, for example by minor modifications of \cite[Theorem 8.2]{Lieberman}, one obtains short time existence: there exists $T>0$ such that a solution to \eqref{MCFgraphNeumann} exists with $\U\in C^{l+1+\a;\frac{l+1+\a}{2}}(\Omega\times[0,T))\cap C^\infty(\Omega\times(0,T))$). 
 
 As stated in Appendix \ref{MCFasGraphs}, the components $\U^A$ of $\U$ satisfy a uniformly parabolic PDE (given by the first line of \eqref{MCFgraphNeumann}) if and only if $v^2$ is bounded. Furthermore, we may apply standard Schauder estimates as soon as we know that 
 $\U^A\in C^{1+\alpha;\frac{1+\alpha}{2}}(\Omega\times [0,T))$ \emph{for all }$A\in\{1,\ldots,m\}$. 
 
 In Lemma \ref{C0EstimatesNeumann} we demonstrate uniform $C^0$ estimates for solutions to (\ref{MCFNeumann}). In Lemma \ref{EstimatesNeumann} we give uniform estimates on $v^2$, which imply both uniform parabolicity and $C^1$ estimates on $\U$. As the Nash--Moser--De Giorgi estimates do not hold for systems we then derive uniform curvature estimates (which imply $C^2$ estimates) in Proposition \ref{curvestNeumann}. As Schauder estimates now apply, by bootstrapping we have the long time existence claimed.
 
 Lemma \ref{ConvergenceNeumann} then implies that the solution converges smoothly to a constant. Uniqueness of the solution is proven in Proposition \ref{Uniqueness}.
\end{proof}

\subsection{Boundary derivatives}

We first study derivative conditions at the boundary, particularly those which are consequences of the Neumann boundary condition.  We begin with two elementary observations.

\begin{lemma}\label{boundaryu} On $\partial M_t$ we have that
$$\n_\mu u^A =0.$$ 
\end{lemma}
\begin{proof} We see that $\n_\mu u^A =\ip{\mu}{e_A}=0$.
\end{proof}

\begin{lemma}\label{boundaryz}
Suppose that $\Omega$ is convex and $0\in \Omega$. Then, at any $y\in \partial M_t$ 
we have
 $$\n_\mu |X|^2(y) >0.$$
\end{lemma}
\begin{proof}
 We have that
$\n_\mu |X|^2 
=  2\ip{\mu}{X} > 0$
 due to the convexity of $\Omega$.
\end{proof}

 By differentiating the Neumann boundary condition \eqref{boundarycondition}, we immediately have the following consequences. As these estimates will be applied to curvature quantities, we need a sufficiently differentiable solution for the curvature evolution equations to be valid. From now on we will assume that $\U\in C^{4+\alpha;\frac{4+\alpha}{2}}(\Omega\times[0,T))$, but remark that often this is overkill, for example for estimates on the gradient we only really require $\U\in C^{1+\alpha;\frac{1+\alpha}{2}}(\Omega\times[0,T))\cap C^{3+\alpha;\frac{3+\alpha}{2}}(\Omega\times(0,T))$. Some of the boundary identities below hold in even weaker function spaces.
 \begin{lemma}\label{boundaryderivatives}
  Suppose that we have a solution to \eqref{MCFgraphNeumann} in 
  $C^{4+\alpha;\frac{4+\alpha}{2}}(\Omega\times[0,T))$. Then for any $t\in[0,T)$, $y\in \partial M_t$ and $U \in T_yM \cap T_y \Sigma$,
    \begin{align*}
    \II(U,\mu)+\sum_A \AS(\nu_A, U)\nu_A &=0\quad\text{and}\quad
    \np_\mu H +\sum_{A}\AS(\nu_A, H)\nu_A=0.
    \end{align*}
 \end{lemma}
\begin{proof}
 We differentiate (\ref{boundarycondition}) in direction $U$ to obtain
 \[0=U\ip{\nu_A}{\mu} =\ip{\ov \n_U \nu_A}{\mu}+\ip{\nu_A}{\ov \n_U\mu} = -\ip{\nu_A}{\II(U,\mu)}+\AS(\nu_A, U),\]
which yields the first claimed equation.
 
  We differentiate (\ref{boundarycondition}) in time to get
 \[0=\ip{\nt \nu_A}{\mu}+\ip{\nu_A}{\nt \mu} = -\ip{\nu_A}{\np_\mu H} + \AS(\nu_A, H),\]
 giving the second claimed equation.
\end{proof}

 \begin{cor} \label{boundaryident}   Suppose that we have a solution to \eqref{MCFgraphNeumann} in  $C^{4+\alpha;\frac{4+\alpha}{2}}(\Omega\times[0,T))$.  Then for any $t\in[0,T)$ on $\partial M_t$ we may calculate that 
 \[\n_\mu w_A^2 =-2\AS(e_A^\top, e_A^\top)\quad\text{and}\quad \n_\mu \|H\|^2 = -2\AS(H,H).\]
 \end{cor}
\begin{proof}
We have 
\[\n_\mu w_A^2 = -2\sum_B \ip{\nu_B}{e_A}\ip{\nu_B}{\II(\mu, e_A^\top)}=2\ip{e_A^\perp}{\II(\mu, e_A^\top)} .\]
Lemma \ref{boundaryderivatives} implies
\[\n_\mu w_A^2 =2 \AS(e_A^\perp, e_A^\top)=2 \AS(e_A-e_A^\top, e_A^\top)=-2\AS(e_A^\top, e_A^\top).\]

 We have that $\n_\mu \| H\|^2 = -\n_\mu|H|^2 =-2\ip{\np_\mu H}{H}$. From Lemma \ref{boundaryderivatives} we have  
 \[\ip{\np_\mu H}{H}-\AS(H,H)=0,\]
 so the result follows.
\end{proof}

We now calculate the second derivatives in space of the boundary condition. At any point $p\in\partial M_t$ we choose an orthonormal basis $E_1, \ldots, E_{n-1}$ of $T_pM_t\cap T_p\Sigma$. For the rest of this section, indices written with a hat such as $\hat{\imath}, \hat{\jmath}, \hat{k}, \ldots$ will be assumed to have values in $1, \ldots, n-1$.
\begin{lemma}\label{secondbdryderivs}
Suppose that we have a solution to \eqref{MCFgraphNeumann} in  $C^{4+\alpha;\frac{4+\alpha}{2}}(\Omega\times[0,T))$. Then for any $t\in[0,T)$, $y\in \partial M_t$ and $U,V \in T_yM \cap T_y \Sigma$ we calculate that at $y$,
 \begin{flalign*}
  \np_\mu \II(U,V)&= \!\sum_{A=1}^{m}\left[-\n^\Si_{\nu_A}\AS(U,V)-\AS(\nu_A, \II(U,V))\right]\nu_A +\AS(U,V)\II(\mu,\mu)\\
 &\qquad-\sum_{{\hat{\imath}}=1}^{n-1} \left[\II(V,E_{\hat{\imath}})\AS(E_{\hat{\imath}},U)+\II(U,E_{\hat{\imath}})\AS(E_{\hat{\imath}},V)\right].
 \end{flalign*}
\end{lemma}
\begin{proof}
We have that
\begin{flalign*}
 0&=V[\ip{\nu_A}{\II(U,\mu)} - \AS(\nu_A,U)]\\
 &=\ip{\np_V \nu_A}{\II(U,\mu)} +\ip{\nu_A}{\np_V \II(U,\mu)}+\ip{\nu_A}{\II(\n_V U,\mu)}+\ip{\nu_A}{\II(U,\n_V\mu)}\\
 &\qquad- \n^\Si_Y\AS(\nu_A,U) - \AS(\n^\Si_V \nu_A, U) - \AS(\nu_A,\n^\Si_V  U)\\
  &=\ip{\nu_A}{\np_V \II(U,\mu)}- \n^\Si_V\AS(\nu_A,U) +\ip{\np_V \nu_A}{\II(U,\mu)}- \AS(\n^\Si_V \nu_A, U) \\
 &\qquad +\ip{\nu_A}{\II(\n_V U,\mu)} - \AS(\nu_A,\n^\Si_V  U)+\ip{\nu_A}{\II(U,\n_V\mu)}\\
 &=\ip{\nu_A}{\np_\mu \II(U,V)}- \n^\Si_{\nu_A}\AS(U,V) -\AS(U,V)\ip{\nu_A}{\II(\mu,\mu)}\\
 & -\!\AS(\nu_A, \II(U,V))\!+\!\sum_{\hat{\imath}=1}^{n-1} \left[\ip{\nu_A}{\II(V,E_{\hat{\imath}})}\AS(E_{\hat{\imath}},U)\!+\!\ip{\nu_I}{\II(U,E_{\hat{\imath}})}\AS(E_{\hat{\imath}},V)\right]
\end{flalign*}
where we used Lemma \ref{boundaryderivatives} repeatedly to obtain the third equality.
\end{proof}
\begin{cor}  Suppose that we have a solution to \eqref{MCFgraphNeumann} in  $C^{4+\alpha;\frac{4+\alpha}{2}}(\Omega\times[0,T))$.
Then, for all $t\in[0,T)$, on $\partial M_t$ we calculate that,
 \begin{flalign*}
\n^\perp_\mu \II(\mu,\mu) &=\sum_{A=1}^{m}\sum_{\hat{\imath}=1}^{n-1}\left[\n^\Si_{\nu_A}\AS(E_{\hat{\imath}},E_{\hat{\imath}})\nu_A\right]+2\sum_{\hat{\imath}, \hat{\jmath}=1}^{n-1}\left[\AS(E_{\hat{\imath}},E_{\hat{\jmath}})\II(E_{\hat{\imath}},E_{\hat{\jmath}})\right]\\
&\qquad -\sum_{A=1}^{m}\AS(\nu_A, \II(\mu,\mu))\nu_A-\sum_{\hat{\imath}=1}^{n-1}\AS(E_{\hat{\imath}},E_{\hat{\imath}})\II(\mu,\mu).
\end{flalign*}
\end{cor}
\begin{proof}
Noting that $\II(\mu,\mu)=H-\sum_{\hat\imath}\II(E_{\hat{\imath}},E_{\hat{\imath}})$, the result follows immediately from Lemmas \ref{boundaryderivatives} and \ref{secondbdryderivs}.
\end{proof}

\subsection{$C^0$ and gradient estimates}
We now derive height, gradient and mean curvature bounds for solutions to spacelike MCF with Neumann boundary conditions.

\begin{lemma}\label{C0EstimatesNeumann}
Suppose we have a spacelike solution to \eqref{MCFgraphNeumann} for $t\in[0,T)$ and $\Omega$ is convex.

\begin{enumerate}[(a)]
 \item Expanding quasi-spheres centred in $\Omega \times \bb{R}^m$ act as outer barriers to the flow.
 \item For all $y\in M_t$ and $1\leq A \leq m$,
\[\inf_{M_0} u^A \leq u^A(y,t) \leq \sup_{M_0} u^A.\]
\item As a graph over $x_0\in \Omega$,
\[\|\U(x_0, 0)-\U(x_0,t)\|\leq \sqrt{2nt} .\]
\end{enumerate}
\end{lemma}
\begin{proof}
Lemma \ref{boundaryz} implies that $\n_\mu(|X|^2+2nt)\geq 0$. Corollary  \ref{C0evol} and the maximum principle (see \cite[Theorem 3.1]{Stahlfirst}) therefore imply that if $|X|^2+2nt\geq -R^2$ initially, then this is preserved, giving (a). 

The claim in (b) follows similarly from Corollary \ref{C0evol} and  Lemma \ref{boundaryu}.

The final statement follows by attaching an expanding quasi-sphere of radius $0$ at   $x\in \text{graph}\, \U_0$. At $t=0$, this is exactly the lightcone. Since $\U_0$ is spacelike the lightcone cannot touch the graph anywhere except this point and so we may apply (a) to see that $M_t$ stays inside the expanding quasi-sphere. This implies (c).
\end{proof}

\begin{lemma}\label{EstimatesNeumann}
 Suppose that $\Omega$ is convex. Then for all $t\in[0,T)$ and $y\in M_t$,
 \begin{flalign*}
 w_A^2(y,t)&\leq \sup_{M_0} w_A^2,\qquad
 v^2(y,t)\leq \sup_{ M_0} v^2,\qquad
  \|H\|(y,t)\leq \sup_{M_0} \|H\|^2.
 \end{flalign*}
\end{lemma}
\begin{proof}
 Convexity implies that $\AS$ is nonnegative definite and so Corollary \ref{boundaryident} yields
 \[\n_\mu w_A^2\leq 0, \qquad \n_\mu v^2\leq 0, \qquad \n_\mu \|H\|^2\leq 0.\]
 The maximum principle (see \cite[Theorem 3.1]{Stahlfirst}), Lemmas \ref{evolgoodwk} and \ref{evolcurv} then give the estimates.
\end{proof}

\subsection{Curvature estimates}
The main difficulty in estimating $\|\II\|^2$ for MCF with a Neumann boundary condition is that we cannot obtain useful estimates on $\n_\mu\|\II\|^2$ directly (in general), as the boundary derivatives we studied above give us no information on $\n_\mu \II(\mu,E_{\hat{\imath}})$. However, since we have a uniform gradient estimate we may use methods similar to those of Edelen \cite{Edelen} to obtain curvature estimates. The idea is to perturb $\II$ so that $\mu$ is an eigenvector, and then get estimates on the perturbed second fundamental form.  This is the core of the technical work in this Neumann problem. 

The curvature estimates will rely on our earlier gradient and mean curvature estimates, and thoughout this subsection we will assume that
\begin{equation}
 v^2\leq C_v, \qquad\qquad \|H\|^2\leq C_H \ . \label{NeumannCurvassump}
\end{equation}

We take smooth uniformly bounded extensions of the tensor $\AS$ and the vector $\mu$ to $\bb{R}^{n,m}$ which, by abuse of notation we will also write as $\AS$ and $\mu$ respectively. For simplicity we also assume that at $\Sigma$, 
\[\ov \n_\mu \mu=0, \qquad \ov \n_\mu \AS = 0.\]
Such smooth extensions exist, see Lemma \ref{extensionexists}. By assumption we may estimate using \eqref{Generaltensorest} and \eqref{NeumannCurvassump} to obtain that on $M_t$, in normal coordinates we have that there exists a constant $C$ depending only on the extension $\AS$ and $C_v$ such that
\[|\ov\n_{k}\ov\n_{i} \AS(X_j, \nu_A)|\leq C\ .\]
Similar estimates hold for all other derivatives of $\AS$ and $\mu$, and this will be used liberally in the following lemmas. 

We now define the $NM$-valued tensor
\[\ov \II(U,V) = \II(U,V)+ c\ip{U}V e_1^\perp+\sum_{A=1}^m\left[\ov{T}(U,V, \nu_A)\right]\nu_A  ,\]
for some $c>0$ to be defined later, where 
\[\ov{T}(U,V,W) = \AS(U,W)\ip{V}{\mu}+\AS(V,W)\ip{U}{\mu}.\]
We note that due to our assumptions on the extensions of $\AS$ and $\mu$,  on $\Sigma$, $\ov\n_\mu T(\cdot,\cdot,\cdot)=0$.

One key property of $\ov \II$ is that, by Lemma \ref{boundaryderivatives}, at the boundary we have that
\[\ov\II(\mu,E_{\hat{\imath}}) = \II(E_{\hat{\imath}},\mu)+\sum_{A=1}^m\AS(E_{\hat{\imath}}, \nu_A)\nu_A=0\]
holds for all $E_{\hat{\imath}}$; that is, $\mu$ is an eigenvector of $\ip{\ov \II (U,V)}{\nu_A}$ for all $A\in\{1,\ldots,m\}$. A second important property is that we may choose $c$, depending on our bounds on $v^2$, $\|H\|^2$ and the tensor $\ov T$, to be sufficiently large so that 
\[\|\ov H\|^2 := \|g^{ij} \ov \II_{ij}\|^2 \geq n .\]
This implies an important lower bound, namely
\begin{equation}\|\ov\II\|^2 \geq 1  .
 \label{LbII}
\end{equation}
From now on we assume $c=c(C_v, C_H, \ov{T})>0$ is sufficiently large so that (\ref{LbII}) holds. As a result, we also have that there exists a constant $C=C(C_v,C_H, \ov{T})$ such that
\begin{equation}\|\II\|^2\leq C(\|\ov\II\|^2+1)\leq 2C\|\ov\II\|^2\ .\label{overlinecurvineq}
\end{equation}

We now estimate the boundary derivative of the size of the perturbed curvature tensor $\ov\II$ in this Neumann setting.
\begin{lemma}\label{perturbedcurvatureatboundary}
 Suppose we have a solution of \eqref{MCFgraphNeumann} satisfying \eqref{NeumannCurvassump}. Then, there exists a constant $\kappa>0$ depending only on $n$, $m$, $\AS$, $\n^\Si \AS$, $\ov{T}$, $C_v$ and $C_H$ 
 such that at any point $p\in\partial M_t$, 
  \[\n_\mu \|\ov\II\|^2 \leq \kappa\|\ov \II\|^2 .\]
\end{lemma}
\begin{proof}
 We see that (using Lemma \ref{Generalnormaloneform})
 \begin{flalign*}
 \np_\mu &\ov \II(E_{\hat{\imath}},E_{\hat{\jmath}})\\ &= \np_\mu \II(E_{\hat{\imath}},E_{\hat{\jmath}})+c\delta_{\hat{\imath}\hat{\jmath}}\II(e_1^\top, \mu)\\
 &\quad
 +\sum_{I=1}^m\left[\ov\n_\mu \ov T(E_{\hat{\imath}},E_{\hat{\jmath}}, \nu_I) + \ov T(\II(\mu, E_{\hat{\imath}}), E_{\hat{\jmath}}, \nu_I) \right.\\
 &\qquad\qquad\qquad\left.+\ov T(E_{\hat{\imath}},\II(\mu ,E_{\hat{\jmath}}), \nu_I))-\ov{T}(E_{\hat\imath}, E_{\hat{\jmath}}, X^j)\ip{\II(X_j, \mu)}{\nu_I}\right]\nu_I\\
 &=\np_\mu \II(E_{\hat{\imath}},E_{\hat{\jmath}})+c\delta_{\hat{\imath}\hat{\jmath}}\II(e_1^\top, \mu)\\
 &\quad +\sum_{I=1}^m\left[ \ov T(\II(\mu, E_{\hat{\imath}}), E_{\hat{\jmath}}, \nu_I) +\ov T(E_{\hat{\imath}},\II(\mu ,E_{\hat{\jmath}}), \nu_I))\right]\nu_I\ .
 \end{flalign*}
Using the inequalities \eqref{NeumannCurvassump}, \eqref{LbII} and \eqref{overlinecurvineq} and Lemma \ref{secondbdryderivs}, we may therefore estimate that
 \[  \|\np_\mu \ov \II(E_{\hat{\imath}},E_{\hat{\jmath}})\| \leq \kappa_1\|\ov \II\|\ ,\]
where $\kappa_1$ depends on $n$, $m$, $\AS$, $\n^\Si \AS$, $\ov{T}$, $C_v$ and $C_H$.
Similarly, 
 \begin{flalign*}
 \np_\mu &\ov \II(\mu,\mu) \\&= \np_\mu \II(\mu,\mu)+c\II(e_1^\top, \mu)\\
 &+\sum_{I=1}^m\left[\ov T(\II(\mu, \mu), \mu, \nu_I) +\ov T(\mu,\II(\mu ,\mu), \nu_I))-\ov T(\mu, \mu, X^i)\ip{\II(\mu,X_i)}{\nu_I}\right]\nu_I\ ,
 \end{flalign*} 
 and again we may estimate
  \[  \|\np_\mu \ov \II(\mu,\mu)\| \leq \kappa_2\|\ov \II\|\ ,\]
 where again $\kappa_2$ depends on $n$, $m$, $\AS$, $\n^\Si \AS$, $\ov{T}$, $C_v$ and $C_H$.
Hence, due to the eigenvector property of $\mu$ in $\ov\II$, we see that
 \begin{flalign*}
 \n_\mu \|\ov\II\|^2 &= -2\ip{\ov\II_{ij}}{\n^\perp_\mu \ov\II^{ij}} \\
 &= -2\ip{\ov\II(\mu, \mu)}{\n^\perp_\mu \ov \II(\mu,\mu)}-2\sum_{\hat\imath,\hat\jmath=1}^{n-1}\ip{\ov \II(E_{\hat{\imath}},E_{\hat{\jmath}})}{\n^\perp_\mu \ov\II(E_{\hat{\imath}},E_{\hat{\jmath}})}\leq \kappa\|\ov \II\|^2
 \end{flalign*}
 as required.
\end{proof}

We now estimate the evolution of $\|\ov\II\|^2$.
\begin{lemma}
Suppose we have a solution of \eqref{MCFgraphNeumann} satisfying \eqref{NeumannCurvassump}. Then there exists a constant $C>0$ depending only on $n$, $m$, $\AS$, $\n^\Si \AS$, $\mu$, $\ov\n \mu$,  $\ov T$, $\ov\n \,\ov T$, $C_v$ and $C_H$
  such that
 \begin{flalign*}
 \ho \|\ov\II\|^2 &\leq -\frac{2} m \|\ov \II\|^4 +C\|\ov \II\|^3.
\end{flalign*}
\end{lemma}
\begin{proof}
We write $C$ for any  constant depending only on the quantities in the statement of the lemma, where $C$ is allowed to change from line to line. We write $\ov \II = \II + S$ and note that
 \[\|\ov\II\|^2 = \|\II\|^2 -2h_{ij}^AS_A^{ij}+|S|^2,\]
where $S_{ij A} = \ov{T}(X_i, X_j, \nu_A)+g_{ij}\ip{e_1}{\nu_A}$. We have that 
\begin{flalign*}
\ov{T}(X_i, X_j, \nu_A) &= \sum_{I=1}^m \left[\AS(X_i, e_I)\ip{X_j}{\mu}+\AS(X_j, e_I)\ip{X_i}{\mu}\right]\ip{e_I}{\nu_A}\\
&\qquad+\sum_{j=1}^m \left[\AS(X_i, f_k)\ip{X_j}{\mu}+\AS(X_j, f_k)\ip{X_i}{\mu}\right]\ip{f_k}{\nu_A}\ .
\end{flalign*}
As we have written $\ov{T}$ (and therefore $S$) as a concatenation of tensors of the forms considered in Lemma \ref{Generaltensorevol} (for the square brackets) and Lemma \ref{Generalnormaloneform} (for the normal inner product), we may see that there exists a $C$ such that (in orthonormal coordinates),
\[\left|\ho S_{ijA} \right| \leq C(\|\II\|^2 +1)\leq C\|\ov\II\|^2\ .\]
Since the same Lemmas also imply that $|\n_kS_{ijA}|\leq C\|\ov\II\|$, we have that
\[\ho |S|^2\leq C\|\ov\II\|^2\ .\]

Similarly we see that using Lemma \ref{evolcurv} 
\begin{flalign*}
\ho (-h_{ij}^AS_{A}^{ij})&=-S_{A}^{ij}\ho h_{ij}^A -h_{ij}^A\ho S_{A}^{ij}+2\n_k h_{ij}^A\n^kS_{A}^{ij}\\
&\leq C\|\ov\II\|^3 +2\n_k h_{ij}^A\n^kS_{A}^{ij}\ .
\end{flalign*}
Therefore,
\begin{flalign*}
 \ho \|\ov\II\|^2 &\leq -\frac{2} m \|\II\|^4 - 2\|\np \II\|^2 +4\n_kh_{ij}^A\n^k\ov S^{ij}_A+C\|\ov \II\|^3\\
 &\leq -\frac{2} m \|\ov \II\|^4 +C\|\ov \II\|^3,
\end{flalign*}
where we used Young's inequality to estimate the third term on the right hand side of the first line.
\end{proof}

Putting all of the results of this subsection together provides the following curvature estimate.
\begin{proposition}\label{curvestNeumann}
 Suppose we have a solution of \eqref{MCFgraphNeumann} satisfying \eqref{NeumannCurvassump}. Then there exists a constant $C$ depending on $n$, $m$, $\AS$, $\n^\Si \AS$, $\mu$, $\ov\n \mu$,  $\ov{T}$, $\ov\n \, \ov{T}$, $M_0$, $C_v$ and $C_H$ 
 such that
 \[\|\II\|^2 \leq C.\]
\end{proposition}
\begin{proof}
Let $C$ be as in the previous Lemma. We take the standard Euclidean distance to $\partial \Omega$ in $\bb{R}^n$ and extend it to $\bb{R}^{n,m}$ by pullback under the standard projection. We call this function $d$ and see that due to the gradient estimate for all $t\in[0,T)$
 \[\ho d \leq C \text{ on } M_t\qquad \text{ and }\qquad\n_\mu d = -1 \text{ on }\partial M_t\ .\]
 Lemma \ref{perturbedcurvatureatboundary} implies that the function $f=\|\ov\II\|^2e^{\l d}$ satisfies
 \[\n_\mu f = (\kappa - \l)f \]
on $\partial M_t$. We choose $\l$ sufficiently large so that $\n_\mu f$
is negative, meaning that no boundary maxima may occur. At any interior increasing maximum,
 \begin{flalign*}
  0&\leq \ho f\\& \leq e^{\l d}\left[-\frac 2 m \|\ov \II\|^4 +C\|\ov \II\|^3 -2\l\ip{\n\|\ov \II\|^2 }{\n d} + C\l\|\ov \II\|^2-\l^2|\n d|^2\|\ov \II\|^2 \right]\\ 
  &=e^{\l d}\left[-\frac 2 m \|\ov \II\|^4 +C\|\ov \II\|^3 + C\l\|\ov \II\|^2+\l^2|\n d|^2\|\ov \II\|^2 \right]\\
  &\leq e^{\l d}\left[-\frac 2 m \|\ov \II\|^4 +C\|\ov \II\|^3  \right].
 \end{flalign*}
Hence at any increasing stationary point, $\|\ov\II\|^2$ is bounded, and so $f$ is bounded. The maximum principle indicates that $f$ is therefore bounded (as $d$ is bounded) and hence $\|\ov\II\|^2$ is bounded everywhere. The result now follows.
\end{proof}

\begin{remark}
 The above proof holds for much more general boundary manifolds $\Sigma$. In fact, for any mean curvature flow in $\bb{R}^{n,m}$ with a perpendicular boundary condition on a smooth manifold $\Sigma$, an identical proof will show that gradient and mean curvature estimates imply full boundary curvature estimates. This therefore replaces the missing Nash--Moser--De Giorgi estimates for this parabolic system.
\end{remark}

\subsection{Convergence} We now complete the proof of Theorem \ref{Neumanntheorem} by proving convergence of spacelike MCF under Neumann boundary conditions. 
\begin{lemma}\label{ConvergenceNeumann}
 Suppose that we have a solution $\U$ to \eqref{MCFgraphNeumann} for $T=\infty$ with uniform $C^{k;\frac{k}{2}}(\Omega\times[0,\infty))$ estimates for all $k\geq0$. Then $\U$ converges uniformly in $C^\infty$ to a constant function as $t \ra \infty$.
\end{lemma}
\begin{proof}
 The proof is in fact identical to the codimension one case \cite{Blatentselfreference}, which we include here for the convenience of the reader.
 
By Lemma \ref{C0EstimatesNeumann} we have that $\|\U\|^2$ is uniformly bounded for all time by its initial value. By considering the metric in terms of the graph function $\U$, we see that 
 \[\int_{M_t} dV = \int_\Omega \sqrt{\det g_{ij}}dx \leq |\Omega|.\]
 We therefore see that since 
 \[\ddt{}\int_{M_t}dV = \int_{M_t} \|H\|^2dV\]
 (which follows from \cite[equation (4.1)]{LiSalavessa} and the Neumann boundary condition) we have that
 \[\int_0^T \int_{M_t}\|H\|^2  \leq |\Omega| -|M_0|<|\Omega|\ .\]
By Corollary \ref{C0evol}, Lemma \ref{boundaryu}, the $L^2$ estimate on $H$, and divergence theorem we see that
\begin{align*}\ddt{}\int_{M_t} u_A^2dV  &= -2\int_{M_t} |\n u_A|^2 + \int_{M_t}u_A^2\|H\|^2dV\\ &  = -2\int_{M_t} (w_A^2-1)dV+ \int_{M_t}u_A^2\|H\|^2dV 
\end{align*}
or, due to the uniform estimate on $u_A$ and $|M_t|$, and calculations in Appendix \ref{MCFasGraphs},
\[\int_0^\infty\int_{M_t}(w_A^2-1) dVdt\leq C(M_0, \Omega)  .\]
Rewriting this over $\Omega$, using the uniform gradient bound and  \eqref{eIidentity} gives
\[\int_0^\infty \int_{\Omega} |D\U_A|^2 dx dt\leq \tilde{C}(M_0, \Omega)\ .\]
The uniform $C^{k;\frac{k}{2}}$ estimates imply  $|D\U_A|\ra 0$ as $t\ra \infty$. The range of $\U_A$ is also monotonically decreasing with time due to estimates as in Lemma \ref{C0EstimatesNeumann}. Therefore, each $\U_A$ converges uniformly to a constant as $t\ra \infty$. The uniform $C^{k;\frac k 2}$ estimates and Ehrling's lemma now imply that the convergence is in fact smooth. 
\end{proof}

\section{Dirichlet boundary conditions}\label{s:Dirichlet}

In this section we wish to consider evolving a topological disk by spacelike mean curvature flow, where the boundary is held on some fixed $(n-1)$-dimensional spacelike submanifold of $\bb{R}^{n,m}$. 

We state this boundary condition graphically. Suppose that $\Omega\subset\bb{R}^{n}$ is a compact domain with smooth boundary $\partial \Omega$ .  We denote the outward unit normal to $\partial \Omega$ by $\mu$. The boundary data for the Dirichlet problem is given by smooth functions $\phi:\ov \Omega \ra \bb{R}^m$. 

To ensure that the Dirichlet problem is well-posed we require a constraint on our choice of boundary data as follows.  

\begin{defses}
We say that $\phi:\ov\Omega\ra\bb{R}^m$ is \emph{acausal} if for all $x, y\in \partial \Omega$
\begin{equation}\|\phi(x)-\phi(y)\|\leq |x-y|.
 \label{acausal}
\end{equation}
Clearly, in a convex domain this is a necessary condition if we are to have a spacelike graph, due to the mean value theorem. Due to compactness, the boundary data is in fact \emph{strictly acausal}: there exists a $\delta>0$ such that \[\|\phi(x)-\phi(y)\|\leq (1-\d)|x-y|\ .\] 
\end{defses}

\noindent The acausal condition for the Dirichlet problem for maximal spacelike submanifolds in $\bb{R}^{n,m}$ arises in the recent work of Yang Li \cite{YangLi}.  We shall assume that our chosen Dirichlet data $\phi$ is acausal.

Mean curvature flow with a Dirichlet boundary condition starting at an initial graph $\U_0$ is now defined by $\U:\Omega\times[0,T)\ra \bb{R}^{m}$ 
where 
\begin{equation}
 \begin{cases}
  \displaystyle\ddt{\U} -g^{ij}(D\U)D_{ij}\U =0 &\text{ on } \Omega\times[0,T),\\[4pt]
  \U =\phi &\text{ on } \partial \Omega\times[0,T),\\
  \U(\cdot, 0)=\U_0(\cdot) &\text{ on } \Omega .
 \end{cases}
 \label{MCFgraphDirichlet}
\end{equation}
As previously, we define $M_t:= \text{graph} \,\U(\cdot,t)$. As in the Neumann case, to have higher order regularity initially, we require some assumptions on $\U_0$.

The $\text{zero}^\text{th}$ order compatibility condition is defined to be that 
\[\U_0(x) = \phi(x) \qquad \text{for all } x \in \partial \Omega\]
We define the $k^\text{th}$ order compatibility condition is given by
\[\frac{d^k}{dt^k}\U_0(x) =0  \qquad \text{for all } x \in \partial \Omega,\]
where $\frac{d}{dt}\U_0$ is defined recursively by the first line of (\ref{MCFgraphDirichlet}).

The key difficulty in proving long time existence for \ref{MCFgraphDirichlet} is in obtaining suitable boundary gradient estimates. Fortunately, Yang Li \cite{YangLi} has recently produced suitable barriers for the Dirichlet problem for the higher codimensional maximal submanifold system (i.e.~the elliptic equivalent of (\ref{MCFgraphDirichlet})): these are the higher codimensional equivalents of the barriers in \cite{BartnikSimon}. We will show below that (unsurprisingly) these also act as barriers to (\ref{MCFgraphDirichlet}). 

We now state our long-time existence and convergence theorem in the Dirichlet setting.  The proof is again quite long and technical, and forms the remainder of the section.  We shall, as in the Neumann case, give an outline of the proof where we assume the key technical results proved below.  We again recall the definitions for an expanding quasi-sphere to be an outer barrier and of uniformly spacelike, from Definitions \ref{quasisphere} and \ref{def:uspacelike} respectively.

 \begin{theorem}\label{Dirichlettheorem}
 Suppose $\Omega$ is a bounded domain with smooth boundary $\partial \Omega$, $\phi$ is acausal boundary data and $\U_0:\Omega\ra\bb{R}^m$ is uniformly spacelike satisfying compatibility conditions to the $l^\text{th}$ order for some $l\geq 0$. There exists a solution $$\U\in C^{2l+\alpha;\frac{2l+\alpha}{2}}(\Omega \times [0,\infty))\cap C^\infty(\Omega \times (0,\infty))$$ of \eqref{MCFgraphDirichlet} which is unique if $l\geq1$ and converges smoothly to the unique maximal submanifold with boundary data $\phi$ as $t\ra \infty$. Furthermore, expanding quasi-spheres centred in $\Omega\times \bb{R}^m$ act as barriers to the flow and we have the uniform bounds 
 \begin{flalign*}
  |\U_0(x)-\U(x,t)| &\leq \sqrt{2nt}\quad\text{and}\quad
  \|u\|^2\leq \sup_{M_0} \|u\|^2,
 \end{flalign*}
 and, if $l\geq 1$,
\[    \|H\|^2 \leq \frac{1}{(\sup_{M_0}\|H\|^2)^{-1}+\frac{2}{n}t}.\]
\end{theorem}
\begin{proof}
The system (\ref{MCFgraphDirichlet}) is in the form of $m$ parabolic PDEs with linear boundary conditions. This implies that short time existence (the existence of $T>0$  such that there is a solution $\U\in C^{l+1+\a;\frac{l+1+\a}{2}}(\Omega\times[0,T))\cap C^\infty(\Omega\times(0,T))$ to (\ref{MCFgraphDirichlet})) follows from a standard application of fixed point theory and Schauder estimates for parabolic PDEs, for example by very minor modifications of \cite[Theorem 8.2]{Lieberman}. 
 
 As stated in Appendix \ref{MCFasGraphs}, each component $\U^A$ of $\U$ satisfies a uniformly parabolic PDE (given by the first line of (\ref{MCFgraphDirichlet})) if and only if $v^2$ is bounded. Furthermore, we may apply standard Schauder estimates as soon as we know that our solution satisfies $\U^A\in C^{1+\alpha;\frac{1+\alpha}{2}}(\Omega\times [0,T))$ \emph{for all }$A$. 
 
 In Lemma \ref{C0EstimatesDirichlet} we demonstrate uniform $C^0$ estimates for solutions to (\ref{MCFgraphDirichlet}). In Proposition \ref{C1Dirichlet} we give uniform estimates on $v^2$, which imply both uniform parabolicity and $C^1$ estimates on $\U$. Lemma \ref{C1alpha} then implies that we have uniform estimates in $C^{1+\a;\frac{1+\a}{2}}$. As Schauder estimates now apply, by bootstrapping we have uniform higher order estimates and the long time existence claimed.
 
 Lemma \ref{DirichletConvergence} finally implies that the solution converges smoothly to the unique maximal submanifold with boundary data $\phi$.  The fact that the maximal submanifold is unique is a consequence of \cite[Theorem 2.1]{YangLi}. Uniqueness of the flow solution is proven in Proposition \ref{Uniqueness}.
\end{proof}

\begin{remark}
The existence and uniqueness of a solution to the Dirichlet problem for maximal submanifolds in $\bb{R}^{n,m}$ with acausal boundary data is given in \cite[Theorem 2.1]{YangLi}.  Theorem \ref{Dirichlettheorem} can be viewed as an extension of this result.
\end{remark}

Throughout this section we will write a point $(x,y)\in \bb{R}^{n,m}=\bb{R}^{n}\oplus\bb{R}^{m}$. For any two vectors $y,z \in \bb{R}^m$ we will write the inner product associated to the norm $\|\cdot\|$ as $y\cdot z$: this is just the standard Euclidean inner product on $\bb{R}^m$.

\subsection{$C^0$ estimates} We first derive some simple $C^0$ bounds on solutions to spacelike MCF with Dirichlet boundary conditions, just as in the Neumann case.
\begin{lemma}\label{C0EstimatesDirichlet}
 For any spacelike solution of \eqref{MCFgraphDirichlet} we have the following.
 \begin{enumerate}[(a)]
  \item Expanding quasi-spheres centred in $\Omega \times\bb{R}^m$ act as outer barriers.
  \item For all $(x,t)\in\Omega\times[0,T)$,  \[|\U(x,t)|\leq \sup_{y\in\Omega} |\U_0(y)|.\]
  \item For all $x\in\Omega$,
  \[\|\U_0(x) - \U(x,t)\|\leq \sqrt{2nt}.\]
 \end{enumerate}
\end{lemma}
\begin{proof}
 Let $p\in \bb{R}^{n,m}$. Suppose that for all $y\in M_0$, $|y-p|^2\geq -R^2$. Clearly, as the boundary  of $M_t$ is fixed, this implies that for all $z\in \partial M_t$, $|z-p|^2\geq -R^2\geq -R^2-2nt$.  The weak maximum principle applied to $f=|X-p|^2+2nt$ and Corollary \ref{C0evol} now imply (a).

 Similarly, (b) follows from Corollary \ref{C0evol} and the weak maximum principle.
 
Part (c) follows from considering an expanding quasi-sphere starting from a light cone centred at $(x,\U(x))$.
\end{proof}

\subsection{$C^1$ estimates} Our goal now is to obtain bounds on the gradient and mean curvature of solutions to spacelike MCF with Dirichlet boundary conditions.  This forms the main technical work required in this Dirichlet problem.

We first recall the barriers constructed in \cite{YangLi}. We consider a 2-parameter family of curves $\Gamma_{K,\Lambda}\subset \bb{R}^{1,1}\subset\bb{R}^{n,m}$ from which we will produce a hypersurface $\tilde{\Gamma}_{K,\Lambda}\subset\bb{R}^{n,m}$ by assuming  an $SO(n-1)\times SO(m-1)$ symmetry. 

\begin{defses}\label{def:barriers}
Let $K>0$ and let $\Lambda\leq 0$.  
Taking orthogonal coordinates $r$, $w$ of $\bb{R}^{1,1}$ (where $\pard{}{r}$ is spacelike) we write $\Gamma_{K,\Lambda}$ graphically as
\[\Gamma_{K, \Lambda}:=\{(r, w)\in \bb{R}^{1,1}|w=f_{K,\Lambda}(|r|)\} ,\]
where 
\[f_{K,\Lambda}(r) = \int_0^r \frac{K+n^{-1} \Lambda t^n}{\sqrt{t^{2n-2}+(K+n^{-1}\Lambda t^n)^2}} dt .\]
Let $\xi\in\bb{R}^n$ and $\eta\in\bb{R}^m$. We define the functions 
\[r(x,y):=|x-\xi|, \qquad w(x,y):=|y-\eta|\]
for $(x,y)\in\bb{R}^{n,m}$.
We then define the \emph{barrier hypersurface} $\tilde{\Gamma}_{K,\Lambda}$, centred at $(\xi,\eta)$ by 
\[\tilde{\Gamma}_{K,\Lambda}:= \{(x,y)\in \bb{R}^{n,m}| w(x,y) = f_{K,\Lambda}\big(r(x,y)\big)\}.\]
When $f'_{K,\Lambda}<1$ we write the unit normal to $\tilde \Gamma_{K,\Lambda}$ at $(x,y)$ by
\[\tilde{n}(x,y) = \frac{1}{\sqrt{1-(f'_{K,\Lambda})^2}}\left(-\ov\n w+f'_{K,\Lambda}\ov\n r \right).\]
\end{defses}

Several observations in \cite[Section 3.1]{YangLi} will be of use to us.
 We note that
 \[f'_{K,\Lambda} = \frac{Kr^{1-n}+\Lambda n^{-1}r}{\sqrt{1+(Kr^{1-n}+\Lambda n^{-1}r)^2}} \quad\text{and}\quad \frac{f'_{K,\Lambda}}{\sqrt{1-(f'_{K,\Lambda})^2}} = \Lambda n^{-1} r + Kr^{1-n}  .\]
We will therefore always assume that 
\[0<r<\left(\frac{nK}{|\Lambda|}\right)^\frac{1}{n}  .\]
 Within this range $\tilde{n}$ is timelike, $\tilde{\Gamma}_{K,\Lambda}$ has a nondegenerate semi-Riemannian metric, and $\Gamma_{K,\Lambda}$ is a spacelike curve. As $r\ra 0$, both $\Gamma_{K,\Lambda}$ and $\tilde{\Gamma}_{K,\Lambda}$ are tangent to the lightcone. 

 We may estimate that if $K=\e^{-1}$, $\Lambda<0$ such that $\e<\left(\frac{n}{2|\Lambda|}\right)^\frac 1 {n+1}$ then 
 \begin{flalign}f_{K,\Lambda}(\e)&=\int_0^\e\sqrt{1-\frac{t^{2n-2}}{t^{2n-2}+(K-n^{-1}|\Lambda| t^n)^2}}dt
 >\e\sqrt{1-4\e^{2n}}.
  \label{fest}
 \end{flalign}

 At any point $p\in \tilde{\Gamma}_{K,\Lambda}\cap\mathcal{C}_{\left(\frac{nK}{|\Lambda|}\right)^\frac{1}{n}}(\xi)$ and any $n$-dimensional spacelike hyperplane $\Pi\subset T_p\tilde \Gamma_{K,\Lambda}$, we define
 \[H_\Pi := -\sum_{i}\ip{b_i}{\ov\n_{b_i}\tilde{n}}\]
 where $b_1, \ldots, b_n$ is an orthonormal basis of $\Pi$. The following observation, proven in \cite[Lemma 3.1]{YangLi}, will be vital in demonstrating that the $\tilde{\Gamma}_{K,\Lambda}$ are barriers.
 \begin{lemma}
Let $p\in \tilde{\Gamma}_{K,\Lambda}\cap\mathcal{C}_{\left(\frac{nK}{|\Lambda|}\right)^\frac{1}{n}}(\xi)$ and $\Pi\subset T_p\tilde \Gamma_{K,\Lambda}$ be an $n$-dimensional spacelike hyperplane. Then
\[H_\Pi\geq -\Lambda\ .\]
\end{lemma}
The following demonstrates that the solutions $\tilde{\Gamma}_{K,\Lambda}$ act as barriers and is a parabolic version of \cite[Lemma 3.2]{YangLi}. 
\begin{lemma}\label{Staybarriers}
 Suppose that $M_t$ is a spacelike solution to \eqref{MCFgraphDirichlet} and $\xi\in\bb{R}^n\setminus \ov\Omega$, $\eta\in\bb{R}^m$,  $K>0$ and $\Lambda<0$ are chosen such that 
 \[r(z)<\left(\frac{nK}{|\Lambda|}\right)^\frac 1 n \text{ for all } z\in \Omega \text{ and } M_0\subset\big\{(x,y)\in\bb{R}^{n,m}|w(x,y)\leq f_{K,\Lambda}\big(r(x,y)\big)\big\}. \]
 For all $t\in[0,T)$,  
 \[M_t\subset\big\{(x,y)\in\bb{R}^{n,m}|w(x,y)\leq f_{K,\Lambda}\big(r(x,y)\big)\big\}.\]
\end{lemma}
\begin{proof}
Let $\tilde{K}>K$, and observe that 
\begin{align*}
 M_0&\subset\big\{(x,y)\in\bb{R}^{n,m}|w(x,y)\leq f_{K,\Lambda}\big(r(x,y)\big)\big\}\\
&\subset \big\{(x,y)\in\bb{R}^{n,m}|w(x,y)< f_{\tilde{K},\Lambda}\big(r(x,y)\big)\big\} .
\end{align*}
We consider the function \[h:=w-f_{\tilde{K},\Lambda}.\] Clearly this is negative on $M_0$. 
We suppose that $t_0$ is the first time when there exists $p_0\in M_{t_0}$ such that $h(p_0)=0$.  As $\partial M_t=\partial M_0$ for all $t$ and 
$h$ is negative on $\ov{M_0}$, $p_0$ cannot be a boundary point. Furthermore, $p_0\in \tilde{\Gamma}_{\tilde{K},\Lambda}$ and $\n h(p_0)=0$, so $T_{p_0}M_{t_0}\subset T_{p_0}\tilde{\Gamma}_{\tilde{K},\Lambda}$. 

Let $b_1, \ldots b_n$ be an orthonormal basis of $T_{p_0}M_{t_0}$. Since $\ov \n h =- \sqrt{1-(f'_{\tilde{K},\Lambda})^2}\tilde{n}$, 
\[g^{ij}\ov\n^2_{ij}h = -\ip{\ov\n_{b_i}(\sqrt{1-(f'_{\tilde{K},\Lambda})^2}\tilde{n})}{b_i} = \sqrt{1-(f'_{\tilde{K},\Lambda})^2}H_{T_{p_0}M_{t_0}} .\]

As $p_0$ is a nondecreasing interior maximum, Lemma \ref{GeneralC0evol} implies
\[0\leq\ho h (p_0) = -g^{ij}\ov\n^2_{ij}h = -\sqrt{1-(f'_{\tilde{K},\Lambda})^2}H_{T_{p_0}M_{t_0}}\leq \sqrt{1-(f'_{\tilde{K},\Lambda})^2}\Lambda .\]
The assumption $\Lambda<0$ yields a contradiction. Therefore, 
\[ M_t\subset \big\{(x,y)\in\bb{R}^{n,m}|w(x,y)< f_{\tilde{K},\Lambda}\big(r(x,y)\big)\big\}\]
for all $t\in[0,T)$ and all $\tilde{K}>K$. Limiting $\tilde{K}$ to $K$ yields the statement.
\end{proof}

We now demonstrate that suitable barriers may be attached to $\partial M_0$.
\begin{lemma}\label{AttachBarriers}
Let $\U_0$ be smooth uniformly spacelike initial data on $\Omega$ with acausal boundary values. Then for any $\hat{x}\in \partial \Omega$, $\theta\in\bb{R}^m$ and $\Lambda<0$ there exists $\xi\in \bb{R}^n$, $\eta\in \bb{R}^m$, $K>0$ and $\d\in(0,1)$ such that the following hold.
\begin{enumerate}[(a)]
 \item  $\Omega \subset B_{\left(\frac{nK}{|\Lambda|}\right)^\frac{1}{n}}(\xi)$.\label{three}
 \item $M_0\subset\big\{(x,y)\in\bb{R}^{n,m}|w(x,y)\leq f_{K,\Lambda}\big(r(x,y)\big)\big\}$.\label{one}
 \item $(\hat{x},\U_0(\hat{x}))\in \tilde{\Gamma}_{K,\Lambda}$.\label{two}
 \item Let $\widetilde{M}:=\text{\emph{graph}} \, \tilde{u}$ for some smooth $\tilde{u}:\Omega \ra \bb{R}^m$, such that $\partial \widetilde{M} = \partial M_0$ and $\widetilde{M}\subset\big\{(x,y)\in\bb{R}^{n,m}|w(x,y)\leq f_{K,\Lambda}\big(r(x,y)\big)\big\}$. Then\label{four}
 \[D_{-\mu}\tilde{u}\cdot \t \leq 1-\d\ .\]
\end{enumerate}
Furthermore, $K$ and $\d$ can be chosen to depend only on $n$, $\Lambda$, $\Omega$, $\underset{M_0} \sup \, v$ and $|\U_0|_{C^3(\Omega)}$.
\end{lemma}
\begin{proof}
Our strategy is to find a suitable $\tilde{\Gamma}_{K,\Lambda}$which touches $M_0$ only at the point $(\hat{x}, \U_0(\hat{x}))\in\partial M_0$. We take $\e>0$ and begin by setting $K=\e^{-1}$.

\medskip

\paragraph{\textbf{Step 1:}} \emph{Pick $\tilde{\Gamma}_{K,\Lambda}$ so that (\ref{two}) holds.}
 We translate and rotate coordinates so that $\hat{x}=0$, $\U(0)=0$ and $\mu=-e_n$. Then, we rotate coordinates in $\bb{R}^{n-1}=T_0\partial \Omega$ so 
 \[D^\partial \left(\U_0\cdot\t\right)|_{0} = ae_1\]
 for some $a\in(0,1)$, where $D^\partial$ is the gradient operator on $\partial \Omega$. 
 
 We now show that we can choose a centre for $\tilde{\Gamma}_{K,\Lambda}$ so that (\ref{two}) holds and $\tilde{\Gamma}_{K,\Lambda}$ is tangent to $\partial M_0$.  Concretely, for any $\e>0$, we set 
 \[\xi = -\frac{\e}{\sqrt{1+b^2}}(b, 0,\ldots, 0,1), \qquad \eta = -f_{ \Lambda,K}(\e)\theta ,\]
 for some $b$ to be determined. We observe that for this choice, $\tilde{\Gamma}_{K,\Lambda}$ goes through the origin and so (\ref{two}) is satisfied.
 
 \medskip
 
\paragraph{\textbf{Step 2:}} \emph{Pick $\e$ so that (\ref{three}) holds.} We observe that for $\e<\e_1= \e_1(\Lambda,n,\text{diam}\, \Omega)<1$, 
 \[\text{diam}\,{\Omega}+2\leq\left(\frac{nK}{|\Lambda|}\right)^\frac{1}{n} .\]
 As $\xi$ is at most distance $\e$ from the origin, this implies that (\ref{three}) is satisfied.
 
 \medskip

\paragraph{\textbf{Step 3:}} \emph{Pick $b$ so that $\tilde{\Gamma}_{K,\Lambda}$ is tangential to $\partial M_0$ at $(\hat{x}, \U_0(\hat{x}))$.} For $\tilde{\Gamma}_{K,\Lambda}$ to be tangent to $\partial M_0$ at $0$, we require that
 \[D^\partial \big[w\big(x,\U_0(x)\big)  -f_{K,\Lambda}\big(r(x,\U_0(x)\big)\big](0)=0;\]
 that is,
 \[D^\partial\left(\U_0\cdot\t\right) - f'_{K,\Lambda}(\e)\frac{b}{\sqrt{1+b^2}}e_1=0 .\]
 We now want to choose $b=b(\e,a)$ so that
 \[f'_{K,\Lambda}(\e)\frac{b}{\sqrt{1+b^2}}=a ,\]
 which is only possible if $f'_{K,\Lambda}>a$. 
By our choice of $K$,
 \[f_{K,\Lambda}'(\e) = \frac{1+n^{-1}\Lambda \e^{1+n}}{\sqrt{\e^{2n}+(1+n^{-1}\Lambda \e^{n+1})^2}}.\]
Since $M_0$ is uniformly spacelike, we deduce that $f'_{K,\Lambda}>a$ is satisfied for all $\e<\e_2(a,n, \Lambda, \e_1)\leq\e_1$, and for such $\e$ we have that
 \[b=\frac{a}{\sqrt{(f'_{K,\Lambda}(\e))^2-a^2}} .\]

\paragraph{\textbf{Step 4:}} \emph{Show that for sufficiently small $\e>0$, (\ref{one}) holds.}
 We consider the function 
 \[g=w\big(x,\U(x)\big)-f_{K,\Lambda}\big(r\big(x,\U(x)\big)\big)= |\U_0(x)+f_{K,\Lambda}(\e)\t| - f_{K,\Lambda}(|x-\xi|)\]
 on $B_{R_1}(0)\subset\bb{R}^n\cap \ov\Omega$. By our construction so far we have that $g(0)=0$ and $D^\delta g(0)=0$, and our aim is to show that this is nonpositive everywhere. We first note that 
 \begin{flalign*}
 D_\mu g(0) = D_\mu \U_0\cdot\t + f'_{K,\Lambda}(\e)\frac{\xi\cdot \mu}{|\xi|} &\leq D_\mu \U_0\cdot\t - \frac{1}{\sqrt{1+b^2}}f'_{K,\Lambda}(\e)\\
  &=D_\mu \U_0\cdot\t -\sqrt{(f'_{K,\Lambda}(\e))^2-a^2}  .
 \end{flalign*}
As $M_0$ is spacelike at $0$, $|D (\U_0\cdot \t)|^2 = |D_\mu \U_0\cdot\t|^2+a^2<1$, and so there exists an $\e_2>\e_3=\e_3(\sup_{M_0} v, n, \Lambda)$ such that for all $\e<\e_3$, $D_\mu g(0)<0$.  

Furthermore, we may calculate that 
\begin{flalign*}
 D^2_{ij}g(0) &= \frac{D_i\U_0\cdot D_j\U_0 - D_i(\U_0\cdot\t)D_j(\U_0\cdot\t)}{f_{K,\Lambda}(\e)}+D_{ij}(\U_0\cdot\t)\\
 &\qquad-f''_{K,\Lambda}(\e)\frac{\xi_i\xi_j}{|\xi|^2}-\frac{f'_{K,\Lambda}(\e)}{\e}\left(\delta_{ij} -\frac{\xi_i\xi_j}{|\xi|^2}\right).
\end{flalign*}
We now suppose that $\e<\e_4=\e_4(\Lambda, n, \e_3)\leq \e_3$ is sufficiently small so that for $\e<\e_4$ the estimate (\ref{fest}) holds. We shall now restrict our attention to $1\leq i,j \leq n-1$; that is, to $T_0\partial \Omega$. On this range we have that
\begin{flalign*}
 D^2_{ij}g(0) &= \frac{D_i\U_0\cdot D_j\U_0 - a^2\delta_{i1}\delta_{j1}}{f_{K,\Lambda}(\e)}+D_{ij}(\U_0\cdot\t)\\
 &\qquad-\delta_{1i}\delta_{1j}f''_{K,\Lambda}(\e)\frac{b^2}{1+b^2}-\frac{f'_{K,\Lambda}(\e)}{\e}\left(\delta_{ij} -\delta_{1i}\delta_{1j}\frac{b^2}{1+b^2}\right)  .
\end{flalign*}

Again using the fact that $M_0$ is uniformly spacelike, there exists $\tau:=\tau(\underset{M_0}\sup\, v)\in[0,1)$ such that, as matrices, $D_i\U_0\cdot D_j\U_0\leq (1-\tau)\delta_{ij}$ (see Appendix \ref{MCFasGraphs}). Hence,
\begin{flalign*}
 D_i\U_0\cdot D_j\U_0& - a^2\delta_{i1}\delta_{j1}-f_{K,\Lambda}(\e)\frac{f'_{K,\Lambda}(\e)}{\e}\left(\delta_{ij} -\delta_{1i}\delta_{1j}\frac{b^2}{1+b^2}\right)\\
 &\leq (1-\tau)\delta_{ij} - a^2\delta_{i1}\delta_{j1}-\sqrt{1-2\e^{2n}}f'_{K,\Lambda}(\e)\left(\delta_{ij} -\delta_{1i}\delta_{1j}\frac{b^2}{1+b^2}\right)\\
 &=(1-\tau)\delta_{ij}+\left[\frac{\sqrt{1-2\e^{2n}}}{f'_{K,\Lambda}(\e)}-1\right]a^2\delta_{i1}\delta_{j1}-\sqrt{1-2\e^{2n}}f'_{K,\Lambda}(\e)\delta_{ij}\\
 &\leq(1-\tau)\delta_{ij}+\e^{2n}a^2\delta_{i1}\delta_{j1}-\sqrt{1-2\e^{2n}}f'_{K,\Lambda}(\e)\delta_{ij}\\
 &\leq -\frac{\tau}{2}\delta_{ij}
\end{flalign*}
for all $\e<\e_5=\e_5(\underset{M_0}\sup\, v,\e_4)\leq \e_4$. Estimating $|f''_{K,\Lambda}|<|\Lambda|+n-1$, $D_{ij}\U_0\cdot \t<C\delta_{ij}$ we finally see that for all $\e<\e_6=\e_6(n,\Lambda,|\U_0|_{C^2(\Omega)}, \e_5)\leq \e_5$ we have
\begin{flalign*}
 D^2_{ij}g(0)&\leq (C+|\Lambda|+n-1)\delta_{ij} - \frac{\tau}{2\e}\delta_{ij}<0.
\end{flalign*}

Therefore, for any $\e<\e_6$ there exists $R_1=R_1(\Lambda, n, \underset{M_0}\sup\, v, |\U_0|_{C^3})$ such that on $B_{R_1}\cap\partial \Omega$, the Hessian of $g$ is negative definite and $\n_\mu g<0$.  We deduce that $g\leq 0$.  Moreover, there exists $R_2<R_1$ depending on the same quantities such that $g\leq 0$ on $\Omega\cap B_{R_2}(0)$. 

For $\e$ small enough we have that $\text{dist}(\xi, \partial B_{R_2})>\frac{R_2}{2}$. As the gradient $f'_{K,\Lambda}$ monotonically increases as $\e\ra 0$, we may choose $\e<\e_7(\e_6, \underset{M_0}\sup\,v, R_2, \partial \Omega, \phi)\leq \e_6$ sufficiently small so that outside $B_{R_2}$, $f_{K,\Lambda}'>\max\{|Du^1|,\ldots, |Du^m|\}$. Integration now implies (\ref{one}) where we note that our estimate here also depends on $\partial \Omega$ and acausality to cross any nonconvex regions. 

\medskip

\paragraph{\textbf{Step 5:}} \emph{Show that (\ref{four}) now holds.} Suppose now that we have some other function $\tilde{u}$ such that $\tilde{u}=\U_0$ on $\partial \Omega$ and $|\tilde{u}-\eta|\leq f_{K,\Lambda}$ in $\Omega$. Then, for all $\l>0$, 
\[\frac{1}{\l}\left(|\tilde{u}(\hat{x}-\l\mu)-\eta|-|\eta|\right)\leq \frac{1}{\l}\left(f_{K,\Lambda}(\hat{x}-\l\mu)-f_{K,\Lambda}(\hat{x})\right)  .\]
Taking the limit as $\l\ra0$,
\[D_{-\mu} \tilde{u}\cdot \t = \frac{d}{d\l}\big|_{\l=0}|\tilde{u}(\hat{x}-\l\mu)-\eta|\leq f_{K,\Lambda}'(\e)\ip{\mu}{\frac{\xi}{|\xi|}}= \frac{f_{K,\Lambda}'(\e)}{1+b^2}\]
Since $f'_{K,\Lambda}(\e)<1$ depends only on $\e_7$, (\ref{four}) holds.
\end{proof}

\begin{proposition}\label{C1Dirichlet}
 Suppose we have a solution of \eqref{MCFgraphDirichlet} over a compact domain $\Omega$ with smooth uniformly spacelike initial data which is acausal at the boundary. There exists a constant $C_v$ depending only on $M_0$ such that for all $t\in[0,T)$ and $y\in M_t$,
 \[v(y,t)\leq C_v.\]
\end{proposition}
\begin{proof}
For any $\hat{x}\in\partial \Omega$  and unit vectors $\t\in \bb{R}^m$, we attach barriers to $M_0$ at $\hat x$ as constructed in Lemma \ref{AttachBarriers}. Lemma \ref{Staybarriers} implies that for all $t$ we have 
 \[M_t\subset\big\{(x,y)\in\bb{R}^{n,m}|w(x,y)\leq f_{K,\Lambda}\big(r(x,y)\big)\big\},\] 
 and so Lemma \ref{AttachBarriers} yields
 \[|D_{-\mu}u(\hat{x},t)\cdot \t| \leq 1-\d ,\]
 for all $\t$ as above. For any $v\perp \mu$, we also have that
 \[|D_{v}u(\hat{x},t)\cdot \t| \leq 1-\tilde{\d} ,\]
 for some $\tilde{\d}>0$, due to the uniform spacelikeness of $M_0$. These conditions now imply that at $\hat x$, $v<C(\delta, \tilde{\delta})$ 
  and so
 \[\sup_{y\in\partial M_t} v^2(y) \leq C_v=C_v(\d, \tilde{\d}).\]
 Applying the maximum principle (using Corollary \ref{wkevolest}) gives the result.
\end{proof}

We now observe that the above estimates give us decay for the mean curvature.
\begin{lemma}\label{HConvergenceDirichlet}
 Suppose that $\U\in C^{4;2}(\Omega\times[0,T))$ is a solution to \eqref{MCFgraphDirichlet}.  Suppose further that the estimates of Proposition \ref{C1Dirichlet} hold and
 \[\sup_{M_0} \|H\|^2 \leq C_H.\]
 For all $t\in[0,T)$, 
 \[\|H\|^2\leq \frac{1}{C_H^{-1}+\frac{2}{n}t}.\]
\end{lemma}
\begin{proof}
 Clearly at  $\hat{x}\in\partial\Omega$ we have   $\ddt \U=0$. Recall $\hat{g}_{AB}=-\sum_{C=1}^m\ip{e_A}{\nu_C}\ip{\nu_C}{e_B}$ from Appendix \ref{MCFasGraphs}, and as $M_t$ is spacelike at $\partial \Omega$, $\hat{g}_{AB}$ is invertible with inverse $\hat{g}^{AB}$. We compute
 \[0=\ddt{\U^A} = \ip{H}{e_B}\hat{g}^{BA}=H^C\ip{\nu_C}{e_B}\hat{g}^{BA}  .\]
Invertibility of $\hat{g}_{AB}$ implies $\ip{H}{e_B}=0$ for all $1\leq B\leq m$, and this in turn implies $H=0$, as $\ip{\nu_D}{e_C}=0$ for all $1\leq C\leq m$ contradicts that $\nu_D$ is timelike. As a result, on $\partial M_t$, \[\|H\|^2=0 .\]
We may now apply the maximum principle to $f=(C_H^{-1}+\frac{2}{n}t)\|H\|^2$. 
\end{proof}

\subsection{$C^{1+\a}$ estimates} We now prove the final estimates required for the long-time existence of spacelike MCF with Dirichlet boundary conditions.
\begin{lemma}\label{C1alpha}
 Suppose that we have a solution to \eqref{MCFgraphDirichlet} such that there is a uniform constant $C_v>0$ so that \begin{equation}\label{eq:C1alpha.v}
 v^2\leq C_v .
 \end{equation}
 Then for any $\e>0$, there exists a constant $C$ depending only on $\e$, $\partial \Omega$, $\U_0$ and $C_v$ such that
 \[|\U|_{C^{1+\a;\frac{1+\a}{2}}(\Omega\times[\e,T))}\leq C\]
\end{lemma}
\begin{proof}
 By Remark \ref{localcurvestremark}, we may take sufficiently small cylinders near the boundary $\partial \Omega$ to deduce there exists $C>0$, depending on $M_0$ and the maximum curvature of $\partial \Omega$, so that for all $t$ such that the flow exists, \[\sup_{M_t}\|\II\|\text{dist}(x, \partial \Omega)<C_1.\]
 The gradient bound \eqref{eq:C1alpha.v} now implies that $|D^2\U|\,\text{dist}(x, \partial \Omega)<C_2$. 
To conclude, we must now deal with boundary estimates. 

We may apply an observation of Krylov, see \cite[Lemma 7.47]{Lieberman}, to  (\ref{MCFgraphDirichlet}), which implies that for all $\l<R\leq R_0=R_0(\partial\Omega, C_v)$ we have 
 \[\text{osc}_{B_R\cap\Omega} \frac{\U^A(\hat x-\l\mu,t)-\phi^A(\hat x)}{\l}\leq C\left(\frac{R}{R_0}\right)^\alpha\]
for all $\hat x\in\partial \Omega$, where $C$ and $\alpha$ depend on $C_v$ and $n$. Interpolation estimates may now be applied, exactly as in \cite[Proposition 4.3]{YangLi} to yield the claim.
\end{proof}

\subsection{Convergence}
We finally demonstrate convergence of spacelike MCF in the Dirichlet setting, as in \cite[Theorem 4.1]{EckerEntire}, thus completing the proof of Theorem \ref{Dirichlettheorem}. We recall that \cite[Theorem 2.1]{YangLi} proves the uniqueness of maximal submanifolds with prescribed acausal boundary data.
\begin{lemma}\label{DirichletConvergence}
 Suppose $\U$ is a smooth solution of \eqref{MCFgraphDirichlet} with $T=\infty$ with uniform $C^{k;\frac{k}{2}}(\Omega\times[0,\infty))$ estimates for all $k\geq 0$, such that for all $t>0$ and $y\in M_t$
 \[v^2(y)<C_v .\]
 Then, $M_t$ converges smoothly to the unique maximal surface with boundary data given by $\partial M_0$. 
\end{lemma}
\begin{proof}
Lemma \ref{HConvergenceDirichlet} implies that
 \[\underset{M_t}\sup \|H\| \ra 0  .\]
Furthermore, Lemma \ref{C1alpha} and Schauder theory imply that we have uniform higher order estimates on $\U$. 
 
 Using the uniform $C^{k;\frac k 2}$ estimates and the Arzel\`a--Ascoli theorem, any sequence of times $t_i\ra \infty$ has a subsequence $t_{i(j)}$ such that $M_{t_{i(j)}}$ converges uniformly to a maximal submanfold $\widetilde{M}=\text{graph}\,\tilde{u}$ for some $\tilde{u}:\Omega \ra\bb{R}^m$.  \cite[Theorem 2.1]{YangLi} states that this limit is the unique maximal submanifold with the given boundary data
. Therefore, 
$\U(x,t)\ra \tilde{u}(x)$ uniformly as $t\ra \infty$ as otherwise we may construct a sequence of times contradicting subsequential convergence to $\tilde{u}$. Smooth convergence now follows using higher order regularity and Ehrling's Lemma. 
\end{proof}

\section{Global properties of entire solutions}\label{s:sensible}

In this section we consider graphs $M_0$ over $\bb{R}^n$ such that for some $C_v, C_H>0$
 \begin{equation}
 \qquad \underset{M_0}\sup \,v^2\leq C_v \quad \text{ and } \quad \underset{M_0}\sup \,\|H\|^2<C_H\ . 
 \label{safetyassume} 
 \end{equation}
Although we have constructed a solution for such initial data, since $M_0$ is noncompact we do not know that the solution from Theorem \ref{EntireExistence} is the only solution and it is plausible that quite wild behaviour is possible in general.   

To deal with this issue, we introduce the following natural class of entire solutions to spacelike MCF.

\begin{defses}\label{def:sensible} An entire solution of \eqref{EntireMCF} with initial data satisfying (\ref{safetyassume}) will be called \emph{\sensible} 
if there exists a continuous real-valued function $f=f(t)$ such that $f(0)=C_H$ and  \[\underset{M_t}\sup \, \|H\|^2\leq f(t).\]
\end{defses}  
\noindent We note {\sensible} solutions always exist due to Theorem \ref{EntireExistence}. We will show that {\sensible} solutions satisfy estimates which are similar to expander solutions.

First we demonstrate a noncompact maximum principle under the assumption of a uniform gradient bound. 
\begin{proposition}\label{GradMaxPrinc}
 Suppose that for all $t\in[0,T)$, $v^2\leq C_v$ uniformly on $M_t$. Suppose that $f\in C^\infty_\text{loc}(M^n\times[0,T))$ is a smooth function such that $f\geq 0$, \[C_f:=\underset{M_0} \sup f<\infty ,\] and there exists $\d>0$ such that
 \[\ho f \leq -\delta f^2 .\]
 Then for all $t\in[0,T)$ and $y\in M_t$ 
  \[f(y,t)\leq \frac 1 {C_f^{-1}+\d t} .\]
\end{proposition}
\begin{proof}
Let $\varphi_R:\bb{R}\ra[0,\infty)$ be a smooth cutoff function such that:
\begin{itemize}
 \item $|\varphi_R(x)|\leq 1$ and $\varphi_R(x)=1$ on $(-\infty,1]$, $\varphi_R(x)=0$ on $[1+R,\infty)$;
 \item $|\varphi_R'|\leq \frac{2}{R}$, $|\varphi_R''|<\frac{10}{R^2}$.
\end{itemize}
It is easy to see that such a cutoff function exists by considering cubic polynomials. 

For $r$ as in Corollary \ref{C0evol}, we see that when $r\geq1$, we have that
\[\left|\ho r\right| \leq C v^2, \qquad |\n r |\leq Cv^2 \]
where $C$ depends only on $n$. Assuming $R>>1$, $p\geq 3$ and writing $C$ for any bounded constant that depends only on $n$ and $p$ which may vary from line to line, we have that at any increasing maximum point of $g:=f \varphi_R^p(r)$,
\begin{flalign*}
 0&\leq\ho f\varphi_R^p  \\
 &\leq-\delta \varphi_R^p f^2-2\ip{\n f}{\n \varphi_R^p}\\
 &\qquad+f \left(p\varphi_R'\varphi_R^{p-1} \ho r - \left[p(p-1)(\varphi_R')^2\varphi_R^{p-2}+p\varphi_R^{p-1}\varphi_R'' \right]|\n r|^2\right)\\
 &=-\delta \varphi_R^p f^2 +f \left(p\varphi_R'\varphi_R^{p-1} \ho r +\varphi_R^{p-2}\left[ p(p+1)(\varphi_R')^2-p\varphi_R\varphi_R''\right] |\n r|^2\right)\\
 &\leq -\delta \varphi_R^p f^2 +fC(R^{-1}+R^{-2})\varphi_R^{p-2}v^2,
\end{flalign*}
where we used that $\n(f\varphi_R^p)=0$. We therefore see that for $p=3$,
\[\delta (\varphi_R^3f)^2\leq \varphi_R^4fC(R^{-1}+R^{-2})C_v,\]
which implies that
\[\varphi_R^3f\leq \max \left\{\frac{\sqrt{2}C(R^{-1}+R^{-2})C_v}{\sqrt{\d}}, C_f\right\}=:\Lambda.\]

Now setting $p=5$, the above evolution inequality for $g=f\varphi_R^p$ implies that for any $\tau\in(0,1)$, if \[g> \sqrt{\tau^{-1}\d^{-1} CC_v(R^{-1}+R^{-2})\Lambda}\] then
\[\ho g \leq -\d g^2 +CC_v(R^{-1}+R^{-2})\Lambda \leq -\d(1-\tau)g^2,\]
where we used that $f\geq g$ everywhere. Therefore, 
\[g\leq \max\left\{\frac{1}{(\sup_{M_0} g)^{-1} +\d(1-\tau)t}, \sqrt{\tau^{-1}\d^{-1} CC_v(R^{-1}+R^{-2})\Lambda}\right\}.\]
Setting $\tau=R^{-\frac{1}{2}}$ and sending $R\ra \infty$ now implies that on $M_t\cap\mathcal{C}_1$, 
\[g\leq \frac{1}{(\sup_{M_0} f)^{-1} +\d t}.\]
As the center of the cylinder was arbitrary, this estimate holds everywhere.
\end{proof}

The next proposition shows that {\sensible} solutions  satisfy estimates similar to expander solutions.
\begin{proposition}\label{prop:tame.est}
Suppose that on $M_0$, \eqref{safetyassume} holds and additionally that we have a solution to \eqref{EntireMCF} that is {\sensible}. Then for all $t\in[0,\infty)$,
\[v^2\leq C_v, \qquad 
\|H\|^2 \leq \frac{1}{C_H^{-1} +2n^{-1}t}\quad \text{ and }\quad \|\II\|^2 \leq \frac{m}{2t}\ .\] 
\end{proposition} 
\begin{proof}
 Since the function $f$ given by Definition \ref{def:sensible} of a {\sensible} solution $\U$  is continuous, there exists a maximal time $\hat T$ such that for all $t\in[0,\hat T)$, $\underset{M_t}\sup\|H\|^2<2C_H$. 
We may now apply Lemma \ref{localw_KHbound} to obtain  the existence of $p=p(n)$ such that at the point $y\in M_t$ with $r(y)=0$ (where $r$ is as in Corollary \ref{C0evol}), 
\[v^2(y,t)\leq C_v\left(\frac{R^2}{R^2 +\|\U\|^2 - 2nt}\right)^pe^{4C_Ht}\]
for all $t<\min\left\{\hat T, \frac{R^2}{2n}\right\}$. We now choose $R=1$ and $\tilde{T}$  to be sufficiently small so that for all $t<\tilde{T}$, $\left(\frac{1}{1 - 2nt}\right)^pe^{4C_H t}<2$. Clearly $\tilde{T}$ depends only on $C_H$ and $n$. Let $\ov{T}=\min\{\tilde{T}, \hat{T}\}$. As the origin may be chosen arbitrarily, we see that for $t<\ov T$, $v<2C_v$ everywhere on $M_t$, and we may apply Proposition \ref{GradMaxPrinc} to $\|H\|^2$ (using Corollary \ref{evolcurv2}) to obtain that on $M_t$, 
\begin{equation}\label{eq:H.Dirichlet.est}
\|H\|^2\leq \frac{1}{C_H^{-1} +\frac{2t}{n}}  .
\end{equation}
We therefore see that $\ov T = \tilde{T}$.

We now use \eqref{eq:H.Dirichlet.est} and Corollary \ref{wkevolest} 
to estimate for any $\d\in(0,(2n)^{-1}]$,
\[\ho w_A^2 \leq -\left(1+\d \right)\frac{|\n w_A^2 |^2}{w_A^2 } +4n\d\frac{1}{C_H^{-1}+\frac 2 nt}w_A^2 .\]

Applying Lemma \ref{evolutioncutoff} to $w_A^2(C_H^{-1}+\frac 2 n t)^{-4n\d }$ on $\mathcal{Q}_R$ 
we see that on $M_t\cap \{x\in \bb{R}^{n,m}| r(x)=0\}$, 
\[v^2\leq C_v\left(\frac{R^2}{R^2 - 2nt}\right)^\frac{1+\d}{\d}\left[1+\frac{2C_H}{n}t\right]^{4n\d } .\]
Suppose now that for some $0<t_0<\tilde T$, and  $x_0\in M_{t_0}$ with $r(x_0)=0$ that
\[v(x_0,t_0) = C_v+\e.\]
Then we may form a contradiction, for example by choosing $\d$ sufficiently small so that $\left[1+\frac{2C_H}{n}t_0\right]^{4n\d }\leq \sqrt{1+\frac \e 2}$, and then choosing $R$ sufficiently large so that $\frac{1}{(1-\frac{2nt_0}{R^2})^\frac{1+\d}{\d}}\leq \sqrt{1+\frac \e 2}$. Therefore, we see that for all $t\in [0,\tilde{T})$,
\[\underset{M_t} \sup \,v^2 \leq C_v.\]
As $\tilde{T}$ depended only on $n$ and $C_v$, we may iterate the above process to get the statement.  The final estimate comes from applying Proposition \ref{GradMaxPrinc} to $\|\II\|^2$, using the evolution equation in Corollary \ref{evolcurv2}.
\end{proof}
\begin{cor}\label{sensibleestimates}
 Suppose that on $M_0$ \eqref{safetyassume} holds and that we have a solution to \eqref{EntireMCF} that is {\sensible}. Then for all $k>1$ there exists a constant $c_k=c_k(n,m,k,C_v)$ such that for all $t\in[0,\infty)$, 
 \[\|\n^k\II\|^2t^{k+1}\leq c_m.\]
\end{cor}
\begin{proof}
 As noted in \cite[Proposition 3.7]{EckerEntire} the evolution equations for $\|\n^k\II\|^2$ may be estimated as in the Euclidean graphical case, and we have the same estimate on $\|\II\|^2$. The proof of \cite[Theorem 3.4]{EckerHuiskenInteriorEstimates} then carries through identically.
\end{proof}

The estimates in Proposition \ref{prop:tame.est} have three straightforward, interesting corollaries.

\begin{cor}\label{hyperbolicexpander}
The expanding quasi-sphere given by Definition \ref{quasisphere} acts as an outer  barrier to any {\sensible} solution to \eqref{EntireMCF}. 
\end{cor}
\begin{proof}
 We consider the function $h_p(x,t) = |x-p|^2+2nt+R^2$, and we suppose that $\inf_{M_0}h\geq 0$. The estimates of Proposition \ref{prop:tame.est} imply that for any $T>0$ there exists $\rho=\rho(T, C_v, C_H)$ such that for all $t\in[0,T)$, $\inf_{M_t\setminus\mathcal{C}_\rho(p)} h_p\geq 0$. Therefore, we may apply the maximum principle on $M_t\cap\mathcal{C}_\rho$ to see that this is preserved on the time interval $[0,T)$. As $T$ was arbitrary, the statement now follows.
\end{proof}

\begin{cor}
 Let $M_t$ be an entire {\sensible} solution to \eqref{EntireMCF} and
suppose there exists $R>0$ such that for all $t>0$, $M_t\cap B_R\neq\emptyset$; 
 i.e.~that $M_t$ does not escape to infinity. Then there exists a sequence $t_i\ra\infty$ such that $M_{t_i}$ converges to an $n$-plane. 
\end{cor}

\begin{cor}
 There are no entire shrinking or translating solutions to spacelike MCF in $\bb{R}^{n,m}$ with bounded $v$ and $H$.
\end{cor}

\section{Convergence of entire solutions}\label{s:sensibleconvergence}
In this section we prove convergence of {\sensible} (in the sense of Definition \ref{def:sensible}) entire solutions to spacelike MCF  (\ref{EntireMCF}) with initial data $M_0$ which is asymptotic to a spacelike cone.  

Suppose that $L$ is a uniformly spacelike cone centred at the origin, smooth away from the origin, given graphically by functions $U:\bb{R}^n\to \bb{R}^m$ such that for all $\l>0$ and $x\in\bb{R}^n$, $U(\l x) = \l U(x)$. 

\begin{defses}\label{def:asymptotic.cone}
We say that $M_0$, given by the graph of $\U_0$, is asymptotic to the cone $L$ if
\begin{equation}\label{asyptoticconvergence}\underset{R\ra\infty}\lim\,\underset{\bb{R}^n \setminus B_R}\sup |\U_0 - U| =0 .
\end{equation}
Note that, in this setting, $L$ satisfies the same gradient estimate as $M_0$.
\end{defses}

We recall how to renormalise solutions to MCF in the standard way. We write $s=\frac 1 2 \log(2t+1)$, and define
\begin{equation}\label{eq:renormalise}
\widetilde{X}(x,s)= \frac{1}{\sqrt{1+2t}}X(x,t).
\end{equation}
We will write all quantities for the rescaled flow with a tilde to avoid confusion.
We also recall that submanifolds satisfying
\begin{equation}\label{eq:self.exp}
\widetilde{H} = \widetilde{X}^\perp
\end{equation} 
 are called \emph{self-expanders} and are critical points of the renormalised MCF.

We now state the following convergence statement, whose proof shall take up the remainder of the section.

\begin{theorem}\label{ConvergenceTheorem}
 Suppose that $M_t$ is a {\sensible} entire solution to spacelike MCF \eqref{EntireMCF} such that $M_0$ is asymptotic  to the cone $L$ as in Definition \ref{def:asymptotic.cone}. For any sequence $t_i\ra\infty$ there is a subsequence (also labelled $t_i$) and a self-expander $\widetilde{M}_{\infty}$ 
  such that the renormalised flow satisfies $\widetilde{M}_{t_i}\rightarrow \widetilde{M}_{\infty}$ in $C_\text{loc}^\infty$ as $i\ra\infty$.  
\end{theorem}


We first show that (\ref{asyptoticconvergence}) is preserved by the flow, by attaching expanding quasi-spheres of arbitrarily large radius to the initial data. For this purpose, we begin with the following result.
\begin{lemma}\label{shiftlemma}  Let $p\in\bb{R}^{n,m}$, let $R>0$ and recall the inside $I_t$ of the quasi-sphere expander from Definition \ref{quasisphere}.    
 Suppose that $M_0\subset I_0$  
 and $M_0$ has $v^2\leq C_v$ everywhere. For $\epsilon<1$, let
 \[O_\e:=\{x+y\in\bb{R}^{n,m}| x\in M_0, y\in \bb{R}^m = \text{\emph{span}}\{e_1, \ldots,e_m\}, \|y\|\leq \e\}.\]
 There exists $C=C(C_v,R)>0$ such that, for all $\e<1$, 
 $O_\e\subset I_{C\e/n}$.
\end{lemma}
\begin{proof}
 Without loss of generality we may take $p=0$. By observations in Appendix \ref{MCFasGraphs}, since $v^2<C_v$ there exists  $\tau\in(0,1)$ such that for any direction $\t\in \bb{R}^m$ (using the notation of Section \ref{s:Dirichlet}),
 \[D_i\U_0\cdot \t<\sqrt{1-\tau} .\]
Hence, a short calculation using the gradient estimate (and mean value theorem) shows that if we take
 \[\rho = R\sqrt{\frac{1-\tau}{\tau}} \]
then the graph of $\U_0$ over $\bb{R}^n\setminus B_{\rho}(0)\subset \bb{R}^n$ cannot intersect the quasi-sphere $S_0$.
 We may therefore estimate 
 \[
 \|\U_0\| \leq
 \begin{cases}
  \sqrt{R^2+|x|^2} &\text{on } B_\rho(0),\\
  \sqrt{1-\tau}\,(|x|-\rho)+\sqrt{R^2+\rho^2} &\text{on } \bb{R}^n \setminus B_\rho(0) .
 \end{cases}
 \]
 We consider a point $x+y \in O_\e$. Over $B_\rho(0)$ we have
 \begin{flalign*}
  \|y\|^2 \leq (\|\U_0\|+\e)^2
  &=\|\U_0\|^2+2\e\|\U_0\|+\e^2 
  \leq |x|^2+R^2+2\e \frac{R}{\sqrt{\tau}}+\e^2,
 \end{flalign*}
and so at any point $x+y\in \mathcal{C}_\rho \cap O_\e$ (where $\mathcal{C}_\rho$ is as in Definition \ref{def:cyl.sqsphere.cutoff}), 
\[|x+y|^2>-R^2-\e\left(2\frac{R}{\sqrt{\tau}}+\e\right)  .\]
 
Now we consider the minimum value of $|x+y|^2$ on $ O_\e\setminus\mathcal{C}_\rho$. We have that
\begin{flalign*}
|x+y|^2 &\geq |x|^2-\left[\sqrt{1-\tau}\,(|x|-\rho)+\sqrt{R^2+\rho^2}+\e\right]^2\\
&=\tau|x|^2 +2\left[\rho (1-\tau) - (\sqrt{R^2+\rho^2}+\e)\sqrt{1-\tau}\right]|x|\\
&\qquad -(1-\tau)\rho^2+2\rho\sqrt{1-\tau}(\sqrt{R^2+\rho^2}+\e)-(\sqrt{R^2+\rho^2}+\e)^2\\
&:=\psi(|x|).
\end{flalign*}
Note that $\psi$ is a quadratic in $|x|$ with positive highest order term.  Therefore, $\psi$ attains its global minimum at
\[|x|=\frac{(\sqrt{R^2+\rho^2}+\e)\sqrt{1-\tau}-\rho(1-\tau)}{\tau}=\rho +\frac{\sqrt{1-\tau}}{\tau}\e  ,\]
and 
\begin{flalign*}
\psi&\geq \left(\rho+\frac{\sqrt{1-\tau}}{\tau}\e\right)^2 -\left[\frac{\e}{\tau} +\sqrt{R^2+\rho^2}\right]^2\\
&=\rho^2-(R^2+\rho^2)+2\left[\rho\sqrt{1-\tau}-\sqrt{R^2+\rho^2}\right]\frac{\e}{\tau}+\left[1-\tau-1\right]\frac{\e^2}{\tau^2}\\
&=-R^2-2\left[\frac{1}{\sqrt{1-\tau}} - \sqrt{1-\tau}\right]\frac{\e\rho}{\tau}-\frac{\e^2}{\tau}  .
\end{flalign*}
The claim now follows.
\end{proof}

\begin{lemma}\label{C0close}
 Suppose that $w$ is a {\sensible} solution to \eqref{EntireMCF} and the initial data satisfies \eqref{asyptoticconvergence}. Then, for all $t>0$, 
 \[\underset{R\ra\infty}\lim\,\underset{x\in\bb{R}^n \setminus B_R}\sup |w(x,t) - U(x)| =0 .\]
 More precisely, for any $\e, T>0$ there exists $\rho=\rho(\e,T)$ such that for all $t\in[0,T)$,
 \[ \underset{(x,t)\in (\bb{R}^n \setminus B_\rho)\times[0,T)}\sup |w(x,t) - U(x)| <\e .\] 
\end{lemma}
\begin{proof} 
Our aim is to show that sufficiently far away from the origin we may attach quasi-sphere  to $M_0$ with arbitrarily large negative 
square radius, which then contain $M_0$. As the barriers starting from these quasi-spheres move arbitrarily slowly the theorem will then be achieved. We make this intuitive argument explicit, due to the difficulty in visualising higher codimension submanifolds. 

For any $p$ in the asymptotic cone $L$ and any unit vector $\nu\in N_pL$ we define
  \[\phi_{p,\nu,\l}(x):=|x-p-\l\nu|^2+\l^2.\]
We now complete the proof in several steps.

\medskip

\paragraph{\bf Step 1:} \emph{There exists $R_0>0$ such that for any $x\in \partial B_1(0)\subset\bb R^n$ and any $\nu\in N_{(x,U(x))}L$, we may attach a quasi-sphere of square radius $-R_0^2$ in any unit direction $\nu$ which contains $L$; that is, $\phi_{(x,U(x)), \nu, R_0}\geq 0$ everywhere on $L$.}
   \begin{proof}
  We observe that at $p=(x, U(x))$, $\phi_{p,\nu, \l}(p)=0$ for all $\l$ and, since $L$ is spacelike, $\phi_{p,\nu, 0}\geq 0$ on $L$. Furthermore, 
  \[\tilde{\l}(p):=\inf\{\l\in [0,1]\,|\,   \phi_{p,\nu,\l}\geq 0 \text{ on }L \text{ for all unit }\nu \in N_p L\}>0\] 
  as otherwise we may contradict uniform spacelikeness using the mean value theorem. Furthermore, the fact that $L$ is uniformly spacelike implies that there exists $\tilde{R}$ such that if $\tilde{\l}(p)<1$, then there is $y\in B_{\tilde{R}}$ such that $\phi_{p,\nu, \tilde{\l}}(y,U(y))=0$. As $\phi_{p,\nu,\lambda}$ is smooth in $p,\nu,\lambda$, we see by standard methods that $\tilde{\l}$ is (Lipschitz) continuous. Therefore $\tilde{\l}$ has a positive minimum, $R_0$, on $\partial B_1(0)$ as claimed.
  \end{proof}
  
\paragraph{\bf Step 2:} \emph{For any $\e>0$, $R>0$, there exists $\rho_0=\rho_0(R,\e, \U_0)>0$ such that for any $x\in\bb{R}^n\setminus B_\rho(0)$, and any $\nu\in N_{(x,U(x))} L$, $M_0$ is contained inside a quasi-sphere of square radius $-R^2$ with centre 
 $(x,U(x))- (R-\e)\nu$.
 }
 \begin{proof} 
By the scaling properties of the cone and  quasi-sphere, we may choose $\rho_1=\rho_1(R,R_0)$ sufficiently large so that Step 1 implies that for any $p\in L\setminus \mathcal{C}_{\rho_1}$ and any unit $\nu\in N_pL$ we can attach a quasi-sphere of square radius $-R^2$ in direction $\nu$ which contains $L$. The condition (\ref{asyptoticconvergence}) now implies that for any $\tilde{\e}>0$, there exists $\rho_1<\rho_2=\rho_2(M_0,\tilde \e, \rho_1)$ such that 
 \[\sup_{\bb{R}^n\setminus B_{\rho_2}(0)} \|\U_0-U\|<\tilde{\e} .\]
 Lemma \ref{shiftlemma} now implies that by choosing $\tilde{\e}$ to be sufficiently small and relabelling constants the claim follows.
 \end{proof}

\paragraph{\bf Step 3:}\emph{Completing the proof.} Given any $\tilde{\e}>0$, $T>0$, we choose $R>0$ so large that the quasi-sphere expander starting from square radius $-R^2$ as in Step 2 moves at most $\frac{\tilde{\e}}{2}$ in direction $\nu$ on the time interval $[0,T]$. We now apply Step 2 to find a $\tilde{\rho}$ such that for all $x\in\bb{R}^n\setminus B_\rho(0)$, we may attach expanding quasi-spheres in any direction as in Step 2. The proof is complete by choosing $\tilde{\e}<\frac{\e}{C_v}$.
\end{proof}

The following result is proved in a similar way to Proposition \ref{GradMaxPrinc}, and indicates that if we have only small osculation of $\|u\|^2$, then $v^2$ decays exponentially in time.
\begin{lemma}\label{C1close}
 For any $\e\in(0,1)$ and {\sensible} entire solution to \eqref{EntireMCF} such that 
 \[v^2<C_v, \qquad \|H\|^2<C_H  ,\]
 there exists $R=R(\e, C_v)>1$ such that: if for all $t\in[0,T)$
 \[0<u_A^2<C_u<1 \text{ on } \mathcal{C}_R \cap M_t,\]
 then for all $y\in \mathcal{C}_1\cap M_t$, 
 \[w_A^2(y,t) \leq (1+\e)e^{C_u}+e^{-t}\sup_{\mathcal{C}_R\cap M_0}w_A^2e^{u_A^2}\] 
\end{lemma}
\begin{proof}
 We see from Lemma \ref{evolgoodwk} that for $f=w_A^2e^{u_A^2}$ we have that
 \[\ho f \leq -e^{u_A^2}w_A^2(w_A^2-1) \leq -e^{-C_u}f^2+f   .\]
 We choose $\varphi_R$ as in Proposition \ref{GradMaxPrinc}, and set $g=\varphi^p_R f$, where $R$ will be determined later. Arguing as in the proof of Proposition \ref{GradMaxPrinc}, we have at an increasing maximum:
 \begin{flalign*}
 0\leq\ho f\varphi^p_R(r) \leq -\delta \varphi^p_R f^2+f\varphi^p_R +fC(R^{-1}+R^{-2})\varphi^{p-2}_Rv^2,
\end{flalign*}
 where $\delta=e^{-C_u}$. Taking $p=3$,
\[\delta (\varphi^3_Rf)^2\leq \varphi^3_Rf[C(R^{-1}+R^{-2})C_v+1],\]
which implies that
\[\varphi^3_Rf\leq \max \left\{\frac{\sqrt{2}(C(R^{-1}+R^{-2})C_v+1)}{\sqrt{\d}}, \sup_{M_0\cap \mathcal{C}_R} f\right\}=:\Lambda.\]
Taking instead $p=5$,
 \begin{flalign*}
 \ho g \leq -\delta g^2+g+C(R^{-1}+R^{-2})\Lambda C_v 
\end{flalign*} 
 or, writing $\tilde{g} = g - \frac{1}{2\d}$, we have that
 \begin{flalign*}
 \ho \tilde{g} \leq -\delta \tilde{g}^2+C(R^{-1}+R^{-2})\Lambda C_v+\frac{1}{4\d}  .
\end{flalign*} 
Using the concavity of $y(x)=C-x^2$, we may estimate  that for \[\delta \tilde{g}\geq\sqrt{\delta C(R^{-1}+R^{-2})\Lambda C_v+1/4}:=b\]
 we have
 \begin{flalign*}
 \ho \tilde{g} \leq -2 b \tilde{g}.
\end{flalign*} 
Hence, 
 \[\tilde{g}\leq \max\left\{e^{-2bt}\sup_{M_0}\tilde{g}, \frac{b}{\d}\right\}< \frac{b}{\d}+\sup_{M_0}\tilde{g}e^{-2bt}  .\]
 Picking $R$ sufficiently large depending on $\e$ and $C_v$ (where we estimate $\d$ by $1$), we may assume that $b\leq \frac{1+2\e}{2}$. Hence on $\mathcal{C}_1\cap M_t$, 
 \[w_A^2e^{u_A^2}\leq (1+\e)e^{C_u}+e^{-(1+2\e)t}\sup_{\mathcal{C}_R\cap M_0}w_A^2e^{u_A^2}  \]
 as claimed.
\end{proof}

Corollary \ref{hyperbolicexpander} implies that there exists a constant $c_0$ depending only on the initial data such that, on $M_t$,
\begin{equation}0<r^2<c_0 +2nt +|X|^2.
 \label{Xvsr}
\end{equation}
Hence, on $M_0$, there exists $C$ depending on the gradient bound $C_v$ such that
\[\frac{|X^\perp |^2}{c_0 + |X|^2}\leq C.\]
We now use \eqref{asyptoticconvergence} to show that $\frac{|X^\perp|^2}{|X|^2}$ decays far away from the origin.
\begin{lemma}\label{Xperpest}
 Suppose $M_0$ satisfies \eqref{asyptoticconvergence}. For all $T>0$, there exists $R(T, L)>0$ and $C=C(C_v)$ such that for all $t\in[0,T)$ and $p\in M_t\setminus \mathcal{C}_R$, 
 \[\frac{\|X^\perp\|^2}{|X|^2}\leq Ce^{-\frac t 2}.\]
\end{lemma}
\begin{proof}
We pick $\e<e^{-2T}$ sufficiently small such that $(1+\e)e^\e-1<e^{-T}$. Let $R$ be as in Lemma \ref{C1close} with this choice of $\e$.

At $p\in L$, let $\{\nu_1, \ldots, \nu_m\}$ be an orthonormal basis of $N_pL$. We rotate about $0$ in $\bb{R}^{n,m}$ so that $\nu_A$ is in the $e_A$ direction. In rotated coordinates we will write all objects with a check, e.g.~$\check{p}$ for the rotation of $p$,  $\check{\mathcal{C}}_R(\check{p})$ for a cylinder of radius $R$ at $\check{p}$ in the new rotated coordinates. Since $L$ is a cone, $\nu_1, \ldots, \nu_m$ is also an orthonormal basis  of $N_{\l p}L$ for all $\l>0$. From scaling properties of the cone, by assuming $|p|$ is sufficiently large, we see that $L$ may be arbitrarily well approximated by its tangent plane, $\bb{R}^n$. Therefore, due to the gradient bound and Lemma \ref{C0close}, there exists $R_1>0$ such that for all ${p}\in M_t \setminus \mathcal{C}_{R_1}$ (in the unrotated coordinates),
\[0<\check{u}_A^2(\check q, t)<\e \text{ for all } t\in[0,T), \check q \in \check{\mathcal{C}}_{R}(\check{p})\cap \check{M_t}.\]
We assume we have made $R_1$ sufficiently large so that $|X|>1$ (if not, we increase $R_1$). 
We therefore may apply Lemma \ref{C1close} so that on $[0,T)$ we have
\[\check{w}_A^2 \leq (1+\e)e^\e+Ce^{-t},\]
where $C=C(C_v)$. Rotating back, we see that for all $q\in M_t \cap\check{\mathcal{C}}_1(\check{p})$ there exist orthonormal timelike constant vectors $\check{e}_1,\ldots , \check{e}_m$ such that 
\[\|\check{e}_A^\perp\|^2\leq (1+\e)e^\e+Ce^{-t} \ .\]
Let $\check{f}_i$ be an orthonormal basis of $T_pL$ (similarly extended). Since $$\nu_A = \sum_i \ip{\check{f}_i}{\nu_A}\check{f}_i - \sum_B \ip{\nu_A}{\check{e}_B}\check{e}_B,$$ we then have that
\[-1=|\nu_A|^2 = \sum_i \ip{\check{f}_i}{\nu_A}^2 - \sum_B \ip{\nu_A}{\check{e}_B}^2.\]
Thus,
\[\sum_i \|\check{f}_i^\perp\|^2 = \sum_B \|\check{e}_B^\perp\|^2-n .\]

We deduce that for any $q\in M_t \cap \check{\mathcal{C}}_1(\check{p})$, we may write $q=\check{x}+\check{w}$ where $\check{x}\in T_p L$ and $\check{w}\in N_pL$, to estimate
\begin{flalign*}
\frac{\|X^\perp\|}{\sqrt{|X|^2}}&\leq 2\frac{\|\check{x}^\perp\|+ \|\check{w}^\perp\|}{\sqrt{|\check{x}|^2-\|\check{w}\|^2}}\\
&\leq 2 \frac{|\check{x}|\sqrt{n}\sqrt{(1+\e)e^\e-1+Ce^{-t}}+ \sqrt{\e} ((1+\e)e^\e+Ce^{-t})}{\sqrt{|\check{x}|^2-m\e}}\\
&\leq \check{C}e^{-\frac{t}{2}}
\end{flalign*}
due to the chosen bounds on $\e$. As $\check{C}$ depends only on $C_v$, the statement follows.
\end{proof}

Recall the renormalising of MCF given in \eqref{eq:renormalise} and that we write all quantities for the rescaled flow with a tilde.
For convenience we define the scaling factor $\l(t) = \frac{1}{\sqrt{1+2t}}$.  We see that, under the renormalisation,
\[\frac{d\widetilde{X}}{ds} = \widetilde{H} -\widetilde{X}.\]

We want to understand how quantities evolve under the renormalised flow based on their evolution under MCF.  We will say a quantity $f$ (for example, a function) is of degree $\alpha$ if when the submanifold is dilated by a factor of $\l$, the rescaled quantity satisfies $\tilde f =\l^\alpha f$.
\begin{lemma}\label{degreeevol}
 Suppose $f$ is a function of degree $\alpha$. Then
  \[\hos \tilde f  = \l^{\a-2}\ho f - \a\tilde{f}\ .\]
 The same statement is true of tensors. 
\end{lemma}
\begin{proof}
The proof for functions is exactly as in \cite[Lemma 9.1]{Huiskenconvex}. The tensor case follows identically, however note that we also rescale the coordinate vectors.
\end{proof}

We now note that we have the following evolution inequalities along the renormalised flow. 
\begin{lemma}\label{finalevols}
There exist $\Upsilon, p>0$ depending on $n$ and $C_v$ such that along the renormalised flow, 
 \[\hos \frac{1}{t^p}e^{-\frac{\tilde{r}^2}{\Upsilon t}}\geq 0 .\]
 where $\tilde{r}^2$ is the renormalisation of $r^2$ in Corollary \ref{C0evol}. Furthermore,
 \begin{flalign*}
\hos \|\widetilde H - \widetilde X^\perp\|^2  &\leq - 2\|\widetilde H - \widetilde X^\perp\|^2 - \frac{1}{2}\frac{\left|\n \| \widetilde H - \widetilde X^\perp\|^2 \right|^2}{\| \widetilde H - \widetilde X^\perp\|^2}.
\end{flalign*}
\end{lemma}
\begin{proof}
We have that 
\[\hos \tilde{r}^2 = -2n - 2\tilde r^2-(\tilde{v}^2-m).\]
Since
\[\hos e^f = e^f\left(\hos f - |\n f|^2\right),\]
if we let $f = -\frac{\tilde{r}^2}{\Upsilon t}-p\log t$, we see that
\begin{flalign*}
 \hos t^{-p}e^{-\frac{\tilde{r}^2}{\Upsilon t}}&= t^{-p}e^{-\frac{\tilde{r}^2}{\Upsilon t}}\left(\frac 1 {\Upsilon t} (2n+2\tilde{r}^2+(\tilde{v}^2-k)) +\frac{\tilde{r}^2}{\Upsilon t^2} - \frac{p}{t} - \frac{4\tilde{r}^2|\n \tilde{r}|^2}{\Upsilon^2 t^2}\right)\\
 &\geq t^{-p}e^{-\frac{\tilde{r}^2}{\Upsilon t}}\left(\frac 1 {\Upsilon t} (2n+2\tilde{r}^2+(\tilde{v}^2-k)-\Upsilon p) +  \frac{\tilde{r}^2(\Upsilon-4C_v)}{\Upsilon^2 t^2}\right).
\end{flalign*}
Hence, taking $\Upsilon= 4C_v$ and $p = \frac{2n}{4C_v}$ gives the first claim.

Since $H_A$ is degree $-1$, using Lemmas \ref{evolcurv} and \ref{degreeevol} we obtain that
\[\hos \widetilde{H}_A=- \widetilde H_B\ip{\nu^B}{\widetilde \II_{ij}}\ip{\widetilde \II^{ij}}{\nu_A}+\widetilde H_A.\]
Similarly, for $Q_A:=\ip{X}{\nu_A}$ of degree $1$, we may use Lemmas \ref{Generalnormaloneform} and \ref{degreeevol} to yield
\[\hos \widetilde{Q}_A = -\widetilde Q_B\ip{\nu^B}{\widetilde \II_{ij}}\ip{\widetilde \II^{ij}}{\nu_A}-2\widetilde H_A-\widetilde Q_A.\]
Writing $\widetilde{W}_A=\widetilde{H}_A+\widetilde{Q}_A$ and the nonnegative 2-tensor $\widetilde{S}^B_A = \ip{\nu^B}{\widetilde \II_{ij}}\ip{\widetilde \II^{ij}}{\nu_A}$, we have that
\begin{align*}
 \hos \widetilde W_A & = -\widetilde W_B \widetilde S^B_A - \widetilde{W}_A
\end{align*}
 As $\|\widetilde H - \widetilde X^\perp\|^2 = \sum_{A=1}^m \widetilde{W}_A^2$, we have that
\begin{flalign*}
\hos \|\widetilde H - \widetilde X^\perp\|^2 &= -2\widetilde W_B \widetilde S^{BA} \widetilde W_A - 2\|\widetilde H - \widetilde X^\perp\|^2 - 2\|\n^\perp( \widetilde H - \widetilde X^\perp)\|^2\\
 &\leq -2\widetilde W_B \widetilde S^{BA} \widetilde W_A - 2\|\widetilde H - \widetilde X^\perp\|^2 - \frac{1}{2}\frac{\left|\n \| \widetilde H - \widetilde X^\perp\|^2 \right|^2}{\| \widetilde H - \widetilde X^\perp\|^2}.
\end{flalign*}
The second claim follows.
\end{proof}

Given equation \eqref{Xvsr}, on $\tilde{M}_t$, we now know that for $c>c_0$ we have 
\begin{equation} 1<c+|\tilde{X}|^2=:\tau_c\qquad \text{ and } \qquad \lim_{R\ra\infty}\inf_{\ti M_t\setminus\mathcal{C}_R} \tau_c \ra \infty.
 \label{NewXest}
\end{equation}
This leads us to a further evolution inequality along the renormalised flow.

\begin{lemma}
There exists $c>0$ such that, along the renormalised flow, 
 \[\hos \|\widetilde H - \widetilde X^\perp\|^2 \tau_c^{-1} \leq 0.\]
\end{lemma}
\begin{proof}
 Writing $f=\|\widetilde H - \widetilde X^\perp\|^2$ and estimating using Young's inequality:
\begin{flalign*}
\hos f \tau^{ -1}_c&\leq -2f\tau_c^{-1} -\frac{|\n f|^2}{2f}\tau_c^{-1}- 2 \ip{\n f}{\n \tau_c^{-1}}\\
&\qquad+f\tau_c^{  -1}\left( -\tau_c^{-1}\hos |\widetilde X|^2  - 2\tau_c^{-2}|\n|\widetilde X|^2|^2\right)\\
&\leq f\tau_c^{  -1}\left(-2+2(n+|\widetilde X|^2)\tau_c^{-1}\right)  .
\end{flalign*}
For $c> n$, we may estimate $(n+|\widetilde X|^2)\tau^{-1}<1$, giving the result.
\end{proof}

Finally, we use an argument similar to \cite[Theorem A.2]{ClutterbuckSchnuerer} to get our claimed convergence.
\begin{proof}[Proof of Theorem \ref{ConvergenceTheorem}]
We write $f=\|\widetilde H - \widetilde X^\perp\|^2 \tau_c^{-1}$. Clearly $f$ is bounded. We claim that
\begin{equation}\limsup_{t\ra\infty} \sup_{M_t} f =0 .
 \label{fconvergence}
\end{equation}

Let $\e>0$. We first show that outside cylinders of sufficiently large radii and for sufficiently large times, $f$ is smaller than $\e$: Under the renormalised flow $H$ is bounded, so we only need to show that $\|\widetilde{X}^\perp\|^2\tau_c^{-1}$ is small. We first wait until 
\[t_0 = 2\log\left(\frac{2C}{\e}\right)\]
where $C$ is as in Lemma \ref{Xperpest}. By Lemma \ref{Xperpest}, there exists  $R=R(t_0,L)$ such that for all $t\in[0,2t_0)$, on $M_t\setminus \mathcal{C}_R$, 
$$f \leq Ce^{-\frac{t}{2}} .$$
In particular, for all $t\in[t_0,2t_0)$, on $M_t\setminus \mathcal{C}_R$, 
$$f \leq \frac{\e}{2} .$$

  On the interval $[t_0, \infty)$, for $a,\delta>0$  we consider the function
\[g = f-\e -a\psi ,\quad\text{where}\quad\psi = (t-t_0+\d)^{-p}e^{-\frac{\tilde{r}^2}{\Upsilon(t-t_0+\d)}}.\]
 We choose $p,\Upsilon>0$ as in Lemma \ref{finalevols} so that the heat operator acting on $\psi$ is positive.  Furthermore, after choosing $\d>0$ small (so that this is a smooth function at time $t=t_0$) we observe that there exists  $a>0$ such that at time $t_0$, $g<0$. Outside $\mathcal{C}_R$, this is trivially true, and inside this follows since $\psi$ is strictly positive and continuous, and $f$ is bounded.

For all  $t>t_0$, Lemma \ref{Xperpest} guarantees that the set such that $f\geq\frac{\e}{2}$ is compact. Observing that 
\[\ho g \leq 0 ,\]
we may therefore apply the weak maxiumum principle on larger and larger compact domains to imply that $g\leq 0$. Since $\psi$ decays uniformly to zero, 
\[\limsup_{t\ra\infty} \sup_{\tilde{M}_t} f \leq 2\e .\]
Since $\e$ was arbitrary, \eqref{fconvergence} holds.

As a result, for any $t_i\ra\infty$, for all $1\leq j\in\bb{N}$ Proposition \ref{prop:tame.est} and Corollary \ref{sensibleestimates} we have uniform curvature and higher order estimates on $\mathcal{C}_j\cap\tilde{M}_{t_i}$. Arzel\'a--Ascoli and \eqref{fconvergence} imply there exists a subsequence which converges to a portion of MCF expander as $i\ra\infty$. Repeating this argument for each $j$ and taking a diagonal sequence implies the statement.
\end{proof}

\begin{appendix}

\section{Spacelike mean curvature flow of graphs}\label{MCFasGraphs}
In this appendix we consider the mean curvature flow equation in terms of a graph; that is, for a function $\U:\bb{R}^n\times[0,T)\ra\bb{R}^m$ defining the parametrisation $\hat{X}:\bb{R}^n\times[0,T) \ra \bb{R}^{n,m}$ by
\[\hat{X}(x,t) = x^if_i + \U^A(x,t)e_A .\]
We see that
\[\hat{X}_i = f_i +D_i\U(x,t)^Ae_A, \qquad g_{ij}= \delta_{ij} - D_i\U^AD_j\U_A,\]
and we consider the flow only when it is spacelike; i.e.~$g_{ij}>0$ as a matrix. Equation (\ref{UnparametrisedMCF}) now reads
\[\left(\ddt {\hat X} \right)^\perp = H = g^{ij}\left(\frac{\partial^2 \hat X}{\partial x^i\partial x^j}\right)^\perp,\]
which implies 
\[\left(\ddt{\U^A} -g^{ij}D_{ij}^2\U^A \right)e_A^\perp=0 .\]
We observe that
\begin{align*}
\hat{g}_{AB}&:=\ip{e_A^\perp}{e_B^\perp} \\
&\,\,= \ip{e_A - \ip{e_A}{\hat{X}_i}g^{ij}\hat{X}_j}{e_B- \ip{e_B}{\hat{X}_k}g^{kl}\hat{X}_l}=-\delta_{AB}-D_l\U_Ag^{lk}D_k\U_B.
\end{align*}
At any point we may take coordinates on $\bb{R}^n$ so that $D_i\U^AD_j\U_A$ is diagonal with eigenvalues $\l_i\in[0,1)$. We see that (without summation in $A$),
\begin{equation}\hat{g}_{AA} = -1 -\sum_j \frac{(D_j\U_A)^2}{1-\l_j}<-1 ,
 \label{eIidentity}
\end{equation}
and so $\hat{g}$ is a symmetric negative definite matrix (which is bounded as $\l_i<1$). Hence, $\hat{g}$ is invertible and spacelike mean curvature flow is equivalent to
\begin{equation}\label{eq:graph.MCF}\displaystyle\ddt{\U^A} -g^{ij}(D\U)D^2_{ij} \U^A=0 \qquad \text{for } A=1,\ldots, m\text{ on } \bb{R}^n \times[0,T)
\end{equation}
where, as usual, $g^{ij}$ is the inverse of $g_{ij}=\delta_{ij} - D_i\U^AD_j\U_A$.

We see that the gradient function
\[v^2 = \sum_A \|e_A^\perp\|^2 = \sum_A -\hat{g}_{AA} = m + D_i\U_Ag^{ij}D_j\U^A=m-n+\sum_i \frac{1}{1-\l_i}\ .\]
Clearly while the gradient function is uniformly bounded, $\l_i<1$ and therefore $g_{ij}$ is positive definite and (\ref{EntireMCF}) is uniformly parabolic. Hence $v^2$ acts as both an estimate on how spacelike the surface is, and also the parabolicity of the PDE. We observe that this is equivalent to a bound for some $c\in(0,1)$ 
\[D_vu\cdot \t:=v^iD_iu^A\t_A<1-c\text{ for all } v\in \bb{R}^n, \t\in \bb{R}^m \text{ where } \|\t\|=1, |v|=1\ .\]
One way to see this is to consider $v$ and $\t$ which are maximisers of $D_vu\cdot \t$ and note that then, due to properties of maximisers, $v$ is the largest eigenvector of $D_{i}u^AD_{j}u_A$ with eigenvalue $\l_i=(D_vu\cdot\t)^2$. The claimed equivalence now follows.

In the Neumann boundary condition case, over the domain $\Omega$ the same equation \eqref{eq:graph.MCF}  holds, but we still need to consider the boundary condition. We require that $\mu$, which by abuse of notation is both the unit normal to $\Sigma$ and the unit normal to $\partial \Omega$, is in $T M$.  Thus
\[\mu = \mu^if_i = \mu^\top = \ip{\mu}{\hat{X}_i}g^{ij}\hat{X}_j = \mu_ig^{ij}\hat{X}_j,\]
so for all $1\leq A\leq m$,
\begin{equation}0=\mu_ig^{ij}D_j\U^A.
 \label{partbdry}
\end{equation}
Multiplying by $D_l\U_A$ and summing over $A$ we have that
\[0=\mu_ig^{ij}D_j\U^AD_l\U_A = \mu_ig^{ij}(\delta_{jl}-g_{jl}) = \mu_ig^{il}-\mu_l \ .\]
Substituting this back in to (\ref{partbdry}) yields that for all $1\leq A\leq m$
\[D_\mu \U^A = \mu^iD_i \U^A = 0.\]

In graphical coordinates the spacelike MCF with Neumann boundary condition thus becomes that, for all $ I=1,\ldots, m$,
\begin{equation}
 \begin{cases}
  \displaystyle\ddt{\U^A} -g^{ij}(D\U)D^2_{ij} \U^A=0 & \text{ on } \Omega \times[0,T),\\[4pt]
  D_\mu \U^A=0 & \text{ on } \partial \Omega \times[0,T),\\
  \U^A(\cdot,0) =\U^A_0(\cdot) & \text{ on } \Omega,
 \end{cases}
\end{equation}
where $g^{ij}$ is strictly positive definite and bounded if and only if  $v<C<\infty$.

\section{Uniqueness of MCF over compact domains.}\label{app:unique}
We demonstrate the uniqueness of solutions to the Dirichlet and Neumann problems simultaneously. 
\begin{proposition}\label{Uniqueness}
 Suppose that $\Omega\subset \bb{R}^n$ is bounded, with smooth boundary $\partial \Omega$. Let $u, z \in C^{2;1}(\Omega \times [0,T))\cap C^\infty (\Omega \times (0,T))$ be solutions to either \eqref{MCFgraphNeumann} or \eqref{MCFgraphDirichlet} which are uniformly spacelike; i.e.~there exists  $c\in(0,1)$ such that 
 \[|D\U^A|,|D\hat{z}^A|<1-c\qquad \text{ for all } 1 \leq A \leq m .\]
If $\hat{u}(\cdot,0)=\hat{z}(\cdot,0)$ then for all $t\in[0,T)$, $\hat{u}(\cdot,t)=\hat{z}(\cdot,t)$.
 \end{proposition}
 \begin{proof}
  We consider $y^A=\U^A-\hat{z}^A$.
  \begin{flalign*}
   \ddt{}y^A&= g^{ij}(D\U) D_{ij}y^A+[g^{ij}(D\U)-g^{ij}(D\hat{z})] D_{ij}\hat{z}^A\\
   &=g^{ij}(D\U) D_{ij}y^A+\int_0^1\frac{d}{d\tau}g^{ij}(Du_\tau) d\tau D_{ij} \hat{z}^A\\
   &=g^{ij}(D\U) D_{ij}y^A-D_l(\U-\hat z)_B\int_0^12D_ku_\tau^Bg^{ik}(Du_\tau)g^{jl}(Du_\tau)d\tau D_{ij}\hat{ z}^A\\
   &=g^{ij}(D\U) D_{ij}y^A-D_ly_B V^{Blij} D_{ij}\hat{ z}^A
  \end{flalign*}
where $u_\tau = \tau \U- (1-\tau) \hat{z} $ and $V^{Blij}$ is bounded due to spacelikeness of $\U$ and $\hat z$. We see that considering everything graphically,
\begin{flalign*} 
\ddt{}\|y\|^2 &= g^{ij}(D\U)D_{ij}\|y\|^2 -D_ly_B V^{Blij} D_{ij}\hat{ z}^Ay_A - 2\sum_{A=1}^m D_iy_Ag^{ij}(D\U)D_jy_A\\
&\leq g^{ij}(D\U)D_{ij}\|y\|^2+C\|y\|^2\ .
\end{flalign*}
where $C$ depends on $C_v$ and $\underset{\Omega \times[0,T)}\sup|D^2v|$. If $\hat{u}$, $\hat{z}$ satisfy \eqref{MCFgraphDirichlet} then on $\partial \Omega$, $\|y\|^2=0$. On the other hand if $\hat{u}$, $\hat{z}$ satisfy \eqref{MCFgraphNeumann} then at the boundary $D_\mu \|y\|^2=0$.

In both cases we may apply the maximum principle to yield
\[\sup_{\Omega}\|y\|^2(\cdot,t) \leq e^{Ct}\sup_{\Omega}\|y\|^2(\cdot,0)\]
which implies the result.
 \end{proof}

\section{Evolution of symmetric 2-tensors}\label{app:evol}

We derive a general evolution equation for a symmetric 2-tensor along a spacelike mean curvature flow $M_t$.
\begin{lemma}\label{Generaltensorevol}
 Suppose that $T$ is a smooth symmetric 2-tensor on $\bb{R}^{n,m}$. We write the restriction of this tensor to $TM_t$ as
 \[T_{ij} = T(X_i, X_j).\]
 Then
 \[\n_kT_{ij} = \ov\n_kT_{ij} +T(\II_{ki},X_j)+T(X_i,\II_{kj}).\]
 and
 \begin{flalign*}
 \ho T_{ij}& = H^Ah_{Ai}^kT_{kj}+H^Ah_{Ai}^kT_{kj}-g^{kl}\ov\n_k\ov\n_l T_{ij}-2\ov\n_kT(\II_i^k,X_j)\\
 &\qquad -2\ov\n_kT(X_i,\II_j^k) -h_{ki}^Bh_{Bl}^pT_{pj}-h_{kj}^Bh_{Bl}^pT_{pi}-2h^A_{ik}h^{Bk}_jT_{AB}.
\end{flalign*}
 \end{lemma}
 \begin{proof}
  We have that
  \begin{flalign*}
  \ddt{}T_{ij} &=\ov \n_H T_{ij} +T(\ov \n_i H, X_j)+T(X_i,\ov \n_j H)\\
   &=\ov \n_H T_{ij} +T(\np_i H, X_j)+T(X_i,\np_j H)+H^Ah_{Ai}^kT_{kj}+H^Ah_{Aj}^kT_{ik},
  \end{flalign*}
  where we used that
  \[\ov \n_i H = \np_i H + H^Ah_{Ai}^kX_k .\]
We calculate
\[\n_kT_{ij} = \ov\n_kT_{ij} +T(\II_{ki},X_j)+T(X_i,\II_{kj}) .\]
We also see
\begin{flalign*}
 \n_l\n_kT_{ij} &= \ov\n_l\ov\n_kT_{ij} + \ov \n_{\II_{kl}}T_{ij} + \ov \n_{k}T(\II_{li}, X_j)+ \ov \n_{k}T(X_i,\II_{lj})\\
 &\qquad+ \ov\n_lT(\II_{ki}, X_j)+T(\np_l \II_{ki},X_j)+h_{ki}^Bh_{Bl}^pT_{pj}+T(\II_{ki}, \II_{lj})\\
 &\qquad+ \ov\n_lT(X_i,\II_{kj})+T(X_i,\np_l \II_{kj})+h_{kj}^Bh_{Bl}^pT_{ip}+T(\II_{li}, \II_{kj}) .
\end{flalign*}
Putting this together and using Codazzi--Mainardi gives the evolution of $T_{ij}$.
\end{proof}

\section{An extension of boundary quantities}\label{app:extension}
We demonstrate the following.
\begin{lemma}\label{extensionexists}
For $\Sigma$ as in Section \ref{s:Neumann}, there exists a smooth extension of $\mu$ and $\AS$,  the outward normal and the second fundamental form of $\Sigma$ respectively, such that on $\Sigma$,
\[\ov\n_\mu \mu =0, \qquad \ov\n_\mu\AS(X,Y)=0.\]
\end{lemma}
\begin{proof}
We consider the function $d:\bb{R}^n\ra \bb{R}$, defined to be the signed minimum distance to $\partial \Omega\subset \bb{R}^n$ where $d$ is positive in $\Omega$. Since $\partial \Omega$ is smooth, there exists a $c_d, C_d$ such that on the set $U:=\{x\in \bb{R}^n|d(x)<c_d\}$, $d$ is smooth and $|D d|+ |D^2 d|+|D^3d|<C_d$. On $U$, we define
\[\widehat{\mu} = -Dd, \qquad \widetilde{A}(X,Y) =\ip{D_X \widehat{\mu}}{Y} \ ,\]
and outside $U$ we take both to be zero. Standard properties of the distance function imply that on $U$, $D_{\widehat{\mu}} \widehat{\mu}=0$, and clearly on $\partial \Omega$, $\widetilde{A}$ is the second fundamental form of $\partial \Omega$. Since $\bb{R}^n$ is flat we have that
\[D_{\widehat{\mu}}\widetilde{A}(X,Y) = \ip{D_{\widehat{\mu}}(D_X {\widehat{\mu}}) - D_{D_{\widehat{\mu}} X} {\widehat{\mu}}}{Y}=\ip{D_X(D_{\widehat{\mu}}{\widehat{\mu}}) - D_{D_X{\widehat{\mu}} } {\widehat{\mu}}}{Y}=-\widetilde{A}^2(X,Y).\]
We therefore set 
\[\widehat{A}(X,Y) = \widetilde{A}(X,Y)-d\widetilde{A}^2(X,Y)\]
and we see that on $\partial \Omega$, $D_{\widehat{\mu}}\widehat{A}(X,Y)=0$. We now choose a smooth cutoff function $\chi:\bb{R}\ra\bb{R}$ such that $\chi(x)=1$ for $|x|<\frac{c_d}{3}$, $\chi(x)=0$ for $|x|>\frac{2c_d}{3}$. For $(x,y)\in \bb{R}^n\times\bb{R}^m$, define $\mu(x,y)=\chi(d(x))\widehat \mu(x)$ and define $\AS$ to be the pullback of $\chi(d(x))\widehat{A}$ by the standard projection. All claimed properties now hold.
\end{proof}

\end{appendix}

\bibliographystyle{plain}
\bibliography{bibspacelike}

\end{document}